\documentclass[11pt,reqno]{amsart}
\usepackage{amsfonts, amssymb, latexsym, epsfig,gensymb} 

\usepackage{amscd,amssymb, mathtools, etoolbox}
\usepackage{verbatim}
\usepackage[mathscr]{euscript}
\usepackage{amsmath}
\input xy
\xyoption{all}

\usepackage{tabu}
\usepackage{tensor}

\usepackage{pifont}

\usepackage[toc,page]{appendix}

\usepackage{titlesec}
\titleformat{\section}{\scshape \bfseries}{\thesection}{20pt}{\LARGE}
\titleformat{\subsection}{\scshape\bfseries}{\thesubsection}{20pt}{\Large}
\titleformat{\subsubsection}{\scshape \bfseries}{\thesubsubsection}{20pt}{\large}
           
\usepackage{afterpage}

\usepackage{parskip}

\usepackage{subcaption}

\makeatletter
\newcommand{\mathleft}{\@fleqntrue\@mathmargin\parindent}
\newcommand{\mathcenter}{\@fleqnfalse}
\makeatother

\newsymbol\pp 1275

\usepackage{tikz, standalone}
\usetikzlibrary{shapes,arrows}

\usetikzlibrary{matrix,arrows,backgrounds,shapes.misc,shapes.geometric,patterns,calc,positioning}
\usetikzlibrary{intersections,patterns,arrows,chains,matrix,positioning,scopes,decorations.pathreplacing,
calc}

\usepackage{hyperref}
\hypersetup{colorlinks,
            pdfstartview=FitV,
            urlcolor=[rgb]{1,1,1},
            linkcolor=[rgb]{.4,.4,1},
            citecolor=[rgb]{.1,.1,.8}}

\usepackage{tikz, standalone, float}
\usepackage{fullpage,subcaption}
\usepackage{graphicx}
\usepackage{enumerate,cleveref}
\def\VR{\kern-\arraycolsep\strut\vrule &\kern-\arraycolsep}
\def\vr{\kern-\arraycolsep & \kern-\arraycolsep}

\usepackage{tikz-cd}
\usepackage{extarrows}

\pgfdeclarelayer{edgelayer}
\pgfdeclarelayer{nodelayer}
\pgfdeclarelayer{background}
\pgfdeclarelayer{foreground}
\pgfsetlayers{background,foreground,edgelayer,nodelayer,main}

\usepackage{blkarray}

\newtheorem{theorem}{Theorem}
\newtheorem*{theoremnonum}{Theorem}

\newtheorem{prop}[theorem]{Proposition}
\newtheorem{lemma}[theorem]{Lemma}
\newtheorem{corollary}[theorem]{Corollary}

\newtheorem{definition}[theorem]{Definition}
\newtheorem{remark}[theorem]{Remark}
\newtheorem{notation}[theorem]{Notation}

\theoremstyle{definition}

\newtheorem{obs}{Observation}

\newtheorem{ex}[theorem]{Example}
\newenvironment{example}[1][]{\begin{ex}[#1]\pushQED{\qed}}{\popQED \end{ex}}

\newtheorem{alg}[theorem]{Algorithm}

\numberwithin{theorem}{subsection}

\newcommand{\zz}{{\mathbb{Z}\mathbb{Z}}}
\newcommand{\an}{{\mathbb{A}}}
\newcommand{\dar}{{\mathrm{AR}}}
\newcommand{\dbl}{{\mathrm{BL}}}

\newcommand{\dwil}{{\mathrm{I}}}
\newcommand{\zzu}{{\mathbb{Z}\mathbb{Z}[\leq\hspace{-.1cm}u]}}
\newcommand{\rop}{{\mathbb{R}^{\mathrm{op}}\times\mathbb{R}}}
\newcommand{\uu}{{\mathbb{U}}}

\numberwithin{equation}{section}

\newcommand{\fakesection}[1]{
  \par\refstepcounter{section}
  \sectionmark{#1}
  \addcontentsline{toc}{section}{\protect\numberline{\thesection}#1}
  }

\newcommand{\lr}[1]{\langle{#1}\rangle}

\usepackage{wrapfig} 

\title{Persistence and Stability of the \(\mathbb{A}_n\) Quiver}
\author{Killian Meehan}
\author{David C. Meyer}

\begin{document}






\maketitle



\begin{abstract}
  We introduce two new distances for zigzag persistence modules. The
  first uses Auslander-Reiten quiver theory, and the second is an
  extension of the classical interleaving distance. Both are defined
  over completely general orientations of the \(\mathbb{A}_n\) quiver.
  We compare the
  first distance to the block distance introduced by M. Botnan and
  M. Lesnick and obtain the full set of sharp Lipschitz bounds between
  the two (as bottleneck distances) over pure zigzag orientations. The
  final portion of the paper presents sharp Lipschitz bounds necessary
  for the extended interleaving distance to dominate the distance that
  is created from the Auslander-Reiten quiver. These bounds are
  obtained for general orientations of the \(\mathbb{A}_n\) quiver.
\end{abstract}


\fakesection{Beginnings}

\subsubsection{Introduction}

Both classical 1-D persistent homology and zigzag persistent homology
use data structures that fall under the same quiver theoretic notion:
they are both orientations of \emph{Dynkin quivers of type A}, which
are written throughout as \(\mathbb{A}_n\) where \(n\) is the number
of vertices of the quiver.

Quiver theory treats all orientations of \(\mathbb{A}_n\) equally
regarding the result that any representation of such a quiver (i.e.,
any persistence module over the underlying poset)
decomposes into \emph{interval} representations \cite{gabriel}, the
collection of which in turn form a \emph{barcode}\textemdash a stable
topological invariant of the representation/persistence module (or the
data set that generated it).

In this paper we propose two new distances on persistence modules over
\(\mathbb{A}_n\)-type quivers. We will spend the
rest of the paper constructing them, laying out their properties and
advantages, and proving stability results between these distances and
some of those already in use in persistent homology literature.

We primarily focus our attention on the comparison of
distances via their induced \emph{bottleneck distances}
(Definition \ref{def_bn}): distances that first associate a pair of
modules to their barcodes (collections of interval summands), and then
pair up the elements of the barcodes in some ``closest'' manner.

Here we briefly introduce and summarize these two new distances on
zigzag persistence modules and relay some of their most overt
properties.

\begin{itemize}

\item \textbf{\(\mathbb{A}_n\)-modules as multisets of vertices
    of the Auslander-Reiten quiver.}

  The \emph{AR distance} (section \ref{sec_ar_full}) can be
  applied to persistence modules over any orientation of
  \(\mathbb{A}_n\) and is a \emph{bottleneck distance} by
  construction. When some notion of 'endpoint parity' between a pair
  of interval modules agrees, their distance is simply sum of
  difference between endpoints (an \(\ell^1\)-type distance when
  considering intervals to be coordinate pairs, as is commonly seen in
  persistence diagrams). The distance behaves
  differently when parity does not agree. Over pure zigzag
  orientations, this distance's change in behavior relative to endpoint
  parity is a feature shared by the block distance
  \cite{botnan_lesnick}, which is reviewed in subsection
  \ref{sec_block} and compared in full with the AR distance in section
  \ref{sec_stab}.

  The properties of the AR distance are strongly influenced by the
  algebra of the underlying quiver. For
  instance, in pure zigzag orientations, interval modules of
  \([\mathrm{sink},\mathrm{sink}]\) endpoint parity are
  close to projective simple modules, and those
  with \([\mathrm{source},\mathrm{source}]\) endpoint party are close
  to injective simples modules (in this situation ``closeness''
  is relative to support size). In general, when a
  pair of intervals has non-matching endpoint parity, the poset
  structure influences their distance to a much greater
  degree than similarity in supports.
  
\item \textbf{\(\mathbb{A}_n\)-modules as persistence modules over a
    suspended poset.}

  The \emph{weighted interleaving distance} (section \ref{sec_dwil})
  considers an arbitrary orientation of \(\mathbb{A}_n\) as a series
  of connected `valleys' (maximal upward posets of the form
  \([\textrm{source},\infty)\)), and then measures
  the distance between two modules by the depth of
  the valleys on which the intervals must be isomorphic. On all
  shallower valleys they are free to differ.

  The general construction was pursued in our previous paper
  \cite{meehan_meyer_1} for the purpose of applying interleaving
  distance to finite posets without inevitably encountering an
  excessive number of module pairs whose interleaving distance was
  infinite.
\end{itemize}

\subsubsection{Contributions}

A summary of our contributions are as follows:
\begin{itemize}
\item We provide full and sharp Lipschitz bounds
  between our AR distance and the block distance,
  with the latter treated as its own induced bottleneck distance
  (Theorem \ref{thm_blar_limit}).

\item Included as part of the elucidation of the AR
  distance is as a topic of potentially independent interest: we
  provide an explicit
  formulation of the Auslander-Reiten quiver for any orientation of
  \(\mathbb{A}_n\) in Section \ref{sec_ar}. While this
  formulation follows from the
  Knitting Algorithm (see \cite{schiffler} for details on the
  Knitting Algorithm and other methods of calculating the
  Auslander-Reiten quiver for orientations of \(\mathbb{A}_n\)), our
  formulation provides full information about the
  Auslander-Reiten quiver without any iterative
  construction.

\item We provide sharp bounds for the weighted interleaving distance
  to dominate the AR distance. (Theorem
  \ref{thm_pairs}.)
\end{itemize}

\subsubsection{Acknowledgements}

The authors would like to thank
Vin de Silva
and
Michio Yoshiwaki
for their discussions and insights, as well as the entirety of the
members of the Hiraoka Laboratory for their support and assistance.

The first named author is also supported in part by
JST CREST Mathematics (15656429).

\subsection{Preliminaries}

\begin{notation}
  Throughout, we say that a \emph{distance} on a set \(X\) is a function
  \(d:X\times X\to[0,\infty]\) such that
  \begin{enumerate}
  \item \(d(x,x)=0\) for all \(x\in X\),
  \item \(d(x,y)=d(y,x)\) for all \(x,y\in X\), and
  \item \(d(x,y)\leq d(x,z)+d(z,y)\) for any \(x,y,z\in X\).
  \end{enumerate}
\end{notation}
That is, from the standard definition of 'metric' we surrender
identification between points with \(d(x,y)=0\) and allow distances to
take on infinite values.

\begin{definition}\label{defn_gpm}
A \emph{generalized persistence module} (GPM) \(F\) over a poset \(P\) in a
category \(\mathcal{D}\) is a functor \(F:P\to \mathcal{D}\). That is,
\(F\) is an assignment
\begin{itemize}
\item \(x\to F(x)\) for all \(x\in P\),
\item \((x\leq y)\to F(x\leq
  y)\in\mathrm{Hom}_\mathcal{D}(F(x),F(y))\) for all \(x\leq y\) in
  \(P\)
\end{itemize}
such that, for any \(x\leq y\leq z\), the inequalities are sent to
morphisms
satisfying \(F(y\leq z)\circ F(x\leq y)=F(x\leq z)\).

The category of such functors is denoted \(\mathcal{D}^P\), where
morphisms in this category are given by natural
transformations of functors.

A \emph{persistence module} is a GPM with values in the category of
finite dimensional vector spaces, and is the object of primary
interest in this document.
\end{definition}

\subsubsection{Quivers} In this paper we will frequently view our
underlying structures as both posets and as quivers. We would like to
work with familiar persistent homology structures while applying
quiver-theoretic machinery. The following
is a short, formal definition of quivers, as well as an explanation for
why they can be thought of as equivalent to posets in our setting.

\begin{definition}\label{defn_quiv}
A \emph{quiver} is a quadruple \((Q_o,Q_1,h,t)\) where
\begin{itemize}
\item \(Q_0\) is some finite set called the \emph{vertex set},
\item \(Q_1\) is a collection of \emph{arrows} between vertices,
\item \(h:Q_1\to Q_0\) is a map that sends each arrow to its
  destination (\emph{head}), and
\item \(t:Q_1\to Q_0\) is a map that sends each arrow to its source
  (\emph{tail}).
\end{itemize}

A \emph{representation} \(V\) of a quiver \(Q\) is
\begin{itemize}
\item a vector space \(V(i)\) assigned to every vertex, and 
\item a linear map \(V(a):V(ta)\to V(ha)\) assigned to every arrow.
\end{itemize}
\end{definition}


The space of finite-dimensional representations of a quiver \(Q\),
denoted \(\mathrm{rep}(Q)\), is a category with morphisms given
pointwise, \(f=\{f_i\}_{i\in Q_0}:V\to W\), such that they satisfy
commutative squares \(f(ha)V(a)=W(a)f(ta)\) for all arrows \(a\in
Q_1\).

Quivers may, in general, have closed loops or multiple arrows between
the same pair of vertices. These
features may prevent such quivers from being posets under the
relation  \[
  \{x\leq y\text{ if and only if there exists a path }
  x\to y\}.
\] However, the converse\textemdash that posets always give rise to
quivers in a canonical way\textemdash is true.

\begin{definition}
  For a poset \(P\), the \emph{Hasse} quiver \(Q(P)\) is
  the quiver given by:
  \begin{itemize}
  \item \(Q_0=P\) as a set of vertices.
  \item There exists an arrow \(i\to j\) whenever \(i\leq j\) in
    \(P\), and there is no \(k\) (distinct from \(i\) and \(j\)) such
    that \(i\leq k\leq j\).
  \end{itemize}
\end{definition}

Under certain restrictions, quivers do give rise to posets in a
fashion that inverts the Hasse construction. When there is such a
bi-directional correspondence, as seen in the following proposition,
\emph{the space of representations of the
quiver is equivalent to the space of persistence modules over
the poset.}

\begin{prop}
  \label{prop_p_and_q}
  Let \(Q\) be a quiver such that:
  \begin{itemize}
  \item \(Q\) has no cycles (including stationary loops),
  \item for any two \(i,j\in Q_0\), there exists at most one arrow
    between \(i\) and \(j\).
  \end{itemize}
  Then \(Q\) is the Hasse quiver of some poset \(P\). Furthermore,
  suppose \(Q\) also satisfies:
  \begin{itemize}
  \item For \(i,j\in Q_0\), there is at most one path from
    \(i\to\ldots\to j\).
  \end{itemize}
  Then, the category of
  finite-dimensional representions of the quiver \(Q\) is equivalent
  to the category of functors from the poset category \(P\) to the
  category of finite-dimensional vector
  spaces. I.e., \[\mathrm{rep}(Q)\cong\mathrm{vect}^P\] as
  cagegories.
\end{prop}

\begin{remark} The extra condition above (at most one path
  \(i\to\ldots\to j\)) is necessary for the equivalence of categories
  for the following reason.
  Any GPM over a poset, by virtue of being a
  functor from a \emph{thin} category (cardinality of any
  \(\mathrm{Hom}\)-space is at most \(1\)), has the
  property that the morphisms given by composition along any two
  parallel paths are equal; see Definition
  \ref{defn_gpm}. Contrast this with Definition \ref{defn_quiv} in
  which there are no parallel-path commutativity conditions on a
  quiver representation.

  (If one wished to obtain equivalence between the two
  categories while allowing for the existence of parallel paths,
  this would require the use of \emph{bound quivers}:
  quivers with commutativity relations, for a general reference see
  \cite{schiffler}. Such pursuits are not within the scope of this
  document.)
\end{remark}

By virtue of this equivalence, from here onward we will denote
quivers/posets by \(P\), rather than \(Q\).

The quiver of interest in this paper is \(P=\mathbb{A}_n\), the
`straight line' quiver, with arbitrary orientations for its
arrows. It satisfies all the conditions
of Proposition \ref{prop_p_and_q}.

\begin{definition}
\label{def_an} For \(n\in\mathbb{N}\),
an \(\mathbb{A}_n\)-type quiver is any quiver with vertex set
\(\{1,\ldots,n\}\) whose arrow set consists of exactly one of \[
i\to i+1\text\,\,{ or }\,\,i\gets i+1
\] for every \(i\).

The corresponding poset (whose Hasse quiver returns the original
quiver) is given by \[
  1\sim \ldots \sim n
\] where each \(\sim\) corresponds to \(<\) (for quiver arrows of the
form \(\rightarrow\)) or \(>\) (for quiver arrows of the form
\(\leftarrow\)).

\(\mathbb{A}_n\) will be said to be \emph{equioriented} if all arrows face
the same way.
\begin{center}
  \includegraphics[scale=1.5]{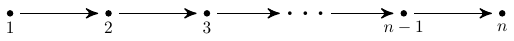}
\end{center}
\(\mathbb{A}_n\) will be said to have \emph{pure zigzag orientation}
if arrows alternate (i.e., each vertex is either a source or sink).
\begin{center}
  \includegraphics[scale=1.5]{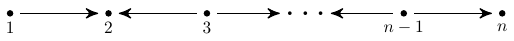}
\end{center}
\end{definition}

The following definition is a fundamental one to persistent homology.

\begin{definition}
  For an orientation of \(\mathbb{A}_n\), define the interval
  persistence module (indecomposable quiver representation) \([x,y]\) to
  be the one \[[x,y](i)=\left\{
      \begin{array}{ll}
        K & \text{if }1\leq x\leq i\leq y\leq n \\
        \\
        0 & \text{otherwise}
      \end{array}\right.
  \] where \(K\) is some base field. The internal morphisms of
  \([x,y]\) are \(1_K\) when possible, and \(0\) otherwise.
  
  From context it should always be
  clear when we mean the indecomposable \([x,y]\) or the
  \(\mathbb{Z}\)-interval \([x,y]\).

  Lastly, we will often abbreviate interval persistence modules of the
  form \([x,x]\) as \([x]\).
\end{definition}

For \(P=\mathbb{A}_n\), as it turns out, every
\(\mathbb{A}_n\)-representation\,\textemdash\,equivalently, every\(P\)
persistence module\,\textemdash\,is isomorphic to a direct
sum of interval persistence modules. The original result cited below
is quiver theoretic in origin, but this result has since been proved
independently for pointwise finite dimensional persistence modules
over \(\mathbb{R}\) \cite{cb}.

\begin{prop}[\cite{gabriel}]\label{prop_gabriel}
  Representations / persistence modules over any \(P=\mathbb{A}_n\)
  decompose into interval persistence modules. This decomposition is
  unique up to ordering and isomorphism of summands.

  Furthermore, interval persistence modules are precisely the
  indecomposable persistence modules (up to isomorphism) of \(P\).
\end{prop}

For a very efficient exposition of the definitions and
features of additive categories, categorical products and coproducts,
and categories possessing unique decomposability properties, the
authors recommend the paper \cite{krause}.

\begin{notation}
  Throughout, by \emph{indecomposable
    representation} of \(P=\mathbb{A}_n\) we mean the unique
  representative of the isomorphism class that is precisely an
  interval representation.
\end{notation}

\subsubsection{The Auslander-Reiten Quiver}
\label{ss_ar}

The following is a crucial piece of quiver theoretic machinery that
renders possible the development of this paper's first distance.

\begin{definition}
Given a quiver \(P\), its \emph{Auslander-Reiten (AR) quiver} is a new
quiver
in which:
\begin{itemize}
\item the vertex set is the collection of isomorphism classes of
  indecomposable representations of \(P\),
\item an arrow exists from one vertex to another whenever
  there exists an irreducible morphism between the corresponding
  \(P\)-indecomposables.
\end{itemize}
When \(P=\mathbb{A}_n\), there are finitely
many indecomposable representations up to isomorphism, and
representatives of the distinct isomorphism classes can be chosen to
be precisely the collection of
\emph{interval} representations of \(P\) (Proposition
\ref{prop_gabriel}). That is, the
Auslander-Reiten quiver of some \(P=\mathbb{A}_n\) has vertex set
consisting of the interval representations of \(P\).
\end{definition}

See any of \cite{elements,schiffler,krause_notes,derksen_weyman}
for general introductions to Auslander-Reiten theory.

What is important to note for now is that, when \(P=\mathbb{A}_n\), its
Auslander-Reiten quiver has a finite vertex set, unique arrows, no
closed loops, and is a
connected graph (\cite{auslander} VI Thm 1.4). The nature of the
Auslander-Reiten quiver of any
\(P=\mathbb{A}_n\) will be discussed in detail in Subsection
\ref{sec_ar}.


\subsection{Classic Persistent Homology Distances}

We now define two fundamental distances to persistent homology.

\subsubsection{Interleaving Distance}

The interleaving distance is a distance on generalized persistence
modules with values in any category \(\mathcal{D}\) over any poset
\(P\) (Definition \ref{defn_gpm}). We offer the following definitions
in their full generality, though in the remainder of the paper they
will be applied only to persistence modules (GPMs with values in
\(\mathrm{vect}\)) over very specific posets.

We first define translations, which are used to `shift' GPMs within a
poset and are how the size of an interleaving is measured.

\begin{definition}\label{def_translation}
A \emph{translation} \(\Lambda\) on a poset \(P\) is a map
\(\Lambda:P\to P\) such that
\begin{itemize}
\item \(x\leq\Lambda x\) for all \(x\in P\),
\item if \(x\leq y\) in \(P\), then \(\Lambda x \leq \Lambda y\).
\end{itemize}
The \emph{height} of a translation is
\[h(\Lambda)=\max_{x\in P}\{d(x,\Lambda x)\},\] where \(d\) is some
distance on \(P\).

The collection of translations over a poset \(P\) form a monoid with
left action on any \(\mathcal{D}^P\), given by the pointwise
statement \[
  F\Lambda(x)=F(\Lambda x)\text{ for all }x\in P.
\]
\end{definition}

In brief, before the full definition below, an interleaving between
two GPMs is a translation
\(\Lambda\) and a pair of morphisms from each GPM to a
\(\Lambda\)-shift of the other such that 
certain commutativity conditions are fulfilled.

\begin{definition}\label{def_il}
An interleaving between two GPMs \(F,G\) in \(\mathcal{D}^P\)
is a translation
\(\Lambda\) on \(P\) and a pair of morphisms (natural transformations)
\(\phi:F\to G\Lambda\),
\(\psi:G\to F\Lambda\) such that the following diagram commutes:
\begin{center}
\includegraphics[scale=1]{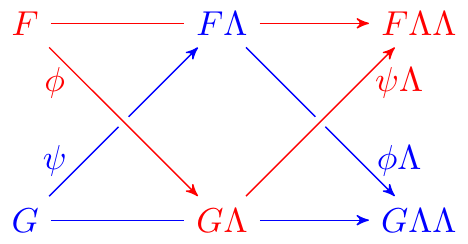}
\end{center}
Alternatively, we say that \(F,G\) are \(\Lambda\)-\emph{interleaved}.

The \emph{interleaving distance} between \(F\) and \(G\) is \[
  D_{\mathrm{IL}}(F,G)=\inf\{\epsilon:F,G\text{ have a }\Lambda\text{-interleaving
    with }h(\Lambda)=\epsilon.\} 
\]
\end{definition}

The translations \(\phi\) and \(\psi\) are sometimes referred to as
``approximate isomorphisms'', and the interleaving distance can be
thought of as the shift distance by which there fails to be a true
isomorphism between the persistence modules.


\begin{remark}
  The above definition is not quite the traditional one seen most often
  in the literature (see \cite{bubenik}). In many definitions
  there are two translations
  \(\Lambda\) and \(\Gamma\) (one to shift \(F\), and the other to
  shift \(G\)), and the height of the interleaving is the
  height of the larger translation. In the posets we are interested in,
  the values of the interleaving distance do not change when allowing
  for two distinct translations rather than using the same translation
  twice. So, for the sake of simplicity, and without altering the
  distance, we have reduced Definition \ref{def_il} to a statement
  involving only a single translation \(\Lambda\).
\end{remark}

The collection of translations on a poset \(P\) is itself a poset
under the partial order given by the relation \[
  \Lambda\leq\Gamma\text{ if
  }\Lambda(x)\leq\Gamma(x)\text{ for all }x\in P.
\] There is rarely a \emph{unique} translation of a given height,
though occassionally it is easier to assume that we are using a
\emph{full} translation of some height.

\begin{remark}\label{rmk_full}
  By a \emph{full translation} of height \(\epsilon\), we will mean a
  maximal element
  in the poset of translations that has height \(\epsilon\). In the case
  \(P=\mathbb{Z},\mathbb{R}\), there is always a unique full translation of
  height \(\epsilon\): the translation
  \(\Lambda_\epsilon(x)=x+\epsilon\) for all \(x\in\mathbb{R}\). In
  posets that are not totally ordered, there may be multiple
  distinct full translations of certain heights.
\end{remark}

By the next result, any \(\epsilon\)-interleaving can always be taken
as using a full translation of height \(\epsilon\).

\begin{prop}
  Let \(\Lambda,\Lambda'\) be two translations over some poset \(P\)
  such that \(\Lambda'\geq\Lambda\), and let \(F,G\) be two GPMs in
  \(\mathcal{D}^P\) for some \(\mathcal{D}\). If \(F,G\) are
  \(\Lambda\)-interleaved,  then \(M,N\) are
  \(\Lambda'\)-interleaved.
\end{prop}

\subsubsection{Bottleneck Distances}

We first define the general notion of a bottleneck distance
(Definition \ref{def_bn}), then present the classic
bottleneck distance (Example \ref{ex_classic_bn}), and lastly put
forward the meaning of a general distance's induced bottleneck
distance (Remark \ref{rmk_induced}).

A bottleneck distance (also a Wasserstein metric\,\textemdash\,see
\cite{wasserstein}) acts on pairs
of multisets of some set \(\Sigma\). It requires
\begin{itemize}
\item a distance \(d\) on \(\Sigma\), and
\item a function \(W:\Sigma\to[0,\infty)\)
\end{itemize}
such that \begin{equation*}\label{eqn_triangle}
  |W(f)-W(g)|\leq d(f,g),\tag{$\Delta$-ineq}
\end{equation*} for all \(f,g\in\Sigma\).

Let \(\Sigma\) be some set, and \(F,G\) two multisets (subsets with
multiplicities of elements) of
\(\Sigma\). A \emph{matching} between \(F\) and \(G\)
is a bijection \[x:F'\leftrightarrow G'\] where
\(F'\subset F\), \(G'\subset G\).

The \emph{height} of a matching \(x:F\leftrightarrow G\) is \[
  h(x)=\max\{\max_{f\in F'}\{d(f,x(f))\},\max_{f\not\in
    F'}\{W(f)\},\max_{g\not\in G'}\{W(g)\}\}.
\] That is, take the maximum over all distances (using \(d\)) between
paired elements, as well as the maxima over all of the `widths' (using
\(W\)) of the unpaired elements of \(F\) and \(G\).

\begin{definition}\label{def_bn}
Given a set \(\Sigma\), and any functions \(d\) and \(W\) as above
holding to the \ref{eqn_triangle} relationship, the \emph{bottleneck
  distance generated by} \(d\) \emph{and} \(W\) between two
multisets \(F,G\) of \(\Sigma\) is \[
  D(F,G)=\min\{h(x):x\text{ is a matching between }F\text{ and }G\}.
\]
\end{definition}

The following connects bottleneck distances to persistence
modules. From \cite{cb}, this can be generalized to \(\mathbb{R}\)
persistence modules.

\begin{definition}\label{def_barcode}
  For \(\mathbb{A}_n\), let \(\Sigma\) denote the set of (isomorphism
  classes of) indecomposable persistence modules: i.e., its
  \emph{intervals}.
  
  For a persistence module \(M\) over \(\mathbb{A}_n\), define its
  \emph{barcode} to be the multiset of \(\Sigma\) containing exactly
  the summands in its decomposition (with existence and uniqueness
  guaranteed by Proposition \ref{prop_gabriel}):
  \[
    \mathcal{B}(M)=\{[x_i,y_i]\}_{i\in I}\text{, where }M=\bigoplus_{i\in
      I}[x_i,y_i].
  \]
\end{definition}

%
%

\begin{example}\label{ex_classic_bn}
  The `classical' bottleneck distance on persistence modules
  over \(\mathbb{R}\)) is the one given by
  \begin{itemize}
  \item \(d(f,g)=D_{\mathrm{IL}}(\{f\},\{g\})\)
  \item \(W(f)=D_{\mathrm{IL}}(\{f\},\emptyset)\),
  \end{itemize} where \(D_{\mathrm{IL}}\) is the interleaving distance of
  Definition \ref{def_il}.
  
  Let \(f=[x_1,y_1],g=[x_2,y_2]\) be indecomposable/interval
  persistence modules over \(\mathbb{R}\). Unpacking the definition of
  interleaving distance yields the equations:
  \begin{itemize}
  \item \(d(f,g)=\max\{|x_1-x_2|,|y_1-y_2|\}\) and
  \item \(W(f)=I(f,0)=1/2(y-x)\).
  \end{itemize}
  This is precisely the (\(\ell^\infty\) or \(\infty\)-Wasserstein)
  bottleneck distance that is most commonly used to measure distance between
  persistence diagrams in persistent homology literature.
\end{example}

\begin{remark}\label{rmk_induced}
  Let \(\mathcal{C}\) be any Krull-Schmidt category \cite{krause} and
  \(D\) any distance on the collection of objects in the category. Then
  there is a unique or canonical bottlneck distance induced by \(D\),
  that being the one in which any two
  objects \(X,Y\) become associated to the multisets corresponding to
  their Krull-Schmidt decompositions \[
    X=X_1\oplus\ldots\oplus X_m,\,\,Y=Y_1\oplus\ldots\oplus Y_n,
  \] and the bottleneck distance between those multisets is given by
  \begin{itemize}
  \item \(d(X_i,Y_j)=D(X_i,Y_j)\) and
  \item \(W(X_i)=D(X_i,0)\).
  \end{itemize}
\end{remark}

\subsubsection{Comparison of Bottleneck Distances}

As one of the goals of this paper is finding minimal Lipschitz bounds
between bottleneck distances, we discuss the relationship between
comparing bottleneck distances directly, and comparing their component
\(d\)'s and \(W\)'s.

For two bottleneck distances
\(D_1=\{d_1,W_1\}\) and \(D_2=\{d_2,W_2\}\), while the inequality
\(D_1\leq D_2\) implies \(W_1\leq W_2\), it does \emph{not}
necessitate that \(d_1\leq d_2\). This has the potential to be a
frustrating obstacle to comparing different bottleneck distances.

To remedy this, we define a canonical \(d\) and \(W\) for
a given bottleneck distance \(D\) that will allow for a more natural
means of comparison.

\begin{definition}[Minimal generators]\label{def_bn_canon}
  For a bottleneck distance \(D=\{d,W\}\) on multisets of some set
  \(\Sigma\), define
  \(\bar{d}(\sigma,\tau)=D(\{\sigma\},\{\tau\})\). Then \(\bar{d}\leq
  d\), and the pairs \(\{d,W\}\) and \(\{\bar{d},W\}\) both generate
  the same \(D\). Call \(\{\bar{d},W\}\)
  the \emph{minimal generators} of the bottleneck distance \(D\).
\end{definition}

A bottleneck distance \(D\) fully
recovers its minimal generators. Specifically:
\begin{itemize}
\item \(\bar{d}(\sigma,\tau)=D(\{\sigma\},\{\tau\})\), and
\item \(W(\sigma)=D(\{\sigma\},\emptyset)\).
\end{itemize}

We now get the desired comparison statement:

\begin{prop}\label{prop_bn_compare}
  For two bottleneck distances \(D_1=\{d_1,W_1\}\) and
  \(D_2=\{d_2,W_2\}\), \(D_1\leq D_2\) if and only if
  \(\bar{d}_1\leq\bar{d}_2\) and \(W_1\leq W_2\).
\end{prop}

\begin{proof}
  The forward implication is immediate from the above statements about
  the recovery of \(\bar{d},W\) from \(D\). The reverse implication is
  immediate from the definition of \(D\) (as is the stronger
  statement: \(d_1\leq d_2\) and \(W_1\leq W_2\)
  \(\implies\) \(D_1\leq D_2\)).
\end{proof}

\begin{notation}
  From this point onward, we allow \(D(\{\sigma\},\{\tau\})\) to be
  shortened to \(D(\sigma,\tau)\) for bottleneck distances.
\end{notation}

\section{AR-Bottleneck Distance}\label{sec_ar_full}

This bottleneck distance uses the graph-structure of some original
quiver \(Q\)'s corresponding Auslander-Reiten quiver as a means
of measuring the distance between indecomposable persistence modules.

\subsection{Definitions}

Let $Q=\mathbb{A}_n$. Let $Q'$ be the AR quiver of $Q$. For
indecomposables \(\sigma,\tau\) of \(Q\),
let \(p=p_0\ldots p_l\) denote an unoriented path in \(Q'\)
from \(\sigma\) to \(\tau\). The tail and head of a path are those of
the first and last vertex, respectively: \(tp=tp_l=\sigma\), and
\(hp=hp_0=\tau\).

\begin{definition}\label{def_ar_dist}
Define the \emph{AR distance} between two indecomposables to
be \[
  \delta_\dar(\sigma,\tau)=
  \min_{p:\sigma\to\tau}\left\{\sum_{i=1}^{l-1}|\mathrm{dim}(Q'(hp_i))
    -\mathrm{dim}(Q'(tp_i))|\right\},
\]
where \emph{dimension} of an indecomposable \(M\) of \(Q\)
(equivalently, a vertex of \(Q'\))
is \[
  \mathrm{dim}(M)=\displaystyle\sum_{i\in
    Q_0}\mathrm{dim}_KM(i),
  \text{ i.e., 
  }\mathrm{dim}([x,y])=y-x+1.
\]
That is, \(\delta_\dar(\sigma,\tau)\) is the dimension-weighted
path-length between \(\sigma\) and \(\tau\), minimized over all
possible paths.
\end{definition}

\begin{ex}
  See Figure \ref{fig_ex_ar}. Consider the interval modules
  \([2,3],[3,6]\).
  \begin{itemize}
  \item Figure \ref{fig_ex1_ar}: \(\delta_\dar([2,3],[3,6])=4\).
  \item Figure \ref{fig_ex2_ar}: \(\delta_\dar([2,3],[3,6])=8\).
  \item Figure \ref{fig_ex3_ar}: \(\delta_\dar([2,3],[3,6])=10\).
  \end{itemize}
\end{ex}

\begin{figure}
  \centering
  \begin{subfigure}[t]{1\textwidth}
    \centering
    \includegraphics[scale=.6]{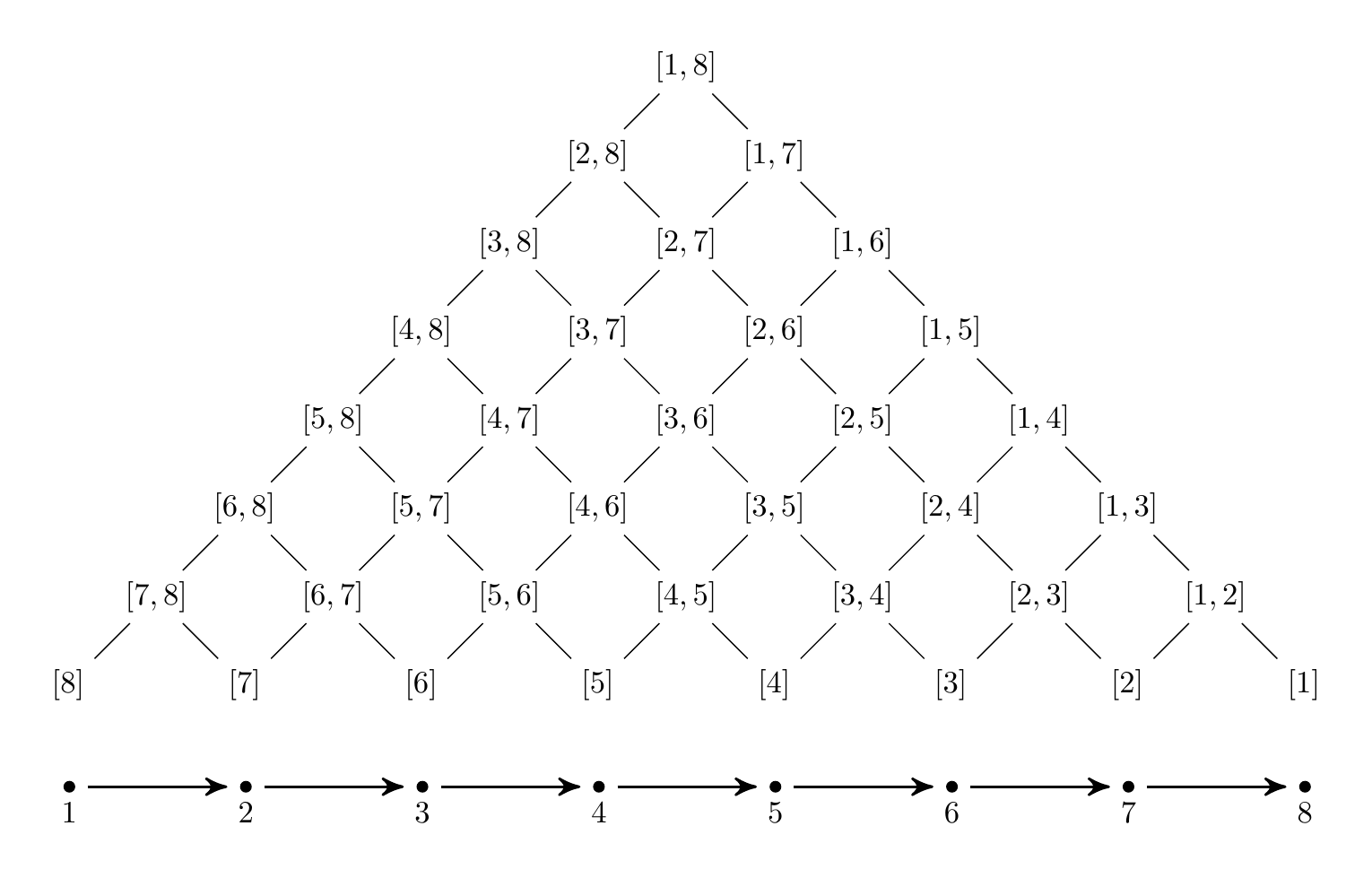}
    \caption{Equi-orientation of \(\mathbb{A}_8\) and its
      corresponding AR quiver.}
    \label{fig_ex1_ar}
  \end{subfigure}
  
  \begin{subfigure}[t]{1\textwidth}
    \centering
    \includegraphics[scale=.6]{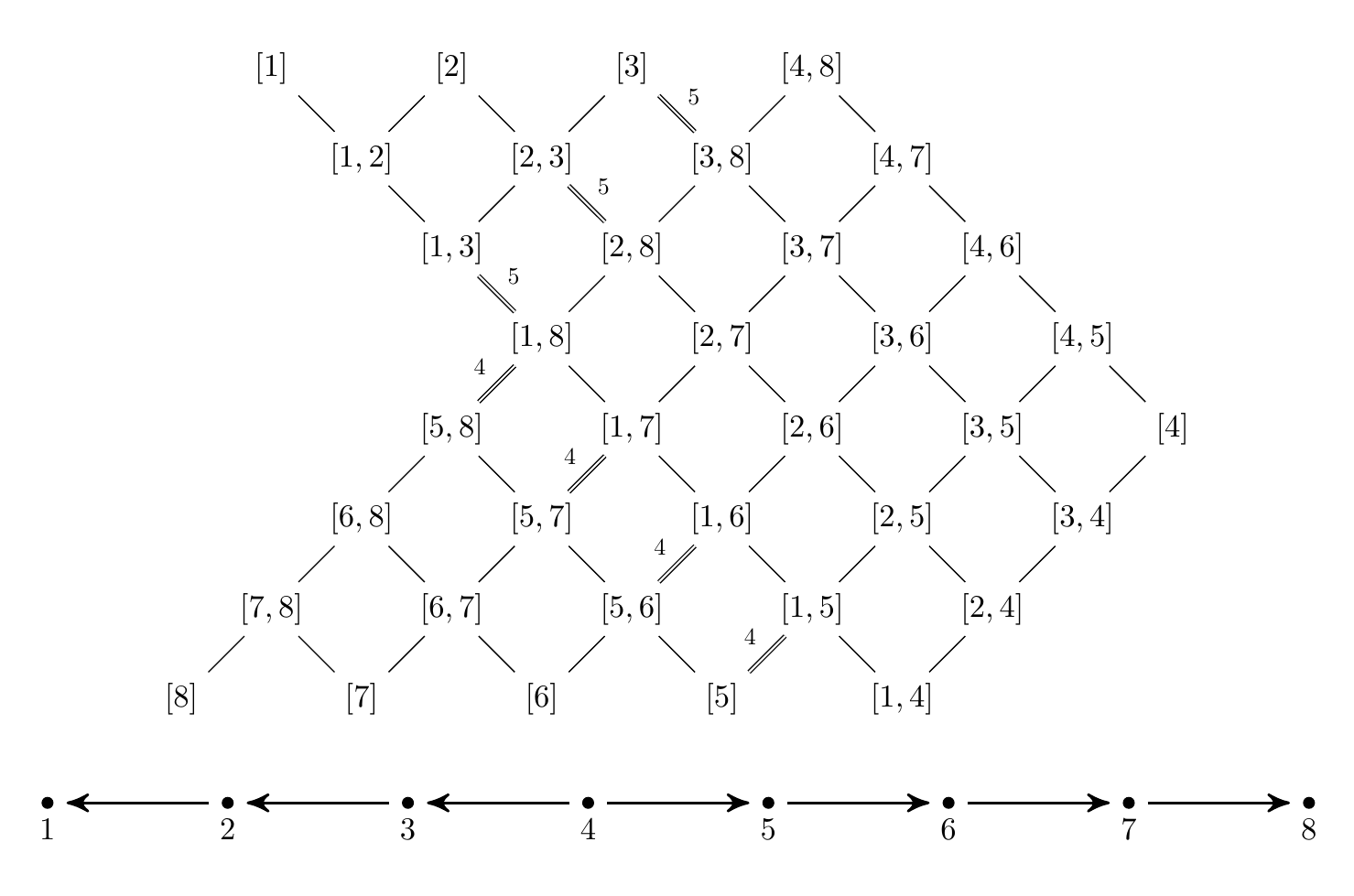}
    \caption{Orientation of \(\mathbb{A}_8\) and its corresponding
      AR quiver.}
    \label{fig_ex2_ar}
  \end{subfigure}
  
  \begin{subfigure}[t]{1\textwidth}
    \centering
    \includegraphics[scale=.6]{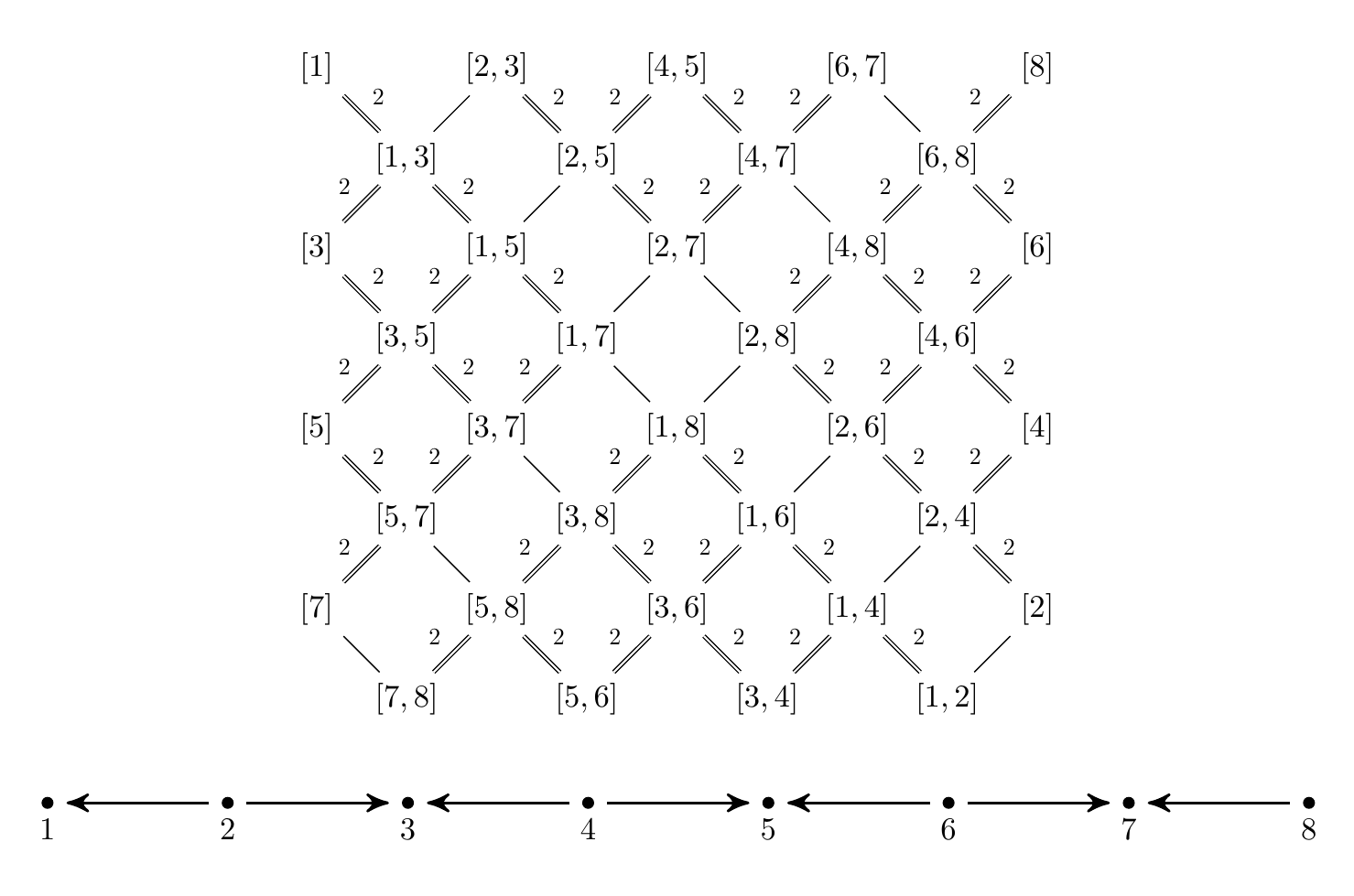}
    \caption{Zigzag orientation of \(\mathbb{A}_8\) and its
      corresponding AR quiver.}
    \label{fig_ex3_ar}
  \end{subfigure}
  \caption{Three orientations of \(\mathbb{A}_8\) and their AR
    quivers, where edges of weight more than \(1\) are drawn with
    double lines and labeled by the difference in dimensions
    between the two indecomposables that they connect.}
  \label{fig_ex_ar}
\end{figure}

\begin{definition}\label{def_dar}
Define the \emph{AR bottleneck distance} \(D_{\mathrm{AR}}\) on the space of
indecomposable representations of $Q$ to be the bottleneck distance
induced by:
\begin{itemize}
\item \(d_\dar(\sigma,\tau)=\delta_\dar(\sigma,\tau)\),
\item \(W_\dar(\sigma)=\displaystyle\min_{t\in
    Q_0}\{\delta_\dar(\sigma,[t])\}+1\).
\end{itemize}
\end{definition}

We can immediately check that \(D_\dar\) is indeed a bottleneck
distance.

\begin{prop}\label{prop_dar_triangle}
  \(D_\dar\) satisfies \ref{eqn_triangle}.
\end{prop}

\begin{proof}
  Simply note that for any \(\sigma,\tau\), and any simple \([t]\), by
  the graph-distance definition of \(\delta_\dar\) it is immediate that \[
    d_\dar(\sigma,[t])\leq d_\dar(\sigma,\tau)+d_\dar(\tau,[t]),
  \] and so, minimizing over \([t]\) with respect to
  \(W_\dar(\tau)\), \[
    W_\dar(\sigma)\leq d_\dar(\sigma,[t])+1\leq d_\dar(\sigma,\tau)+W_\dar(\tau).
  \] Combining with the symmetric statement (swapping \(\sigma\) and
  \(\tau\)) we get the full statement of the equation
  \ref{eqn_triangle}.
\end{proof}

\begin{remark}
The reason for the \(+1\) in the definition of \(W_{\mathrm{AR}}\)
above is simply that there are no zero representations in the AR
quiver. As in \cite{escolar_hiraoka}, we account for the distance to
zero being distance to a simple indecomposable, plus one additional
traversal (of dimension-weight \(1\)).

Put another way, we attach a zero representation to every simple
indecomposable in the AR quiver (see Figure \ref{fig_ex1_ar}).
For \(Q=\mathbb{A}_n\), let
\(\bar{Q}'\) denote the AR quiver of
\(Q\) supplemented with the vertices \(0_i\) for all vertices \(i\)
of \(Q\), and with extra edges \([i]\rightarrow 0_i\). Then we may
alternatively define \(W_\dar(\sigma)=\displaystyle\min_{i\in
  Q_0}\{\delta_\dar(\sigma,0_i)\}\).
\end{remark}

\begin{ex}
  See Figure \ref{fig_ex_ar}. Consider the interval modules
  \([2,3],[3,6]\).
  \begin{itemize}
  \item Figure \ref{fig_ex1_ar}: \(D_\dar([2,3],[3,6])=4\).
  \item Figure \ref{fig_ex2_ar}: \(D_\dar([2,3],[3,6])=4\).
  \item Figure \ref{fig_ex3_ar}: \(D_\dar([2,3],[3,6])=8\).
  \end{itemize}
\end{ex}

\subsection{AR Quiver Construction Algorithm}\label{sec_ar}

From here we present an algorithm for determining the shape of the
Auslander-Reiten quiver for any quiver of the form
\(Q=\mathbb{A}_n\). This algorithm arises as a consequence of the
Knitting Algorithm (see \cite{schiffler} Chapter 3.1.1), but has been
streamlined to the specific case of \(Q=\mathbb{A}_n\), and is able
to elucidate the full structure of such AR quivers without
the sequential construction method that the Knitting Algorithm and
other similar methods require.

We maintain the convention of many quiver theoretic publications, in
which the AR quiver is drawn with arrows always
directed left to right, with the leftmost indecomposables being simple
projectives and the rightmost indecomposables being simple
injectives. Vertical orientation is arbitrary, but will be fixed under
the following method. Key to this structural result about AR quivers
for arbitrary orientations of any \(\mathbb{A}_n\) is the fact that
the indecomposables fit into a \emph{diagonal grid} with axes
for the left and right endpoints of the intervals. The algorithm
instructs the formation of these axes, which subsequently induce the
entire shape of the AR quiver.

\begin{notation}
There are two separate and obvious orderings on the vertices of any
orientation of \(\mathbb{A}_n\), the first being the ordering of the
vertices according to their labeling as a subset of
\(\mathbb{Z}\), and the second being the ordering given by
the poset relation \(\leq_P\). The following discussions are carried
out in the language of the \emph{vertices as a subset of}
\(\mathbb{Z}\). So, by all
comparative words (increasing, decreasing, greater, lesser) we will
mean relative to the inherited \(\mathbb{Z}\)-ordering of the vertices
from left to right in the poset.
\end{notation}

\afterpage{
  \begin{figure}
    \centering
    \includegraphics[scale=.75]{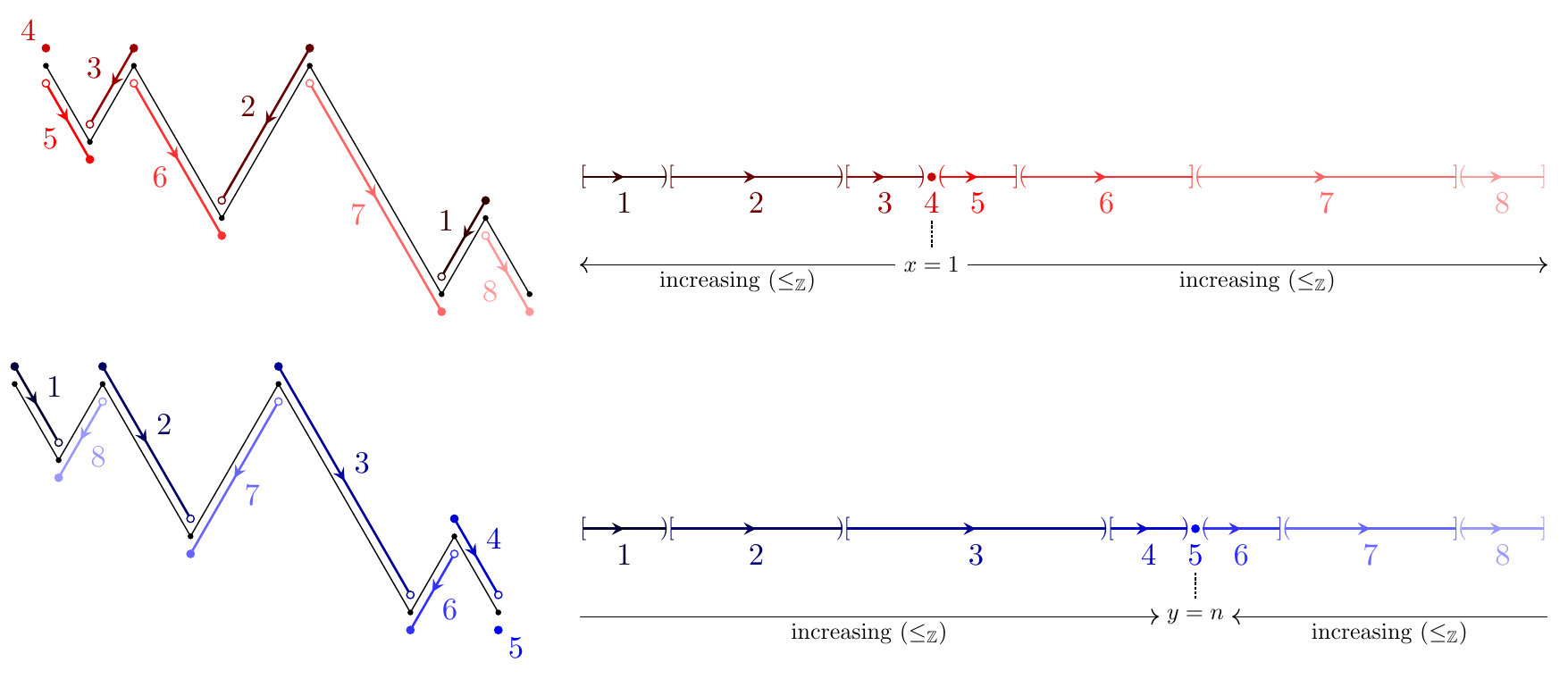}
    \caption{An orientation of \(\mathbb{A}_n\) and the subsequent
      arrangements of the \(x\) and \(y\) axes.}
    \label{fig_ar_super1}
  \end{figure}
  
  \begin{figure}
    \centering
    \includegraphics[scale=.75]{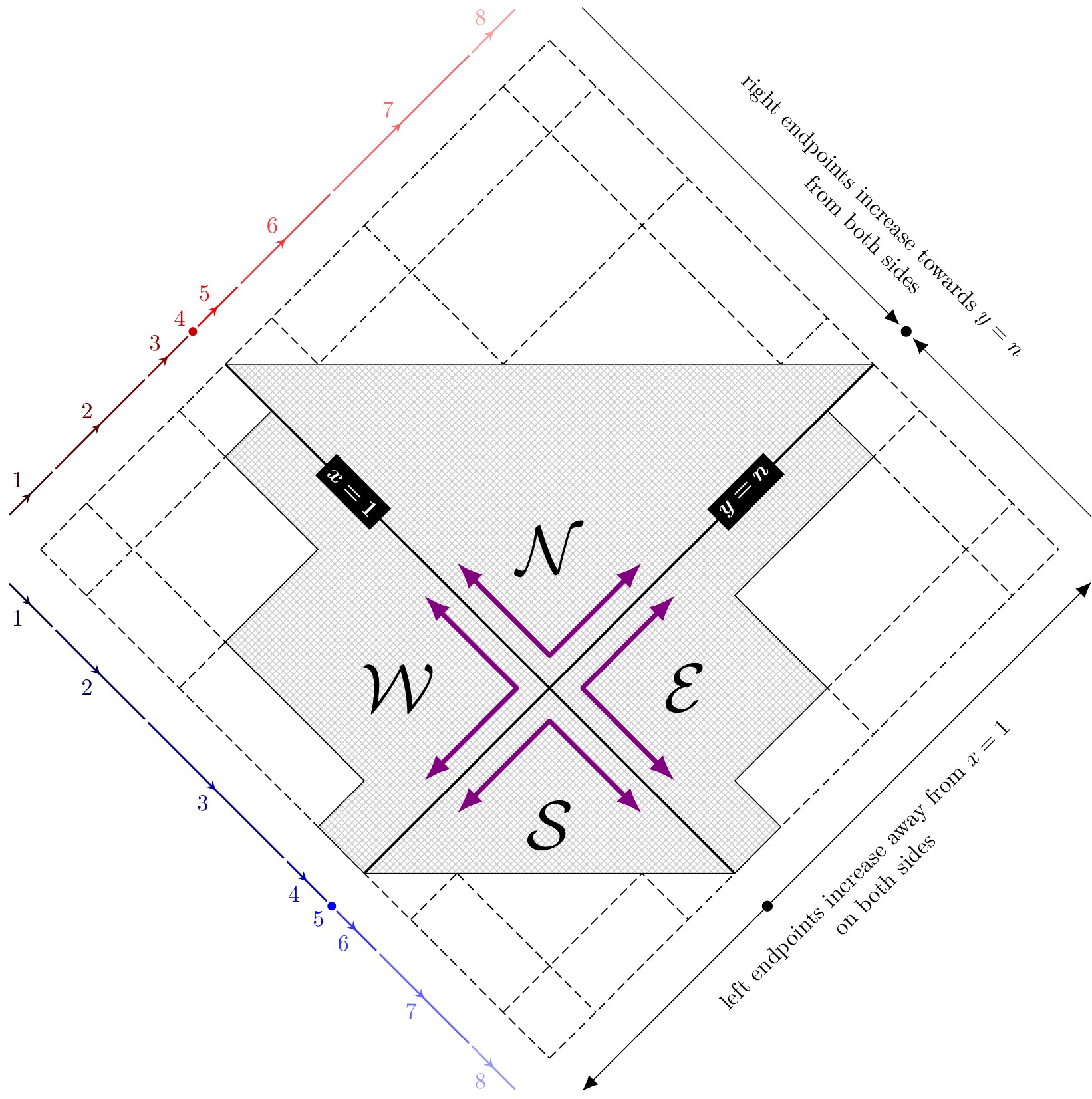}
    \caption{AR quiver for an orientation on \(\mathbb{A}_n\). The
      purple arrows denote the direction of decreasing dimension of
      indecomposables (always away from the diagonals \(x=1\) and
      \(y=n\)).}
    \label{fig_ar_super2}
  \end{figure}
\clearpage
}

\begin{alg}\label{alg_main} The construction of the left and right
  (\(x\) and \(y\)) axes of the AR quiver for some \(Q=\mathbb{A}_n\)
  are as follows.
  \begin{itemize}
  \item
    For the \(x\)-axis (south west to north east), list the vertices
    in the following order:
    
    Take all vertices of \(\mathbb{A}_n\) that are in some segment of
    the form \((\text{min},\text{next max}]\), and list them on the axis
    in \emph{reverse} \(\leq_{\mathbb{Z}}\) order. Then, take all
    remaining vertices and list them in forward \(\leq_{\mathbb{Z}}\)
    order.
    
    
    Note that the values of this x-axis 
    always increase \emph{away} from \(x=1\).

  \item
    For the \(y\)-axis (north west to south east), list the
    vertices in the following order:
    
    Take all vertices of \(\mathbb{A}_n\) that are in some segment of
    the form \([\text{max},\text{next min})\), and list them on the axis
    in forward \(\leq_{\mathbb{Z}}\) order. Then, take all remaining
    vertices and list them in \emph{reverse} \(\leq_{\mathbb{Z}}\)
    order.
    
    
    Note that the values of this y-axis 
    always increase \emph{toward} \(y=n\).
    
  \end{itemize}
\end{alg}

\begin{ex}
  In Figure \ref{fig_ar_super1}, we represent an orientation of \(\mathbb{A}_n\)
  with an implied arbirary density of vertices along the edges. Segments of the
  poset are taken and rearranged to form the \(x\) and \(y\) axes
  according to the algorithm.
\end{ex}

\begin{notation}\label{not_ewsn}
  From the separation made by the diagonals \(x=1\) and \(y=n\), we
  label the corresponding regions of the AR quiver by the four
  cardinal compass directions.

  \(\mathcal{E}_Q\subset\Sigma_Q\) is the collection of all interval
  modules \([x,y]\) where the vertex \(x\) is contained in some \(Q\)
  interval of the form \((\text{sink},\text{next source}]\), and \(y\)
  is in some \([\text{source},\text{next sink})\) (and \(x\neq 1,y\neq
  n\)).

  \(\mathcal{W}_Q\subset\Sigma_Q\) is the collection of all interval
  modules \([x,y]\) where the vertex \(x\) is contained in some
  \((\text{source},\text{next sink}]\), and \(y\) is in some
  \([\text{sink},\text{next source})\) (\(x\neq 1,y\neq n\)).

  \(\mathcal{S}_Q\subset\Sigma_Q\) is the collection of all interval
  modules \([x,y]\) where the vertex \(x\) is contained in some
  \((\text{source},\text{next sink}]\), and \(y\) is in some
  \([\text{source},\text{next sink})\) (\(x\neq 1,y\neq n\)).

  \(\mathcal{N}_Q\subset\Sigma_Q\) is the collection of all interval
  modules \([x,y]\) where the vertex \(x\) is contained in some
  \((\text{sink},\text{next source}]\), and \(y\) is in some
  \([\text{sink},\text{next source})\) (\(x\neq 1,y\neq n\)).

  Let \(\bar{\mathcal{E}}\) (similarly \(\bar{\mathcal{W}}\),
  \(\bar{\mathcal{S}}\), \(\bar{\mathcal{N}}\)) denote the original region
  along with all diagonal modules (those with either \(x=1\) or
  \(y=n\)) that are adjacent to it in the AR quiver. In addition, in
  all four cases, let this set also include the module \([1,n]\).
\end{notation}

\begin{remark}\label{rmk_monotone}
  Within each of the regions \(\bar{\mathcal{E}}\), \(\bar{\mathcal{W}}\),
  \(\bar{\mathcal{S}}\), and \(\bar{\mathcal{N}}\), the \(x\) and \(y\)
  coordinate axes are \emph{monotone} (Figure \ref{fig_ar_super2}).
\end{remark}

The following is a direct consequence of Algorithm \ref{alg_main} (and
Remark \ref{rmk_monotone}).

\begin{prop}{Formula for \(\delta_\dar\).}\label{prop_formula}
  
  Let \(\sigma = [x_1,y_1]\) and \(\tau = [x_2,y_2]\) be indecomposables over
  \(Q\). Then the graph distance \(\delta_\dar(\sigma,\tau)\) of
  Definition \ref{def_dar} is given by \[
    \delta_\dar(\sigma,\tau)=\delta^x(x_1,x_2)+\delta^y(y_1,y_2),
  \] where \[
    \delta^x(\sigma,\tau)=\left\{
      \begin{array}{ll}
        |x_1-x_2| & \text{if }\sigma,\tau\in\bar{\mathcal{E}}\cup\bar{\mathcal{N}}\\
                  & \text{ or }\sigma,\tau\in\bar{\mathcal{W}}\cup\bar{\mathcal{S}},
        \\
        x_1-1+x_2-1 & \text{otherwise}.
      \end{array}\right.
  \] and \[
    \delta^y(\sigma,\tau)=\left\{
      \begin{array}{ll}
        |y_1-y_2| & \text{if
                    }\sigma,\tau\in\bar{\mathcal{W}}\cup\bar{\mathcal{N}}\\
                  & \text{ or }\sigma,\tau\in\bar{\mathcal{E}}\cup\bar{\mathcal{S}},
        \\
        n-y_1+n-y_2 & \text{otherwise}.
      \end{array}\right.
  \]
\end{prop}

Proposition \ref{prop_formula} follows immediately from the
monotonicity of the two axes in each of the four regions of the AR
quiver.

\subsection{Distance to Zero in \(D_\dar\)}

The dimension of an indecomposable is a lower bound for its
\(W_\dar\) value. The following characterizes precisely when this is
achieved.

\begin{prop}\label{prop_decreasing_path}
  For any indecomposable \(\sigma = [x,y]\), \(W_\dar(\sigma)\geq
  y-x+1\). Furthermore, \(W_\dar(\sigma)=y-x+1\) if and only if there
  is a path of decreasing dimension from \(\sigma\) to a simple
  indecomposable in the AR quiver.
\end{prop}

\begin{proof}
  The first statement is immediate from the dimension-weighting of the
  edges in the definition of \(\delta_\dar\) (Definition
  \ref{def_ar_dist}) and the induced distance \(D_\dar\) (Definition
  \ref{def_dar}).
  
  Let \(\sigma=[x,y]\) be an indecomposable with decreasing path to
  some simple \([t]\). Then necessarily \(x\leq t\leq y\), and the
  existence of a decreasing path guarantees that \([x,y]\) and \([t]\)
  are in the same
  \(\bar{\mathrm{compass}}\) region. Hence,
  \(\delta_\dar([x,y],[t])=t-x+y-t=y-x\), and so
  \(W_\dar(\sigma)=y-x+1=\mathrm{dim}(\sigma)\).

  The converse also follows from the definitions cited above.
  If there is not a path of decreasing dimension, then
  any path of minimal weight from \([x,y]\) to \([t]\) must be of the
  form \[
    [x,y]\to\ldots\to [x_1,y_1]\to[x_2,y_2]\to\ldots\to[t]
  \] where \([x,y]\supset[x_1,y_1]\subset[x_2,y_2]\supset[t]\). Then,
  \(\delta_\dar([x,y],[t])\geq t-x+y-t+(x_1-x_2)+(y_2-y_1)\) where at
  least one of the parenthetical terms is strictly positive.
\end{proof}

\begin{corollary}\label{cor_escape_east_west}
For any
indecomposable \([x,y]\in\bar{\mathcal{E}}\cup\bar{\mathcal{W}}\),
\[
  W_\dar(\sigma)=\mathrm{dim}(\sigma).
\]
\end{corollary}

\begin{proof}
  Note that the projective simple and injective simple
  indecomposable modules form (respectively) the outer corners of the
  east and west regions, and it is immediate from the shape of the AR
  quiver (Algorithm \ref{alg_main}) that there are decreasing
  paths from any module in \(\bar{\mathcal{E}}\) or \(\bar{\mathcal{W}}\) to
  one of these.
\end{proof}

For any indecomposable in the north and south
regions, from Figure \ref{fig_ar_super1}
we see that there exists a path of decreasing dimension to the flat
north or south \emph{boundary}, but these boundaries are
not comprised of exclusively simple indecomposables. This complicates
the situation for \(W_\dar(\sigma)\) when
\(\sigma\in\mathcal{N}\cup\mathcal{S}\).


\begin{definition}
The \emph{north boundary} is the collection of indecomposables that
comprise the very top of the AR quiver. As a consequence of Algorithm
\ref{alg_main} (see also Notation \ref{not_ewsn}) this is exactly the
set \[
  B_N=\{N_i=[\text{source},\text{next
    sink}]\}\cup\{[s]:s\not\in\cup_iN_i\}\subset\bar{\mathcal{N}}.
\] The intervals are listed left to right on the boundary of the
AR quiver in \emph{increasing} order of their endpoints (as a subset
of \(\mathbb{Z}\)). This is the
construction pictured above: the north boundary is all red intervals
and blue simples listed in sequence according to \(\leq_\mathbb{Z}\).

The \emph{south boundary} is \[
B_S=\{S_j=[\text{sink},\text{next
  source}]\}\cup\{[s]:s\not\in\cup_jS_j\}\subset\bar{\mathcal{S}}.
\] These are listed left to right in the AR quiver in
\emph{decreasing} order (as a subset of \(\mathbb{Z}\)).
\end{definition}

\begin{example}\label{ex_bd}
  Consider the orientation of \(Q=\mathbb{A}_{10}\) and its
  north and south boundaries as seen in Figure \ref{fig_trunc}.
  
  The red intervals are the starting points for finding intervals
  with \(W_\dar>\mathrm{dim}\). Do note first that by Corollary
  \ref{cor_escape_east_west} the red intervals \([1,2]\) and
  \([9,10]\) in fact satisfy \(W_\dar=\mathrm{dim}\) as they are in
  \(\bar{\mathcal{W}}\) and \(\bar{\mathcal{E}}\) respectively.

  The boundary intervals contained strictly within \(\mathcal{N}\) or
  \(\mathcal{S}\) are of
  potential concern. Any non-simple such indecomposables
  have \(W_\dar>\mathrm{dim}\). This is immediate by observing that
  all paths leading away from these indecomposables are paths of
  \emph{increasing} dimension, violating the condition of Proposition
  \ref{prop_decreasing_path}.
  \begin{figure}
    \centering
    \includegraphics[scale=.7]{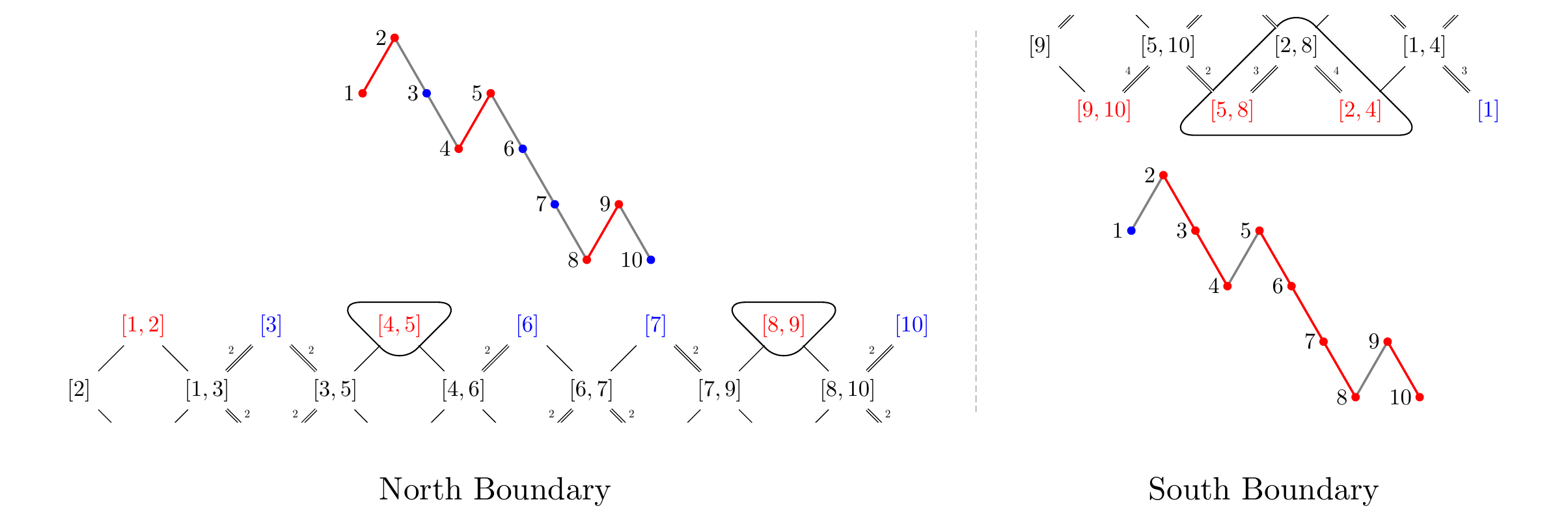}
    \caption{From Example \ref{ex_bd}. These are the truncated views
      of an AR quiver highlighting the
      structures of its north and south boundaries. The
      indecomposables with \(W_\dar>\mathrm{dim}\) are outlined.}
    \label{fig_trunc}
  \end{figure}
  These are still not the only intervals with \(W_\dar>\mathrm{dim}\),
  however. In this example, we see that the full collection of such
  intervals is
  \begin{itemize}
  \item North: \([4,5],[8,9]\).
  \item South: \([5,8],[2,4],[2,8]\).
  \end{itemize}
  The southern collection of intervals manifest the final feature of
  interest: since the boundary intervals \([5,8]\) and \([2,4]\) are
  adjacent, the interval \([2,8]\) caught above them also has no
  decreasing path to a simple indecomposable.
\end{example}

The preceding discussion motivates the following classification.

\begin{definition}\label{defn_hull}
For an orientation \(Q\) of \(\mathbb{A}_n\), define
\(\mathrm{hull}(Q)\) to be the union of the following sets:
\[
H_N=\{[\text{source},\text{sink}]\subset[2,n-1]:\text{any subintervals of
  the form }[\text{sink},\text{next source}]\text{ have length one}\},
\] and \[
H_S=\{[\text{sink},\text{source}]\subset[2,n-1]:\text{any subintervals of
  the form }[\text{source},\text{next sink}]\text{ have length one}\}.
\] Each set vaccuously includes the intervals with no interior
subintervals of opposite orientation.

Call \(H_N\subset\mathcal{N}\) the \emph{north hull} and
\(H_S\subset\mathcal{S}\) the \emph{south hull}.
\end{definition}

To conclude this section we will provide explicit formulas for
\(W_\dar(\sigma)\) when \(\sigma\in\mathrm{hull}(Q)\), resulting in
upper and lower bounds on \(W_\dar\) (Proposition \ref{prop_diam}).

\begin{lemma}{Values of \(W_\dar\) for
    \(\mathrm{hull}(P)\).}\label{lemma_hull_w}

  Let \(Q\) be some orientation of
  \(\mathbb{A}_n\) with non-empty hull.
  Suppose \([x,y]\in H_N\) (symmetrically, \([x,y]\in H_S\)). Define
  \([x_\bullet,y_\bullet]\) to be the largest interval containing
  \([x,y]\) that is also in \(H_N\). Let \(e=x_\bullet-1\) and
  \(E=y_\bullet+1\). Then \(W_\dar([x,y])\) is attained by passing
  through one of the simples \([e]\) or \([E]\). That is, \[
    W_\dar([x,y])=\min\{\delta_\dar([x,y],[e]),\delta_\dar([x,y],[E])\}+1.
  \] Moreover, the precise distances to these indecomposables are
  given by
  \begin{equation*}
    \delta_\dar([x,y],[e])=
    \left\{
      \begin{array}{ll}
        x+y-2 & \text{if }e>1\text{ and is the leftmost sink}\\
        \\
        x+y-2e & \text{otherwise}
      \end{array}\right.
  \end{equation*}
  and
  \begin{equation*}
    \delta_\dar([x,y],[E])=
    \left\{
      \begin{array}{ll}
        2E-(x+y) & \text{if }E<n\text{ and is the rightmost source}\\
        \\
        2n-(x+y) & \text{otherwise}.
      \end{array}\right.
  \end{equation*}
\end{lemma}

\begin{proof}
  Let \([x,y]\in H_N\), meaning that \(x\) is a source and \(y\) is a
  sink. Let \(1\leq t\leq n\).

  \textit{Case \(t<x_\bullet\) :}
  We proceed by possible regions in which
  \([t]\) may lie and give the corresponding \(\delta_\dar\).
  \begin{flalign*}
    & [t]\in\bar{\mathcal{N}}:
    & & \delta_\dar([x,y],[t]) = x+y-2t & \tag{low-N} \\
    & [t]\in\bar{\mathcal{E}}\setminus\bar{\mathcal{N}}:
    & & \delta_\dar([x,y],[t]) = 2(n-t)-(y-x) & \tag{low-E} \\
    & [t]\in\bar{\mathcal{W}}\setminus\bar{\mathcal{N}}:
    & & \delta_\dar([x,y],[t]) = x+y-2 & \tag{low-W} \\
    & [t]\in\mathcal{S}:
    & & \delta_\dar([x,y],[t]) = 2(n-1)-(y-x) & \tag{low-S}\\
\textit{Case \(y_\bullet<t\) :}\\
    & [t]\in\bar{\mathcal{N}}:
    & & \delta_\dar([x,y],[t])=2t-x-y \tag{high-N} \\
    & [t]\in\bar{\mathcal{E}}\setminus\bar{\mathcal{N}}:
    & & \delta_\dar([x,y],[t])=2n-x-y \tag{high-E} \\
    & [t]\in\bar{\mathcal{W}}\setminus\bar{\mathcal{N}}:
    & & \delta_\dar([x,y],[t])= 2(t-1)-(y-x) \tag{high-W} \\
    & [t]\in\mathcal{S}:
    & & \delta_\dar([x,y],[t])= 2(n-1)-(y-x) \tag{high-S}
  \end{flalign*}
  \textit{Case \(x_\bullet\leq t\leq y_\bullet\) :} The only
  possibilities are
  \begin{enumerate}
  \item \([t]\) is some source with \(x_\bullet\leq t<y_\bullet\), so
    \([t]\) is in the east region. That is, \[
      \delta_\dar([x,y],[t])=|x-t|+2n-y-t.
    \] Clearly, this value is minimized by all sources \(x\leq
    m<y_\bullet\). Choosing any of these gives us
    \begin{equation*}
      \delta_\dar([x,y],[t])=\delta_\dar([x,y],[x])=2n-x-y. \tag{mid-E}
    \end{equation*}
  \item \([t]\) is some sink with \(x_\bullet<t\leq y_\bullet\), so
    \([t]\) is in the west region. That is, \[
      \delta_\dar([x,y],[t])=x-1+t-1+|y-t|.
    \] Clearly, this value is minimized by all sinks \(x_\bullet<
    m\leq y\). Choosing any of these gives us
    \begin{equation*}
      \delta_\dar([x,y],[t])=\delta_\dar([x,y],[y])=x+y-2. \tag{mid-W}
    \end{equation*}

  \item \([t]\) is anything else, in which case it is interior to a
    segment of the form \([\textrm{source},\textrm{next sink}]\), and thus
    lies on the south boundary. That is,
    \begin{equation*}
      \delta_\dar([x,y],[t])=2(n-1)-(y-x) \tag{mid-S}
    \end{equation*}
  \end{enumerate}

  We exclude various equations from consideration.
  \begin{itemize}
  \item It is easy to check that (low-N) \(\leq\) (low-W) \(\leq\)
    (high-W) \(\leq\) (high-S). As \(x_\bullet\) is a source and
    \(x_\bullet\geq 2\), there always exists some sink \(t<x_\bullet\)
    (and thus
    \([t]\in\bar{\mathcal{W}}\)), so we need never use the biggest two
    equations. 
  \item Similarly, (high-N) \(\leq\) (high-E) \(\leq\)
    (low-E) \(\leq\) (low-S). As \(y_\bullet\) is a sink and
    \(y_\bullet\leq n-1\), there always exists some source
    \(t>y_\bullet\) (and thus
    \([t]\in\bar{\mathcal{E}}\)), so we need never use the biggest two
    equations. 
  \item All mid-type equations are unnecessary for consideration as
    well. Simply note that (mid-E) \(=\) (high-E), (mid-W) \(=\)
    (low-W), and (mid-S) = (low,high-S).
  \end{itemize}


  From this, we can conclude that no matter the poset orientation, the
  only candidates for minimizing \(\delta_\dar([x,y],[t])\) are
  (low-N), (low-W), (high-N), and (high-E).

  The only time that there is \emph{no} (low-N) candidate is if
  \(e\) is the leftmost sink and \(e\neq 1\). But in this case,
  \(e=x_\bullet-1\) is a candidate for (low-W).
  Conversely, if there is
  \emph{any} (low-N) candidate, then \(e=x_\bullet-1\) is a also a
  candidate, and minimizes the equation.

  The symmetric statements are true of (high-N) and (high-E), which
  are minimized by substituting \(E\).

  The statement of the lemma follows.
\end{proof}

\begin{lemma}\label{lemma_non_hull}
  If
  \([x,y]\in(\mathcal{N}\cup\mathcal{S})\setminus\mathrm{hull}(P)\),
  then \(W_\dar([x,y])=\mathrm{dim}([x,y])\).
\end{lemma}

\begin{proof}
  If \([x,y]\not\in\mathrm{hull}(P)\), then there exists \(t\in[x,y]\)
  such that \([t]\) is on the boundary of the same region in which
  \([x,y]\) lies. Then \(\delta_\dar([x,y],[t])=y-t+x-t\), and so
  \(W_\dar([x,y])=y-x+1=\mathrm{dim}([x,y])\).
\end{proof}

The subsequent proposition follows from Lemmas
\ref{lemma_hull_w}, \ref{lemma_non_hull} and Corollary
\ref{cor_escape_east_west}:

\begin{prop}\label{prop_diam} All intervals \(\sigma\) have the
  property that \[
    \mathrm{dim}(\sigma)\leq W_\dar(\sigma)\leq n.
  \]
  The set \(\mathrm{hull}(P)\) is precisely the collection of
  intervals \(\sigma\) such that
  \(W_\dar(\sigma)>\mathrm{dim}(\sigma)\). Furthermore, the diameter
  \(W_\dar=n\) is always attained by the indecomposable \([1,n]\).
\end{prop}

And, as \(D_\dar(\sigma,\tau)\leq\max\{W_\dar(\sigma),W_\dar(\tau)\}\)
for all pairs \(\sigma,\tau\), we get the following corollary.

\begin{corollary}\label{cor_dar_n}
  For any \(P=\mathbb{A}_n\), \(D_\dar\leq n\).
\end{corollary}

\subsection{Behavior of \(D_\dar\) on Pure Zigzag
  Orientations}\label{sec_zz_shape}

Recall that in Definition \ref{def_an} we say \(P=\mathbb{A}_n\) has
\emph{pure zigzag orientation} if the directions of
any two adjacent arrows are opposite; alternatively, if every vertex is
a source (minimal) or a sink (maximal).

\begin{figure}
  \centering
  \includegraphics[scale=1]{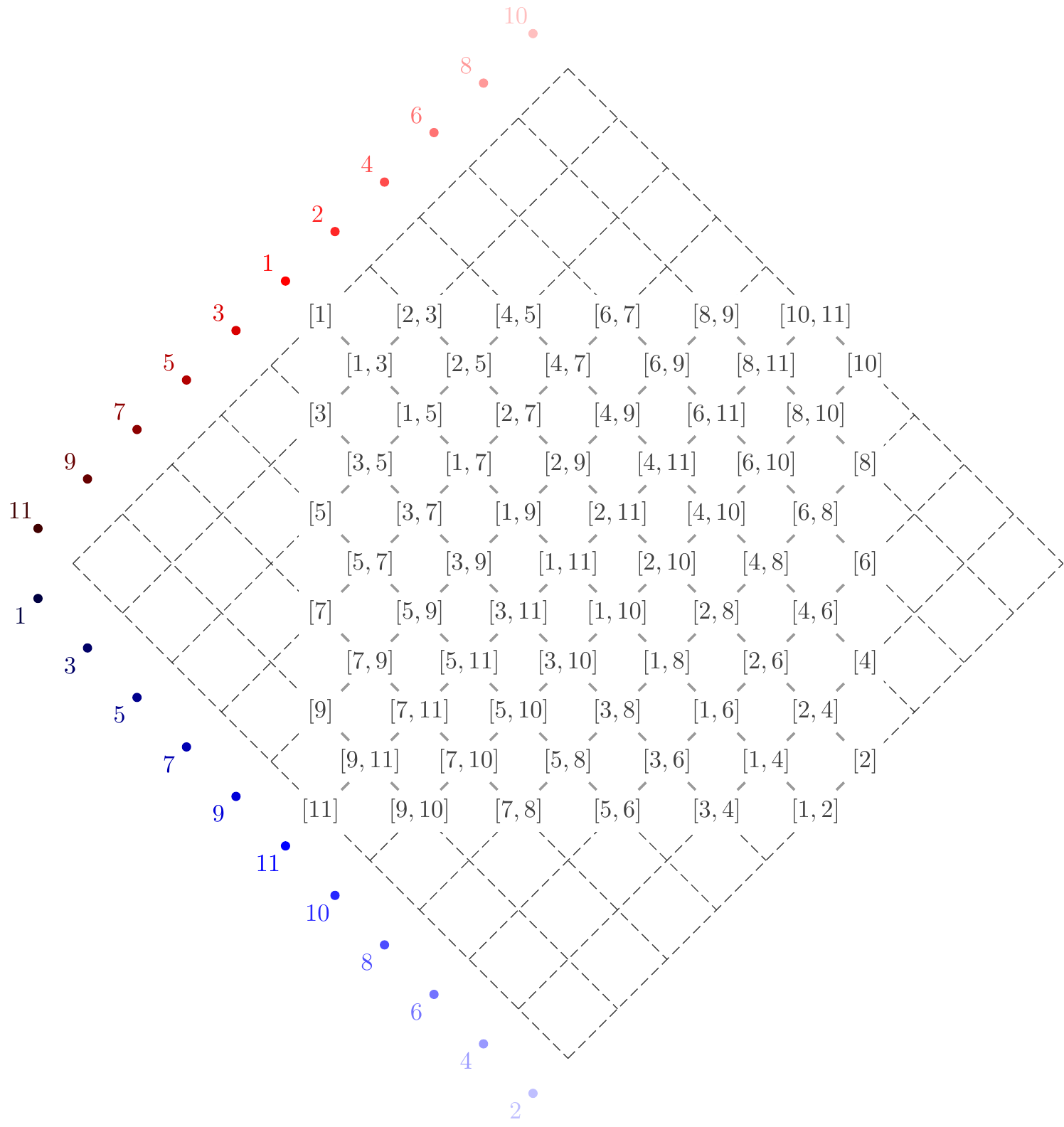}
  \caption{AR quiver of the \(\mathbb{A}_{11}\) zigzag quiver with
    upward orientation.}
  \label{fig_zig_ar}
\end{figure}

As zigzag is an orientation that is often of particular independent
interest, we will here espouse some properties of \(D_\dar\)
specifically for the zigzag setting.

The Auslander-Reiten quiver of a zigzag orientation of
\(\mathbb{A}_{11}\) is shown in Figure \ref{fig_zig_ar}.

\begin{notation} There are
slight differences in the AR quiver based on the original orientation
starting and ending at a max or min. This results in four zigzag
orientation types, which we label as follows for convenience:
\begin{itemize}
\item in (uu) orientation, \(1\) and \(n\) are sinks,
\item in (ud) orientation, \(1\) is a sink and \(n\) is source,
\item in (du) orientation, \(1\) is a source and \(n\) is a sink,
\item in (dd) orientation, \(1\) and \(n\) are sources.
\end{itemize}
\end{notation}

\begin{remark}[Hull of zigzag orientation]
  From Definition \ref{defn_hull}, we immediately see that an
  \(\mathbb{A}_n\) quiver with zigzag orientation has
  \(H_N=\{[\textrm{min},\textrm{max}]\subset[2,n-1]\}\) and
  \(H_S=\{[\textrm{max},\textrm{min}]\subset[2,n-1]\}\). That is,
  \(\mathrm{hull}(P)\) is precisely the
  the entire north and south regions of AR quiver (which excludes the
  diagonals).
\end{remark}

As an immediate consequence of Lemma \ref{lemma_hull_w}, we have the
following.

\begin{corollary}[to Lemma \ref{lemma_hull_w}]\label{cor_1zz_w}
  For a zigzag orientation of some \(\mathbb{A}_n\) quiver, any
  \(\sigma=[x,y]\) in \(\mathrm{hull}(P)\) has \[
    W_\dar(\sigma)=\min\{x+y-1,2n-x-y+1\}.
  \]
\end{corollary}

\begin{example}\label{ex_bad}
  For \(D_\dar\) over zigzag orientations, there will be intervals of
  small dimension and large \(W_\dar\) value. Consider
  \(\mathbb{A}_{100}\) with either (ud) or (du) zigzag
  orientation. In either case,
  \(\sigma=[50,51]\) has a dimension of \(2\),
  but by Corollary \ref{cor_1zz_w},
  \(W_\dar(\sigma)=100=\mathrm{diam}(W_\dar)\) (Proposition
  \ref{prop_diam}).
\end{example}

\begin{example}\label{ex_worse}
  To extend the previous example to any zigzag orientation of
  \(\mathbb{A}_n\), consider:
  \begin{itemize}
  \item if \(n\) is even, the indecomposable \([n/2,n/2+1]\) has
    dimension \(2\) and \(W_{\mathrm{AR}}\) value of \(n\),
  \item if \(n\) is odd, the indecomposable
    \([\frac{n-1}{2},\frac{n-1}{2}+1]\) has dimension \(2\) and
    \(W_{\mathrm{AR}}\) value of \(n-1\).
  \end{itemize}
\end{example}

\begin{remark}\label{rmk_alt_zigzag}
  \afterpage{
    \begin{figure}
      \centering

      \begin{subfigure}[t]{1\textwidth}
        \centering
        \includegraphics[scale=.4]{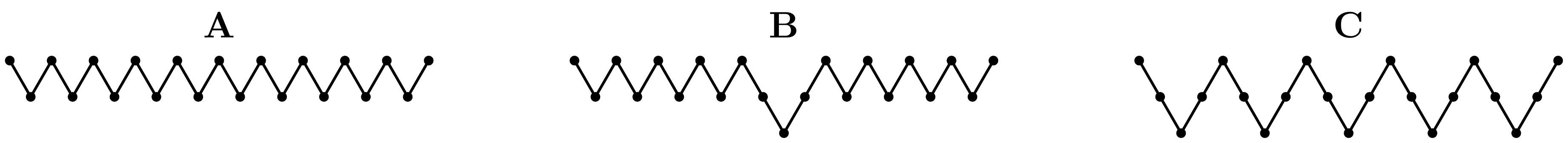}
      \end{subfigure}

      \vspace{.5cm}
      \begin{subfigure}[t]{1\textwidth}
        \centering
        \includegraphics[width=\textwidth]{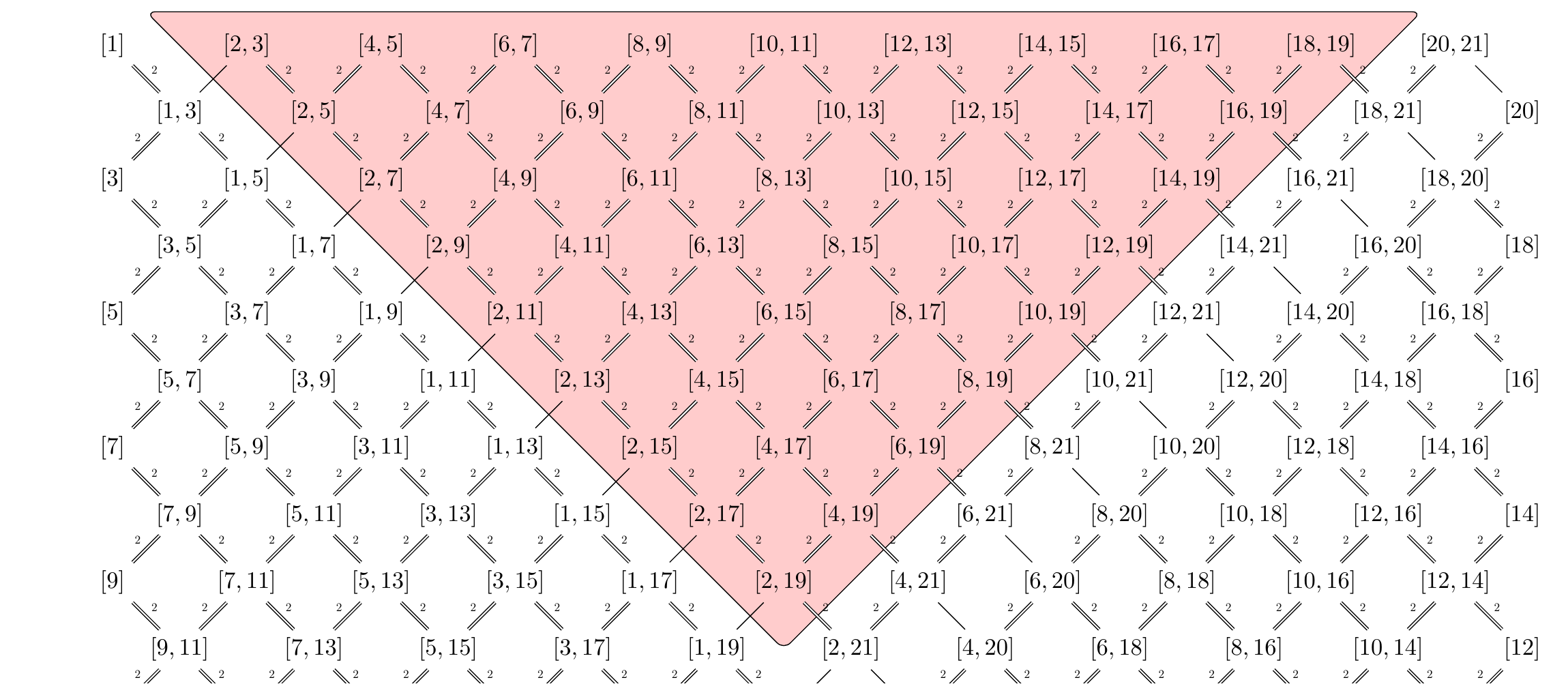}
        \caption{'s AR quiver.}
        \label{fig_7_1}
      \end{subfigure}
      
      \vspace{.5cm}
      \begin{subfigure}[t]{1\textwidth}
        \centering
        \includegraphics[width=\textwidth]{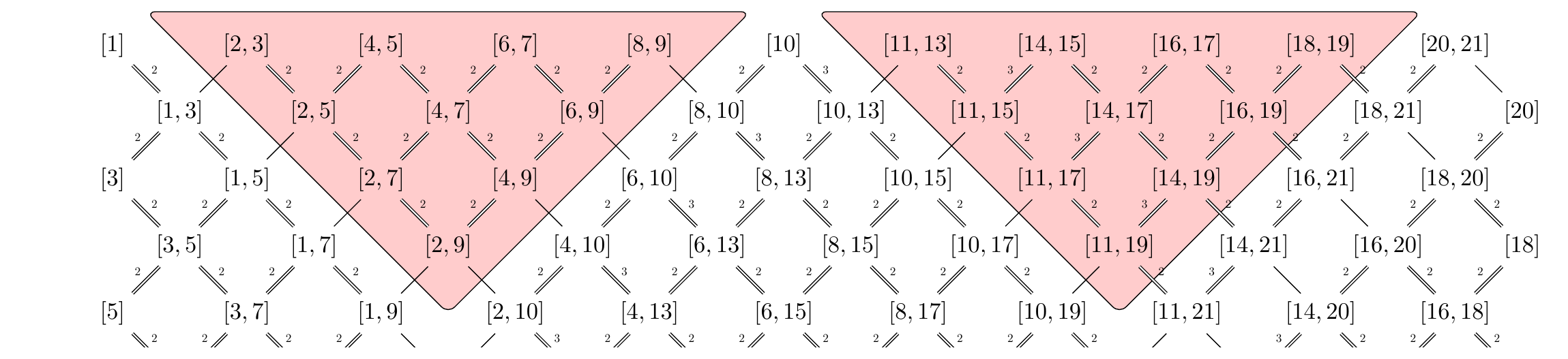}
        \caption{'s AR quiver.}
        \label{fig_7_2}
      \end{subfigure}
      
      \vspace{.5cm}
      \begin{subfigure}[t]{1\textwidth}
        \centering
        \includegraphics[width=\textwidth]{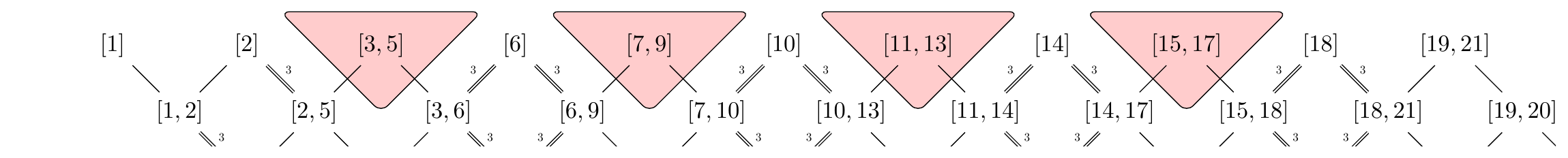}
        \caption{'s AR quiver.}
        \label{fig_7_3}
      \end{subfigure}
 
      \vspace{.5cm}
     
      \caption{Three orientations of \(\mathbb{A}_{21}\) and the north
        boundaries of their AR quivers. Not only do the ``problem''
        regions shrink, but the disparities between \(W_\dar\) and
        dimension of the problem indecomposables also shrink.}
      \label{fig_boundary_ex_triple}
    \end{figure}
    \clearpage
  }
  Note that any orientation less than ``pure'' zigzag (Figure
  \ref{fig_7_1}) will possess
  reduced dimension-to-\(W_\dar\) disparities.

  For example, consider a poset with zigzag orientation everywhere
  save for the middle of the poset, in which there is a consecutive
  pair of rightward (or leftward) edges
  \(\rightarrow\rightarrow\) (Figure \ref{fig_7_2}). This splits the
  entire north
  region from one giant hull into two hulls by introducing the simple
  \([10]\) in the middle of the north boundary, providing a path of
  decreasing dimension to a simple for many modules formerly in the
  hull.

  For another example, if
  \(\mathbb{A}_n\) has orientation
  \(\cdot\cdot\cdot\hspace{-.1cm}
  \rightarrow\rightarrow\leftarrow\leftarrow\rightarrow\rightarrow
  \hspace{-.1cm}\cdot\cdot\cdot\)
  where the zigzag feature switches every \emph{other} vertex (Figure
  \ref{fig_7_3}), then it turns out that the
  difference \(W_\dar(\sigma)-\mathrm{dim}(\sigma)\in\{0,2\}\) for all
  indecomposables \(\sigma\) due to a high distribution of simples
  over the north boundary.

  The last orientation in this example proves to be a worthwhile
  course of investigation for zigzag persistence, and is the focus of
  Section \ref{sec_stab_r}.
\end{remark}

\subsection{\(D_\dbl\) and \(D_\dar\): Features and
  Stability}\label{sec_further}

In this section we discuss the block distance \(D_\dbl\) of
\cite{botnan_lesnick} and explore the differences and similarities
between \(D_\dbl\) and the Auslander-Reiten quiver distance
\(D_\dar\).

There is one rather cumbersome notational concern to be overcome when
considering these two distances: for quiver theoretic purposes we have
labeled our vertices in sequential order on the zigzag quiver itself,
while recent literature considers zigzag intervals as indexed over
a particular poset denoted \(\mathbb{Z}\mathbb{Z}\), which then
corresponds to some persistence module in \(\mathbb{R}^2\). The
disparity of notation and structure will be addressed
with care when it comes time to consider the distances side by side
(Definition \ref{defn_convert}),
but is worth bearing in mind throughout. As such we will take to the
following convention:

\begin{notation}
  For a zigzag interval \(I\), denote by \(I_\an\) the interval as
  viewed over a \(\mathbb{Z}\)-labeled \(\mathbb{A}_n\) quiver, and by
  \(I_\zz\) a corresponding interval over the poset \(\zz\)
  (Definition \ref{defn_zz}).

  Of a final note is that there is \emph{no canonical association} of
  vertices in \(\mathbb{A}_n\) with points in \(\zz\). Throughout, we
  refuse to declare any point at which \(\mathbb{A}_n\) and \(\zz\) are
  ``fused''. The reader is encouraged to keep this in mind during
  the subsequent material, and to be convinced that this lack of choice
  is of no consequence to the work provided. This is in fact ideal
  when taking into account that we will eventually consider extending
  to \emph{limits} of zigzag quivers with unbounded length (Section
  \ref{subsub_limits}).
\end{notation}

\subsubsection{Posets}

\begin{definition}\label{defn_zz}
  Let \(\mathbb{Z}\mathbb{Z}\) be the poset consisting of all points
  \(\{(i,i),(i,i-1)\in\mathbb{Z}^2\}_{i\in\mathbb{Z}}\) and having the
  subposet order
  inherited from \(\mathbb{Z}^{\mathrm{op}}\times\mathbb{Z}\).
  Generally, an \emph{interval} of this poset is written as \(\langle
  i,j\rangle\), which denotes one of \[
    (i,j),[i,j),(i,j],\text{ or }[i,j].
  \] An interval \(\langle i,j\rangle\) in \(\zz\) is the convex set \[
    \langle i,j\rangle=\{(x,y):i\sim x,y\sim j\},
  \] where the \(\sim\) represent either \(\leq\) or \(<\) depending
  on the respectively closed or open endpoints of \(\langle i,j\rangle\).

  An \emph{interval representation} of \(\mathbb{Z}\mathbb{Z}\) is written
  \(\langle i,j\rangle_{\mathbb{Z}\mathbb{Z}}\). For any point
  \((x,y)\in\mathbb{Z}\mathbb{Z}\),
  \[
    \langle i,j\rangle_{\mathbb{Z}\mathbb{Z}}(x,y)=\left\{
    \begin{array}{ll}
      K & \text{if }(x,y)\in\langle i,j\rangle \\
      \\
      0 & \text{otherwise}.\\
    \end{array}\right.
  \] The internal maps of \(\langle
  i,j\rangle_{\mathbb{Z}\mathbb{Z}}\) are \(1_K\) where possible, and
  \(0\) otherwise.


    

\end{definition}

\begin{definition}
  Let \(\uu\subset\mathbb{R}^{\mathrm{op}}\times\mathbb{R}\)
  be the subposet consisting of all points \((x,y)\in\mathbb{R}^2\)
  such that \(x\leq y\). We will have it inherit the ordering of
  \(\mathbb{R}^{\mathrm{op}}\times\mathbb{R}\): that
  \((x,y)\leq(w,z)\) if and only if \(x\geq w\) and \(y\leq z\).

\end{definition}

The connection between \(\zz\) and \(\uu\) as subposets of
\(\mathbb{R}^{\mathrm{op}}\times\mathbb{R}\) is shown in Figure
\ref{fig_zz_and_u}.

\begin{figure*}[t]
  \centering
  \includegraphics[scale=1]{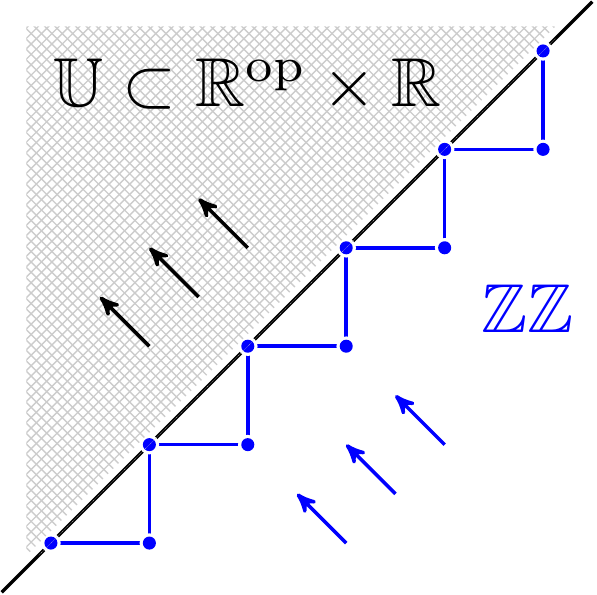}
  \caption{A visualization of the connection between the posets
    \(\mathbb{Z}\mathbb{Z}\) and \(\uu\) as subposets of
    \(\mathbb{R}^{\mathrm{op}}\times\mathbb{R}\). The arrows denote
    the diagonal increasing vector under
    \(\leq_{\mathbb{R}^{\mathrm{op}}\times\mathbb{R}}\).}
  \label{fig_zz_and_u}
\end{figure*}


\begin{definition}[See \cite{botnan_lesnick} sections 2.5 and 3
  for original details]\label{defn_embed}

  For a point \(u\in\uu\), define \(\zzu\) to be the
  subposet of \(\zz\) consisting of all the points of \(\zz\) that are
  \(\leq u\) when considering both \(\uu\) and \(\zz\) as
  subposets of \(\rop\) (see
  Figure \ref{fig_zz_and_u_nice}).

  \begin{figure*}[t]
    \centering
    \includegraphics[scale=1]{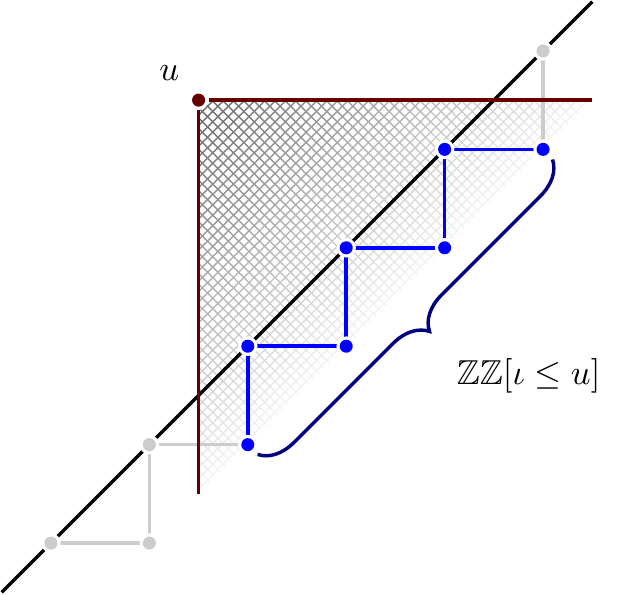}
    \caption{The restriction of the poset \(\zz\) under a point
      \(u\in\uu\).}
    \label{fig_zz_and_u_nice}
  \end{figure*}
  
  For a zigzag persistence module \(M_\zz\), define \(M|_{\zz[\leq
    u]}\) to be the restriction of \(M\) to the subposet
  \(\zzu\). Then define the colimit functor
  \(\tilde{E}:\mathrm{vect}^{\zz}\to\mathrm{vect}^{\rop}\)
  by: \[\tilde{E}(M)(u)=\varinjlim M|_{\zz[\leq u]},\] the colimit
  of the diagram given by the \([\leq u]\) restriction, for every
  \(u\in\uu\subset\rop\).
    
  Under \(\tilde{E}\), interval \(\zz\) modules are sent to the
  following block modules. (See figure \ref{fig_all_embed}.)
  %
  \begin{center}
    \begin{tabular}{ l l l }
      \(\tilde{E}((i,j)_\zz)\)&\(=(i,j)_\dbl\)&\(=\{(x,y)\in\uu:i<x,y<j\}\)\\
      \(\tilde{E}([i,j)_\zz)\)&\(=[i,j)_\dbl\)&\(=\{(x,y)\in\uu:i\leq y<j\}\)\\
      \(\tilde{E}((i,j]_\zz)\)&\(=(i,j]_\dbl\)&\(=\{(x,y)\in\uu:i<x\leq j\}\)\\
      \(\tilde{E}([i,j]_\zz)\)&\(=[i,j]_\dbl\)&\(=\{(x,y)\in\uu:x\leq i, j\leq y\}\)
    \end{tabular}
  \end{center}
\end{definition}

\begin{figure*}[t]
  \centering
  \begin{tabular}{ c c c c }
      \includegraphics[scale=.5]{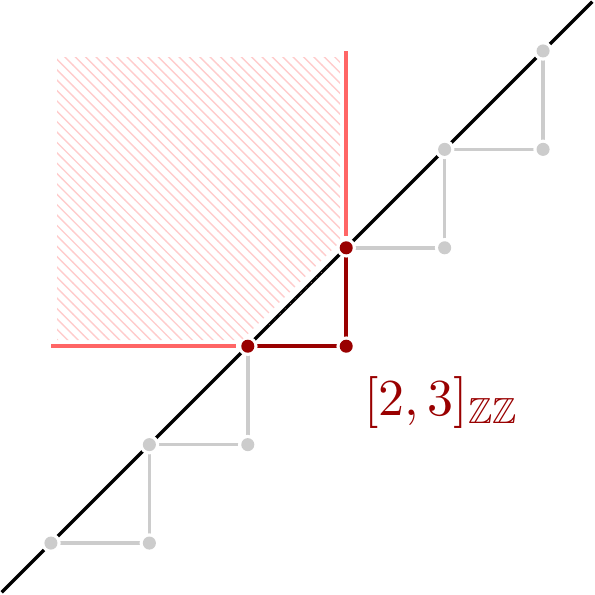}&\hspace{.5cm}
      \includegraphics[scale=.5]{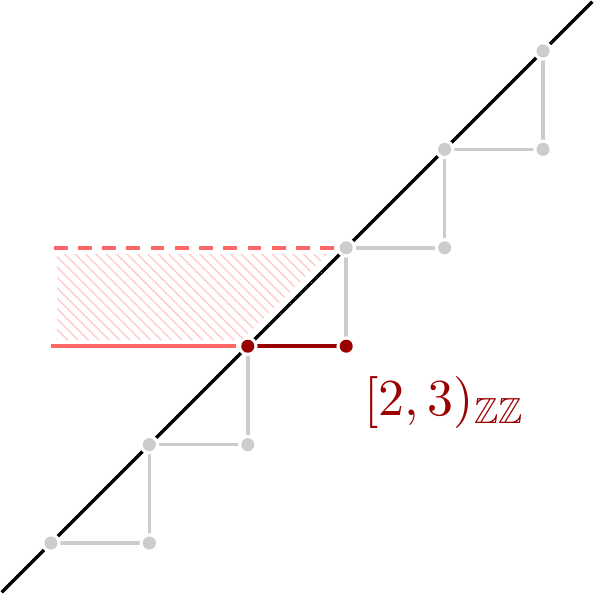}&\hspace{.5cm}
      \includegraphics[scale=.5]{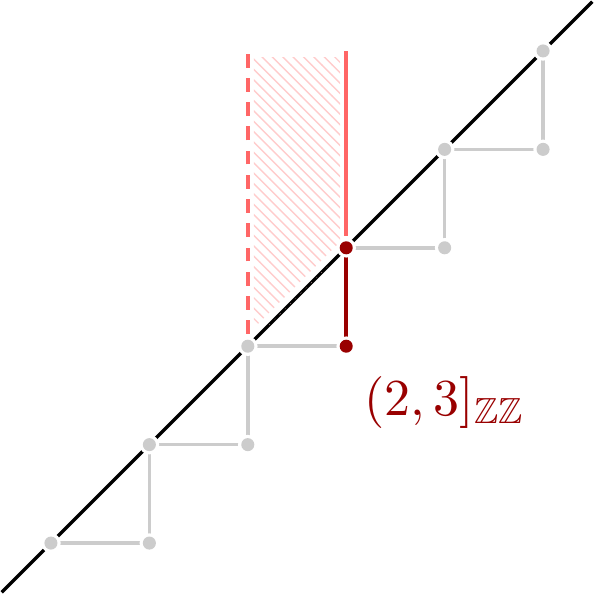}&\hspace{.5cm}
      \includegraphics[scale=.5]{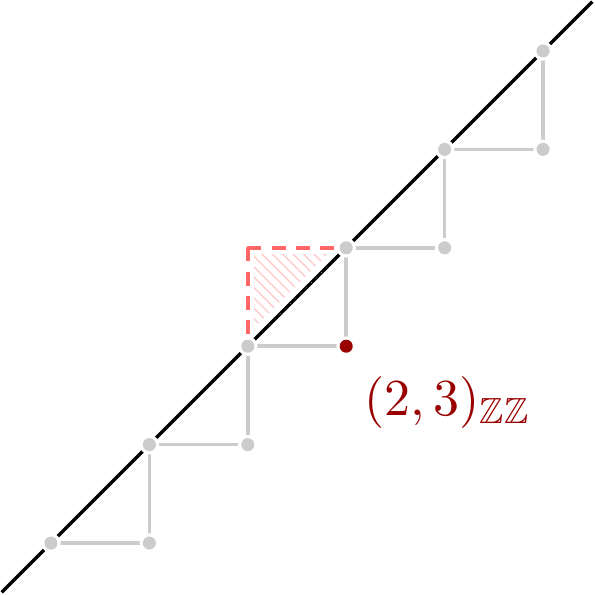}
  \end{tabular}
  \caption{The various types of \(\zz\) interval modules and their
    corresponding \(\uu\) modules under \(\tilde{E}\).}
  \label{fig_all_embed}
\end{figure*}

\subsubsection{Block Module Distance}
\label{sec_block}

\begin{definition}\label{defn_u_il}
  Let \(D_\uu\) denote the \emph{interleaving distance} on \(\uu\):

  Let \(\bar{\epsilon}=(-\epsilon,\epsilon)\in\rop\) be the
  ``increasing'' vector of length \(\epsilon\) for
  \(\uu\). For a \(\uu\) persistence module \(M\), define
  the new \(\uu\) persistence module
  \(M(\bar{\epsilon})(u)=M(u+\bar{\epsilon})\). Similarly, for a
  morphism of persistence modules \(\phi\), define
  \(\phi(\bar{\epsilon})(u)=\phi(u+\bar{\epsilon})\).

  For any \(M\), \(\epsilon\), let \(1_{M,M(\bar{\epsilon})}\) be the
  morphism that takes the value \(1_K\) on
  \(u\in\mathrm{supp}(M)\cap\mathrm{supp}(M(\bar{\epsilon}))\), and is
  zero otherwise. (It is simple to check that the \(K\)-span of this
  morphism gives precisely \(\mathrm{Hom}(M,M(\bar{\epsilon}))\).)

  Two \(\uu\) persistence modules are said to be
  \(\epsilon\)-interleaved if there exist morphisms \(\phi:M\to
  N(\bar{\epsilon})\) and \(\psi:N\to M(\bar{\epsilon})\) such that
  \begin{itemize}
  \item \(\psi(\bar{\epsilon})\circ\phi=1_{M,M(2\bar{\epsilon})}\), and
  \item \(\phi(\bar{\epsilon})\circ\psi=1_{N,N(2\bar{\epsilon})}\).
  \end{itemize}
  
  For two \(\uu\)
  persistence modules \(M,N\), \[D_\uu(M,N)=\inf\{\epsilon:M,N\text{
      are }\epsilon\text{-interleaved}\}.\]

  (In full generality, this definition would make use of arbitrary
  \(\uu\)-translations, but implicit in this distance is the use of an
  \(\ell^\infty\) norm, in which case we may as well default to the
  \emph{diagonal} vector of length \(\epsilon\) to define the
  translation at all points. This aligns with the earlier notion of a
  \emph{full} translation of a given height.)
\end{definition}

\begin{definition}
  From Definitions \ref{defn_embed} and \ref{defn_u_il}, define the
  block distance to be the composition
  \[
    D_\dbl(M_\zz,N_\zz)\vcentcolon= (D_\uu\circ \tilde{E}) (M_\zz,N_\zz)
    = D_\uu(\tilde{E}(M_\zz),\tilde{E}(N_\zz)).
  \]

\end{definition}

\begin{prop}[\cite{botnan_lesnick}, Lemma 3.1]\label{prop_features_bl}
  
  The bottleneck distance induced by \(D_\dbl\) can be generated by the
  following \(W_\dbl\) and \(d_\dbl\).

  \begin{itemize}
  \item
    \(W_\dbl((i,j)_\zz
    )=1/4(j-i).
      \)
  \item
    \(W_\dbl([i,j]_\zz
    )=\infty.
      \)
    \item
      \(W_\dbl([i,j)_\zz
      )=1/2(j-i).
      \)
    \item
      \(W_\dbl((i,j]_\zz
      )=1/2(j-i).
      \)
  \end{itemize}

  If \(\langle i_1,j_1\rangle_\zz\) and
  \(\langle i_2,j_2\rangle_\zz\)
  are two zigzag/block modules \emph{of the same endpoint parity}, then
  \begin{itemize}
  \item 
    \(d_\dbl
    (\langle i_1,j_1\rangle_\zz,\langle
           i_2,j_2\rangle_\zz)
    =
    \max\{|i_1-i_2|,|j_1-j_2|\}
           \)
  \end{itemize}
  Otherwise, define
  \(d_{\mathrm{BL}}=\max\{W_\dbl(\lr{i_1,j_1}_\zz),W_\dbl(\lr{i_2,j_2}_\zz)\}\),
  the max of the \(W\)-values.
  
\end{prop}

The above result on interval modules is obtained from the more general
definition, in which the projection of \(\zz\) interval modules to
\(\dbl\) interval modules is by \emph{left Kan extension} via
colimit. See the original work \cite{botnan_lesnick} for more detail.

\subsubsection{Intervals of zigzag \(\mathbb{A}_n\) as intervals of
  \(\mathbb{Z}\mathbb{Z}\)}
Finally, in order to make comparisons between \(D_\dar\) and
\(D_\dbl\), we need to be able to relate \(\mathbb{A}_n\) modules to
\(\zz\) modules before embedding via \(\tilde{E}\).

\begin{definition}\label{defn_convert}
  For some \(P=\mathbb{A}_n(z)\) define the functor
  \(\mathcal{Z}:\mathrm{vect}^P\to\mathrm{vect}^\zz\) by how it acts
  on the following indecomposables. For any \(x\in P\), there is some
  associated \((i,i)\in\zz\) (the positioning in which \(P\) is
  ``fused'' to \(\zz\) is fixed ahead of time and is entirely
  arbitrary).
  \begin{itemize}
  \item \(\mathcal{Z}([x+1,x+2k-1]_\an)=(i,i+k)_\zz\).
  \item \(\mathcal{Z}([x,x+2k]_\an)=[i,i+k]_\zz\).
  \item \(\mathcal{Z}([x,x+2k-1]_\an)=[i,i+k)_\zz\).
  \item \(\mathcal{Z}([x+1,x+2k]_\an)=(i,i+k]_\zz\).
  \end{itemize}
\end{definition}

\begin{definition}
  Let \(P=\mathbb{A}_n(z)\) and let \(Z\) be the \(\zz\)-interval (not
  module) given by
  \(\mathcal{Z}([1,n]_\an)\). Define \(\Sigma_\zz(P)\) to be the
  subcategory of \(\mathrm{vect}^\zz\) given by all modules with
  support contained in the \(\zz\)-interval
  \(Z=\mathcal{Z}([1,n]_\an)\). 
\end{definition}

\begin{prop}
  The functor (natural transformation)
  \[
    \mathcal{Z}:\mathrm{vect}^P\to\Sigma_\zz(P)
  \] is an equivalence of categories (natural equivalence).
\end{prop}

\begin{proof}
  The inverse of \(\mathcal{Z}\) is given by the reverse statements
  of Definition \ref{defn_convert}.
\end{proof}

\begin{definition}
  For a \(\zz\) module \(I_\zz\), define dimension
  \(\mathrm{dim}(I_\zz)=\sum\limits_{i\in\zz}\mathrm{dim}_K(I_\zz(i))\) to
  be the sum of the dimensions of the vector spaces of \(I_\zz\).
\end{definition}

\begin{notation}
  In any setting where we have fixed some \(P=\mathbb{A}_n(z)\) and
  some \(\mathcal{Z}:\mathrm{vect}^P\to\Sigma_\zz(P)\) (``some'' only
  because this is technically dependent on our consistently hand-waved
  choice of \(\an\leftrightarrow\zz\) anchor), we will drop the
  equivalence \(\mathcal{Z}\) altogether and simply denote by
  \(\sigma_\an\) and \(\sigma_\zz\) the same module viewed as a
  member of either of the two equivalent categories.
  
  Also, despite the disparity in \emph{labeling} between \(\an\) and
  \(\zz\) modules (Definition \ref{defn_convert}), the
  \emph{dimension} of \(\sigma\) is the same in both contexts: \[
    \mathrm{dim}_\an(\sigma_\an)=\mathrm{dim}_\zz(\mathcal{Z}(\sigma_\an)=\sigma_\zz).
  \] For this reason, we will simply write \(\mathrm{dim}\)
  with no need for subscripting based on the category.
\end{notation}

\section{Stability Between \(D_\dbl\) and \(D_\dar\) over Pure Zigzag}
\label{sec_stab}

Algebraic stability results usually refer to obtaining bounds between
some distance and its induced bottleneck distance (Remark
\ref{rmk_induced}). The following are two important examples of
stability that have been paraphrased into this paper's vocabulary.

The first stability result is in fact an \emph{isometry}.

\begin{theoremnonum} For \(\mathrm{vect}\)-valued
  persistence modules over \(\mathbb{R}\), the interleaving distance
  and its induced bottleneck distance are \emph{isometric}.

  That is, the interleaving distance can be taken to be
  \emph{diagonal} over the indecomposable summands without any loss
  of sharpness.
\end{theoremnonum}

The fact that a distance is a lower bound on its own induced
bottleneck distance is trivial. The non-trivial direction for the above
result is seen originally in \cite{cs17}. It was then algebraically
presented and proved in \cite{algebraic_stability} (Theorem 4.4). The
categorically focused ``induced matching'' version of the result
appears in \cite{induced_matchings} (Theorem 3.5), which is emphasized
even further in the entirety of \cite{bauer_lesnick2} (particularly
Theorems 1.4, 1.7).

The following is the initial stability result for the block distance.

\begin{theoremnonum}[\cite{botnan_lesnick} Proposition 2.12 and Theorem
  3.3]\label{thm_stab_bl} For \(\mathrm{vect}\)-valued persistence
  modules over \(\zz\) embedded via \(\tilde{E}\) as block \(\uu\)
  persistence modules, \(D_\dbl\) and its induced bottleneck distance
  \(\widehat{D}_\dbl\) satisfy
  \[
    D_\dbl\leq \widehat{D}_\dbl\leq \frac{5}{2}D_\dbl.
  \]
\end{theoremnonum}

As the block distance separates by \(\langle\cdot,\cdot\rangle_\zz\)
type, the result above is proved independently for each of the four
cases. In three of these cases the above statement is tight with the
constant
of \(5/2\). In \cite{bjerkevik} it is shown that for the case of
\((\cdot,\cdot)_\zz\) modules, the block distance and its induced
bottleneck distance are \emph{isometric} (i.e., the \(5/2\) can be
replaced with \(1\)).

These theorems are immensely important results for the topic at hand, but
do not reflect the sort of stability theorem that we will provide for
\(D_\dar\). As it has been defined, \(D_\dar\) is foundationally a
bottleneck distance in the first place, and thus \emph{is} its own
induced bottleneck distance. As such,
any algebraic stability result of the type discussed here would be
trivial for \(D_\dar\). Instead, we examine comparative stability of
the kind \(D_\dar\leq A\cdot D_\dbl\) and \(D_\dbl\leq B\cdot D_\dar\)
over pure zigzag orientations.

The following is our final result for this section: full minimal
Lipschitz constants comparing a modification of \(D_\dar\) to
\(D_\dbl\) over the four kinds of \(\zz\) modules (this echoes the
piecewise stability results of the block distance
\cite{botnan_lesnick}, as this modified \(D_\dar\) also shares
the trait that it ``separates'' modules by
\(\langle\cdot,\cdot\rangle_\zz\) type).

\begin{theoremnonum}[Theorem \ref{thm_blar_limit}]
  The following are the minimal Lipschitz constants comparing
  \(D_\dbl\) with the modification \(D_\dar^{2,\infty}\) of \(D_\dar\)
  over some poset \(P=\mathbb{A}_n(z)\) of pure zigzag orientation.
  \begin{itemize}
  \item If \(\sigma_\zz,\tau_\zz\in(\cdot,\cdot)_\zz\), then \(
      \,\,2D_\dbl\leq D_\dar^{2,\infty}\leq 16D_\dbl.
    \)
  \item If \(\sigma_\zz,\tau_\zz\in[\cdot,\cdot]_\zz\), then \(
      \,\,2D_\dbl\leq D_\dar^{2,\infty}\leq 4D_\dbl\,\,
    \) (if \(D_\dbl < \infty\)).
  \item If \(\sigma_\zz,\tau_\zz\in[\cdot,\cdot)_\zz\), then \(
      \,\,2D_\dbl\leq D_\dar^{2,\infty}\leq 8D_\dbl.
    \)
  \item If \(\sigma_\zz,\tau_\zz\in(\cdot,\cdot]_\zz\), then \(
      \,\,2D_\dbl\leq D_\dar^{2,\infty}\leq 8D_\dbl.
    \)
  \end{itemize}
\end{theoremnonum}

\subsection{Partitioning of Intervals and Modifications of \(D_\dar\)}

Throughout, we compare \(D_\dbl\) with the original \(D_\dar\) and
then two further modifications of it. \(D_\dar^r\) is a modification
of \(D_\dar\) that acts by projecting into a poset refinement of pure
zigzag (called \(r\)-zigzag) in order to avoid a large hull, all while
preserving the structure of the projected modules over sources and
sinks. \(D_\dar^{r,\infty}\) is a further modification that views
original zigzag modules over \(r\)-zigzag posets of unbounded length.
This perspective both compares more favorably with \(D_\dbl\) and may
be of independent interest to anyone who does not wish to be limited
to bounded zigzag posets in the first place.

The remainder of this section chronicles Lipschitz stability between
\(D_\dbl\) and original \(D_\dar\) and the fact that in both
directions the minimal Lipschitz constants involve \(n\) itself (the
length of \(P=\mathbb{A}_n\). The first modification \(D_\dar^r\)
removes one of these dependencies, while the second modification to
\(D_\dar^{r\infty}\) removes the other.

The most persistent discrepancy (the one removed by the
\(D_\dar^{r\infty}\) modification) is discussed in the following
remark.

\begin{remark}[Partitions of \(\Sigma_P\): \(\zz\) vs\(.\) compass]
  \label{rmk_w_table}
  We require a brief discussion of the connection between the subsets
  \(\mathcal{E},\mathcal{W},\mathcal{S},\mathcal{N}\) of \(\Sigma_P\)
  and the subsets
  \((\cdot,\cdot)_\zz,[\cdot,\cdot]_\zz,[\cdot,\cdot)_\zz,(\cdot,\cdot]_\zz\)
  of \(\Sigma_\zz(P)\) under the functor \(\mathcal{Z}\) (Definition
  \ref{defn_convert}). When trying to pair the compass regions
  precisely to the partitions by \(\zz\)-type, the inconvenience
  becomes that the diagonals (of the AR quiver) belong to different
  members of the \(\zz\)-partition \emph{depending on the
  orientation of} \(\mathbb{A}_n\).
  
  Define the sets:
  \begin{itemize}
  \item \(\mathcal{D}_{nw}=
    \{[1,\cdot]\in\Sigma_P:[1,\cdot]\text{ is northwest of
    }[1,n]\}\).

  \item \(\mathcal{D}_{ne}=
    \{[\cdot,n]\in\Sigma_P:[\cdot,n]\text{ is northeast of
    }[1,n]\}\).

  \item \(\mathcal{D}_{se}=
    \{[1,\cdot]\in\Sigma_P:[1,\cdot]\text{ is southeast of
    }[1,n]\}\).

  \item \(\mathcal{D}_{sw}=
    \{[\cdot,n]\in\Sigma_P:[\cdot,n]\text{ is southwest of
    }[1,n]\}\).

  \end{itemize}
  Supplement when necessary with the bar notation from Notation
  \ref{not_ewsn}, i.e.,
  \[\bar{\mathcal{N}}=\mathcal{N}\cup\mathcal{D}_{nw}
    \cup\mathcal{D}_{ne}\cup\{[1,n]\}.\]
  

  See Table \ref{tab_partitions}.

  \begin{table}
    \centering
    {\tabulinesep=1.5mm \begin{tabu}{c|c|c|c|c}
        & \((\cdot,\cdot)_\zz=\)
        & \([\cdot,\cdot]_\zz=\)
        & \([\cdot,\cdot)_\zz=\)
        & \((\cdot,\cdot]_\zz=\)
        \\
        \hline
        \(\mathbb{A}_n^{\mathrm{uu}}(z)\)
        & \(\mathcal{E}\)
        & \(\bar{\mathcal{W}}\)
        & \(\mathcal{S}\cup\mathcal{D}_{se}\)
        & \(\mathcal{N}\cup\mathcal{D}_{ne}\)
        \\
        \hline
        \(\mathbb{A}_n^{\mathrm{ud}}(z)\)
        & \(\mathcal{E}\cup\mathcal{D}_{ne}\)
        & \(\mathcal{W}\cup\mathcal{D}_{nw}\)
        & \(\bar{\mathcal{S}}\)
        & \(\mathcal{N}\)
        \\
        \hline
        \(\mathbb{A}_n^{\mathrm{du}}(z)\)
        & \(\mathcal{E}\cup\mathcal{D}_{se}\)
        & \(\mathcal{W}\cup\mathcal{D}_{sw}\)
        & \(\mathcal{S}\)
        & \(\bar{\mathcal{N}}\)
        \\
        \hline
        \(\mathbb{A}_n^{\mathrm{dd}}(z)\)
        & \(\bar{\mathcal{E}}\)
        & \(\mathcal{W}\)
        & \(\mathcal{S}\cup\mathcal{D}_{sw}\)
        & \(\mathcal{N}\cup\mathcal{D}_{nw}\)
        \\
    \end{tabu}}
  \caption{Equality of partitions by compass regions of the AR quiver
    and by endpoint type in \(\zz\), dependent on orientation of
    \(P\).}
  \label{tab_partitions}
\end{table}

This leads to Lemmas \ref{lemma_x} and \ref{lemma_y}, which introduce
\(n\)-dependence in \(D_\dbl\leq A\cdot D_\dar,D_\dar^r\) Lipschitz
constants. This is resolved at last when comparing with the
modification \(D_\dar^{r,\infty}\), as seen in Proposition
\ref{prop_d_inf}.
\end{remark}

Finally, we introduce a notational convention for use in Tables
\ref{tab_w1} and \ref{tab_d1}.

\begin{notation}\label{not_comp_diff}
  For the remainder of the work on stability, we invoke the following
  notational conventions for the sake of filling out Tables
  \ref{tab_w1} and \ref{tab_d1} with greater readability.

  Let \(\sigma=[x_1,y_1]_\an\) and \(\tau=[x_2,y_2]_\an\). We will
  denote the by the following values various quantities originating in
  Proposition \ref{prop_formula}:
  \begin{itemize}
  \item \(\mathrm{LH}^{\mathrm{diff}}(\sigma,\tau)=|x_1-x_2|\), the left
    hand \emph{support difference} of the modules,
  \item \(\mathrm{RH}^{\mathrm{diff}}(\sigma,\tau)=|y_1-y_2|\), the right
    hand \emph{support difference} of the modules,
  \item \(\mathrm{LH}^{\mathrm{comp}}(\sigma,\tau)=x_1-1+x_2-1\), the left
    hand \emph{support complements} of the modules, also allowing for
    the notation \(\mathrm{LH}^{\mathrm{comp}}(\sigma)=x_1-1\),
  \item \(\mathrm{RH}^{\mathrm{comp}}(\sigma,\tau)=n-y_1+n-y_2\), the right
    hand \emph{support complements} of the modules, also allowing for
    the notation \(\mathrm{RH}^{\mathrm{comp}}(\sigma)=n-y_1\).
  \end{itemize}
\end{notation}

\subsection{Unmodified Stability}

\begin{table}
  \begin{subtable}[t]{\textwidth}
    \centering
    {\tabulinesep=1.5mm \begin{tabu}{c|c|c|c|c}
        & \((i,j)_\zz\)&
        \([i,j]_\zz\)
        & \([i,j)_\zz\)&
        \((i,j]_\zz\)
        \\
        \hline
        \(W_\dbl\) & \((j-i)/4\) & \(\infty\) & \((j-i)/2\) &
        \((j-i)/2\) \\
        \hline
        \(W_\dar\) & \(y-x+1\) & \(y-x+1\) &
        \(\min\left\{\begin{array}{l}x+y-1\\2n-x-y+1\end{array}\right\}\)
        &
        \(\min\left\{\begin{array}{l}x+y-1\\2n-x-y+1\end{array}\right\}\)
        \\
        \hline
        \(W_\dar^r\) & \(r(y-x)+1\) & \(r(y-x)+1\) &
        \(r(y-x)+1,3\) & \(r(y-x)+1,3\) \\
        \hline
        \(W_\dar^{r,\infty}\) & \(r(y-x)+1\) & \(r(y-x)+1\) &
        \(r(y-x)+1,3\) &
        \(r(y-x)+1,3\) \\
      \end{tabu}}
    \caption{Table of \(W\)-values over any poset of pure
      zigzag orientation, partitioned by \(\zz\) interval type. For
      sources of individual formulas see:
      Row 1, Prop \ref{prop_features_bl};
      Row 2, Corollary \ref{cor_escape_east_west} and Corollary
      \ref{cor_1zz_w};
      Row 3, Corollary \ref{cor_rzz_w};
      Row 4, Proposition \ref{prop_d_inf}.}
    \label{tab_w11}
  \end{subtable}
  \begin{subtable}[t]{\textwidth}
    \centering
    {\tabulinesep=1.5mm \begin{tabu}{c|c|c|c|c}
        & \((i,j)_\zz\)&
        \([i,j]_\zz\)
        & \([i,j)_\zz\)&
        \((i,j]_\zz\)
        \\
        \hline
        \(W_\dbl\) & \((\mathrm{dim}+1)/8\) & \(\infty\) &
        \((\mathrm{dim}+1)/4\) & \((\mathrm{dim}+1)/4\) \\
        \hline
        \(W_\dar\) & \(\mathrm{dim}\) & \(\mathrm{dim}\) &
        \(\mathrm{dim}+\min\left\{\begin{array}{l}2\cdot\mathrm{LH}^{\mathrm{comp}}\\
                                        2\cdot\mathrm{RH}^{\mathrm{comp}}\end{array}\right\}\)
        &
        \(\begin{array}{c}\text{same as}\\\text{previous}\\\text{column}\end{array}\)
        \\
        \hline
        \(W_\dar^r\) & \(r\cdot\mathrm{dim}\) & \(r\cdot\mathrm{dim}\) &
        \(r\cdot\mathrm{dim}\) & \(r\cdot\mathrm{dim}\) \\
        \hline
        \(W_\dar^{r,\infty}\) & \(r\cdot\mathrm{dim}\) & \(r\cdot\mathrm{dim}\) &
        \(r\cdot\mathrm{dim}\) & \(r\cdot\mathrm{dim}\) \\
      \end{tabu}}
    \caption{Simplified table of approximate \(W\)-values that
      emphasize major scaling features.}
    \label{tab_w12}
  \end{subtable}
  \caption{In both tables, recall that the difference between
    \(\mathrm{dim}_\zz\) and \(\mathrm{dim}_\an\) (which is the
    difference between \(j-i\) and \(y-x\)) is given by Definition
    \ref{defn_convert}, and is in all cases essentially a factor of
    \(2\) (with \(\mathrm{dim}_\an\) being the larger one).}
  \label{tab_w1}
\end{table}

\begin{prop}[Unmodified Right-Hand Stability]\label{prop_stab1}
  Over pure zigzag orientation,
  \[
    D_\dar\leq 2n\cdot D_\dbl
  \]
  is the minimal Lipschitz constant satisfying the above inequality.
\end{prop}

\begin{proof}
  Necessity is obtained from Example \ref{ex_worse}. A module of the
  form \([x,x+1]_\an\) can have \(W_\dar=n,n-1\), and this corresponds
  to some module of the form \([i,i+1)_\zz\) or \((i,i+1]_\zz\), both
  of which have \(W_\dbl=1/2\).
  Sufficiency follows from Corollary \ref{cor_dar_n}.
\end{proof}

In the other direction, we must address the misalignment issues
brought to attention in Remark \ref{rmk_w_table}.

\begin{lemma}[Partitioning Non-alignment (see Remark
  \ref{rmk_w_table})]\label{lemma_x}
  Let \(P=\mathbb{A}_n(z)\) be a poset of pure zigzag
  orientation. Then if \(D_\dbl<\infty\), \[D_\dbl\leq n/4\cdot
    D_\dar\] where \(n/4\) is a \emph{lower bound} for the Lipschitz
  constant in the inequality above.
\end{lemma}

\begin{proof}
  No matter the orientation of \(P\), one of
  \(\sigma=[1,n]_\an,\tau_1=[2,n]_\an\) is in some
  \((\cdot,\cdot\}_\zz\) and the other is in the associated
  \([\cdot,\cdot\}_\zz\). Similarly, one of
  \(\sigma=[1,n]_\an,\tau_2=[1,n-1]_\an\) is in some
  \(\{\cdot,\cdot)_\zz\) and the other is in the associated
  \(\{\cdot,\cdot]_\zz\). That is to say,
  \(D_\dbl(\sigma,\tau_i)=\max\{W_\dbl(\sigma),W_\dbl(\tau_i)\}\approx
  n/4\) or \(\infty\) (for \(i=1,2\). See Table \ref{tab_w12}).

  However, both pairs have a \(D_\dar\)
  distance of \(1\) (recall that all \(D_\dar\) distances from a
  diagonal to an adjacent region are of the form
  \(\mathrm{LH}^{\mathrm{diff}}+\mathrm{RH}^{\mathrm{diff}}\)).
\end{proof}

We are now prepared to state this stability result.

\begin{prop}[Unmodified Left-Hand Stability]\label{prop_stab2}
  Over pure zigzag orientation, so long as \(D_\dbl<\infty\),
  \[
    D_\dbl\leq n/4\cdot D_\dar
  \]
  where \(n/4\) is the minimal Lipschitz constant satisfying the above
  inequality \emph{across all pairs of indecomposables}.
\end{prop}

\begin{proof}
  Necessity is given by Lemma \ref{lemma_x}. Sufficiency follows
  below.

  Sufficiency follows from Tables \ref{tab_w1} and \ref{tab_d1}, with
  special concern being given to the final column of Table
  \ref{tab_d1}. The most extreme comparison from this column (suppose
  \(\mathrm{uu}\) orientation for ease of notation) are the
  pair of modules \(\sigma_\an=[1,n-1]_\an\) and
  \(\tau_\an=[2,n]_\an\), which correspond to
  \(\sigma_\zz=[i,i+(n-1)/2)_\zz\) and \(\tau_\zz=(i,i+(n-1)/2]_\zz\)
  for some \(i\in\mathbb{Z}\). But though \(n\)-dependent, these only
  require a Lipschitz constant of \(n/8\), and thus \(n/4\) remains
  permissible.
\end{proof}

\begin{table}
  \centering
  {\tabulinesep=1.5mm \begin{tabu}{c|c|c|c|c}
      & \(\begin{array}{l}\sigma,\tau\text{ are of}\\
            \text{same }\zz\text{ type}\\\end{array}\) &
          \(\begin{array}{c}
              \sigma\in\langle\cdot,\cdot\}\\
              \tau\in\,\,\rangle\cdot,\cdot\}
            \end{array}\)
            & \(\begin{array}{c}
                  \sigma\in\{\cdot,\cdot\rangle\\
                  \tau\in\{\cdot,\cdot\langle
                \end{array}\)
                & \(\begin{array}{c}
                      \sigma\in\{\cdot,\cdot\rangle\\
                      \tau\in\,\,\}\cdot,\cdot\langle
                    \end{array}\)
                    \\
        \hline
        \(d_\dbl\) & \(\max\left\{\begin{array}{l}
                                    \mathrm{LH}^{\mathrm{diff}},\\
                                    \mathrm{RH}^{\mathrm{diff}}
                                  \end{array}\right\}\)
                                & \(\max\left\{\begin{array}{l}
                                    W_\dbl(\sigma),\\
                                    W_\dbl(\tau)
                                  \end{array}\right\}\)
                                & \(\max\left\{\begin{array}{l}
                                    W_\dbl(\sigma),\\
                                    W_\dbl(\tau)
                                               \end{array}\right\}\)
                                & \(\max\left\{\begin{array}{l}
                                    W_\dbl(\sigma),\\
                                    W_\dbl(\tau)
                                               \end{array}\right\}\)\\
        \hline
        \multicolumn{2}{c}{} \\[.5em]
        & \(\begin{array}{l}
            \sigma,\tau\in\bar{\mathcal{C}}\text{ where}\\
            \mathcal{C}\in\{\mathcal{E},\mathcal{W},\mathcal{S},\mathcal{N}\}
            \end{array}\)
        &
        \(\begin{array}{c}(\sigma,\tau)\in\text{ one of}\\
            \mathcal{N}\cup\mathcal{D}_{ne}\times\mathcal{W}\cup\mathcal{D}_{sw},\\
            \mathcal{E}\cup\mathcal{D}_{ne}\times\mathcal{S}\cup\mathcal{D}_{sw}\phantom{,}\end{array}\)
        &
        \(\begin{array}{c}(\sigma,\tau)\in\text{ one of}\\
            \mathcal{N}\cup\mathcal{D}_{nw}\times\mathcal{E}\cup\mathcal{D}_{se},\\
            \mathcal{W}\cup\mathcal{D}_{nw}\times\mathcal{S}\cup\mathcal{D}_{se}\phantom{,}\end{array}\)
        &
        \(\begin{array}{c}(\sigma,\tau)\in\text{ one of}\\
            \mathcal{N}\times\mathcal{S},\\
            \mathcal{W}\times\mathcal{E}\phantom{,}\end{array}\)  
        \\
        \hline
        \(d_\dar\) &
        \(\mathrm{LH}^{\mathrm{diff}}
        +\mathrm{RH}^{\mathrm{diff}}\)
        & \(\mathrm{LH}^{\mathrm{comp}}+\mathrm{RH}^{\mathrm{diff}}\)
        &
        \(\mathrm{LH}^{\mathrm{diff}}+\mathrm{RH}^{\mathrm{comp}}\)
        &
        \(\mathrm{LH}^{\mathrm{comp}}+\mathrm{RH}^{\mathrm{comp}}\)
        \\
        \hline
        \(d_\dar^r\) &
        \(r\cdot(\mathrm{LH}^{\mathrm{diff}}+\mathrm{RH}^{\mathrm{diff}})\)
        &
        \(r\cdot(\mathrm{LH}^{\mathrm{comp}}+\mathrm{RH}^{\mathrm{diff}})\)
        &
        \(r\cdot(\mathrm{LH}^{\mathrm{diff}}+\mathrm{RH}^{\mathrm{comp}})\)
        &
        \(r\cdot(\mathrm{LH}^{\mathrm{comp}}+\mathrm{RH}^{\mathrm{comp}})\)
        \\
        \hline
        \(d_\dar^{r,\infty}\)
        &
        \(r\cdot(\mathrm{LH}^{\mathrm{diff}}+\mathrm{RH}^{\mathrm{diff}})\)
        & \(\max\left\{\begin{array}{l}
                                    W_\dar^r(\sigma)\\
                                    W_\dar^r(\tau)
                       \end{array}\right\}\)
        & \(\max\left\{\begin{array}{l}
                                    W_\dar^r(\sigma)\\
                                    W_\dar^r(\tau)
                       \end{array}\right\}\)
        & \(\max\left\{\begin{array}{l}
                                    W_\dar^r(\sigma)\\
                                    W_\dar^r(\tau)
                                  \end{array}\right\}\) \\
  \end{tabu}}
\caption{Table of \(d\)-values over any
  poset of pure zigzag orientation, partitioned by \(\zz\) interval
  type. For sources of individual formulas see: row 1, Prop
  \ref{prop_features_bl}; row 2, Prop \ref{prop_formula} and Notation
  \ref{not_comp_diff}; row 3, Remark \ref{rmk_convert_2} and Example
  \ref{ex_refine}; row 4, Proposition \ref{prop_d_inf}.}
\label{tab_d1}
\end{table}

\subsection{Stability with \(r\)-zigzag}\label{sec_stab_r}

It seems to the authors that the AR distance's tendency to have
\emph{hulls} in pure zigzag orientations such that intervals with small
supports have \(W\)-values at or near the entire diameter of
\(D_\dar\) is undesirable under quite a few perspectives (namely,
for finding Lipschitsz bounds with other more ``well-behaved''
distances). See Example \ref{ex_worse} and its subsequent discussion
Remark \ref{rmk_alt_zigzag} for motivation, from
which we have already seen in Proposition \ref{prop_stab1} that that
any relationship \(D_\dar\leq A\cdot D_\dbl\) requires a constant that
scales with \(n\).

\begin{definition}\label{defn_r_zz}
  Let \(P=\mathbb{A}_n(z)\) be some pure zigzag orientation and
  \(r\in\mathbb{Z}_{\geq 2}\). Define
  \(P^r=\mathbb{A}_n(z,r)\) to be the following poset. Let \(P^r\) have
  sources and sinks collectively labeled \(1_r,2_r,\ldots,
  (n-1)_r,n_r\), alternating from source to sink in the same sequence
  as the vertices \(1,2,\ldots,n-1,n\) of \(P\). For each \(1\leq
  i\leq n-1\), add \(r-1\) vertices between \(i_r\) and \((i+1)_r\)
  such that the segment \([i_r,(i+1)_r]\) is totally ordered.
  \begin{center}
    \includegraphics[scale=.5]{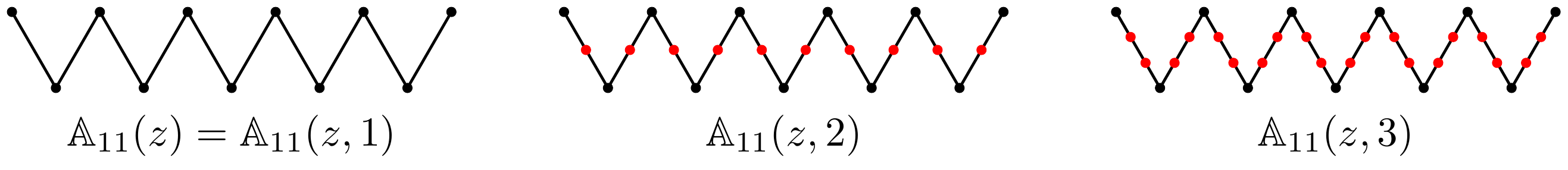}
  \end{center}
  Let \(R\) be the embedding from \(\Sigma_P\to\Sigma_{P^r}\) (the
  collections of isomorphism classes of indecomposable representations
  over each poset) given
  by \(R([x,y])=[x_r,y_r]\). (We note that \(R\) clearly depends on
  the original\(P\) and the choice of \(r\), but we will simply write
  \(R\) in all cases and leave the dependence on \(P,r\) clear by
  context.)

  Finally, define \(D_\dar^r\) on the set of indecomposable
  representations of \(P\) by \[
    D_\dar^r(\sigma,\tau)=D_\dar(R(\sigma),R(\tau)).
  \] where the right hand \(D_\dar\) is the AR distance over
  \(P^r\).
\end{definition}


The endpoint conversion from \(\mathbb{A}_n(z,r)\) intervals to \(\zz\)
intervals is similar to that of Definition \ref{defn_convert}, but has
the labeling disparities increased by a factor of \(R\).

\begin{remark}\label{rmk_convert_2}
  For some module \([x,y]\) over a pure zigzag orientation
  \(\mathbb{A}_n(z)\) and some \(r\in\mathbb{Z}_{>0}\), \[
    \mathrm{dim}([x_r,y_r])=r\cdot[\mathrm{dim}([x,y])-1]+1=r\cdot(y-x)+1.
  \]
\end{remark}

The following result is immediate from Definition \ref{defn_hull}.

\begin{remark}\label{rmk_rz_hull}
  Let \(P=\mathbb{A}_n(z)\) have pure zigzag orientation and \(P^r\) be
  its \(r\)-zigzag refinement. As
  \(\mathrm{hull}(\mathbb{A}_n^r(z))=\{[x_r,(x+1)_r]|1\leq x< n\}\),
  it follows that \(\{R([x,x+1])\}_{1\leq x< n}=\mathrm{hull}(P^r)\).
\end{remark}

\begin{corollary}[to Lemma \ref{lemma_hull_w}]\label{cor_rzz_w}
  For any indecomposable \(\sigma\) over pure zigzag orientation, if
  \(\sigma=[x,x+1]\) then \[
    W_\dar^r(\sigma)=\mathrm{dim}(\sigma)+2=r+3,
  \] and otherwise \[
    W_\dar^r(\sigma)=\mathrm{dim}(\sigma).
  \]
\end{corollary}

\begin{example}\label{ex_refine}
  The following is a visualization of the module embedding \(R\) from
  \(P=\mathbb{A}_6^{\mathrm{du}}\) to its \(3\)-zigzag refinement
  \(P^3\).
  \begin{center}
    \includegraphics[scale=.75]{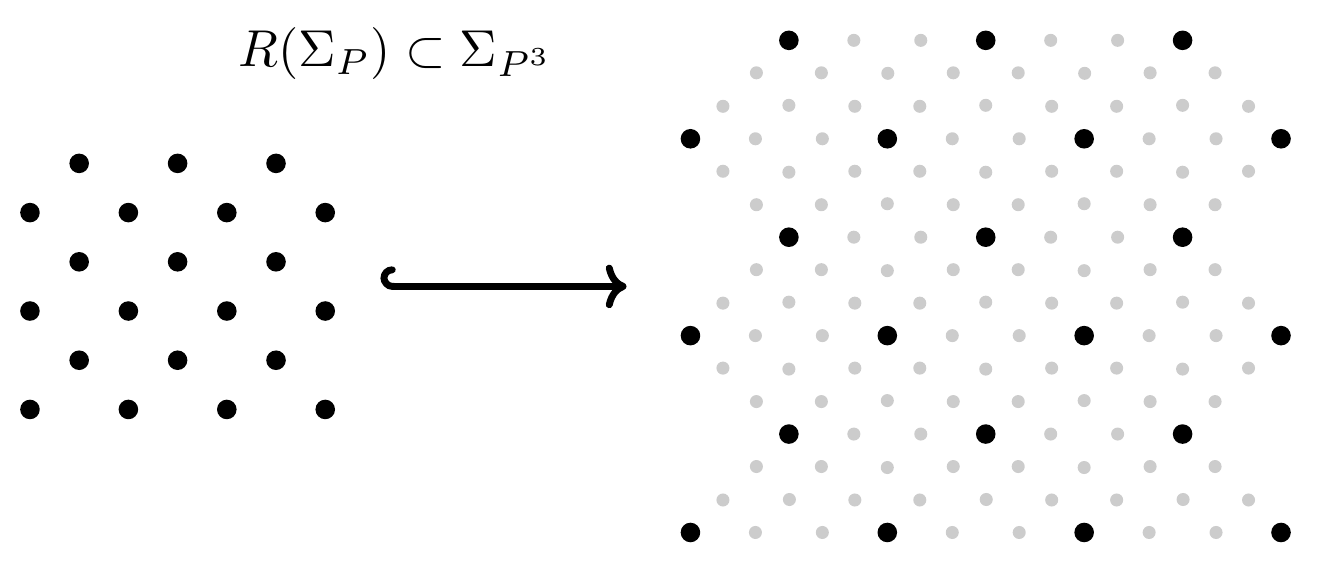}
  \end{center}
  Though unlabeled for clarity, interval modules maintain the same
  relative position across the two AR quivers under
  \(R\). Shape, location, and relative distance between
  indecomposables are essentially unchanged. However, along the north
  and south boundaries, it is immediate that the gray dots in this
  area are simples, removing the presence of pure zigzag's large
  hulls.
\end{example}

\begin{prop}[\(r\)-zigzag Right-Hand Stability]\label{prop_stabr1}
  Over an \(r\)-zigzag orientation \(P=\mathbb{A}_n(z,r)\) with
  \(r\geq 2\), \[
    D_\dar\leq 8r\cdot D_\dbl.
  \]
\end{prop}

Compare with Proposition \ref{prop_stab1} in which the large hull of
unmodified \(W_\dar\) caused \(n\)-dependence in the inequality.

\begin{proof}
  Necessity comes from the first column of Table
  \ref{tab_w12}. Sufficiency of the remaining columns for
  \(W\)-values is easy to check.
  
  First column \(d\)-values in Table \ref{tab_d1} require only a
  constant of \(2r\). We only show the sufficiency of \(8r\) when
  comparing fourth column intervals from Table \ref{tab_d1}.
  
  Suppose then that \(\sigma,\tau\) are two indecomposables with
  opposite parity of both left and right
  endpoints. \(d_\dar^r(\sigma,\tau)\) becomes large (and \(d_\dbl\)
  becomes small) when \(\sigma,\tau \) have small supports and are
  positioned centrally within the poset. However, if the supports are
  too small \(d_\dar^r\) will revert to \(\max W_\dar^r\) values,
  which we already know are stable.

  The largest value of \(d_\dar^r(\sigma,\tau)\) such that
  \(d_\dar^r<\max W_\dar^r\)'s is with \(\sigma\) and \(\tau\) both having
  supports as close as possible to
  \(\mathcal{Z}([n/6,5n/6])=[n_r/6,5n_r/6]\), while still
  possessing opposite parity on left and right endpoints. In
  such a situation, \(d_\dar^r(\sigma_\an,\tau_\an)\approx
  W_\dar^r(\sigma_\an)\approx
  W_\dar^r(\tau_\an)\approx r\cdot(2n/3)\). But then, \(d_\dar^r\approx 2r\cdot
  d_\dbl\), and so \(8r\) remains permissible.
\end{proof}

Considering the opposite inequality, we encounter a repeat of the
partition misalignments.

\begin{lemma}[Partitioning Non-alignment for \(r\)-zigzag]\label{lemma_y}
  Let \(P=\mathbb{A}_n(z)\) be a poset of pure zigzag
  orientation and \(P^r\) be its \(r\)-zigzag extension. Then if
  \(D_\dbl<\infty\), \[D_\dbl\leq \dfrac{n}{4r}\cdot D_\dar^r\] where
  \(n/4r\) is a \emph{lower bound} for the Lipschitz constant in the
  inequality above.
\end{lemma}

\begin{proof}
  The proof follows identically to that of Lemma \ref{lemma_x}, where
  the example modules \(\sigma,\tau_1,\tau_2\) are all viewed through
  the functor \(R_P\).
\end{proof}

\begin{prop}[\(r\)-zigzag Left-Hand Stability]\label{prop_stabr2}
  Over pure zigzag orientation, so long as \(D_\dbl<\infty\),
  \[
    D_\dbl\leq \dfrac{n}{4r}\cdot D_\dar^r
  \]
  where \(n/4r\) is the minimal Lipschitz constant satisfying the
  above inequality.
\end{prop}

\begin{proof}
  Necessity follows from Lemma \ref{lemma_y}.
  
  Sufficiency parallels the proof of Proposition
  \ref{prop_stab2} using Remark \ref{rmk_convert_2}. (In the event
  that it is of interest to the reader, outside of the misalignment
  cases handled by Lemma \ref{lemma_y}, the smaller weight \(n/8r\)
  suffices for all remaining cases. This is a
  further mirroring of the proof of Proposition \ref{prop_stab2}.)
\end{proof}



By projecting from pure zigzag into an \(r\)-zigzag poset and removing
the hull, we have successfully eliminated the \(n\) dependence of one
side of our inequalities. The final modification at last removes the
other.

\subsection{Stability with Poset Limits}\label{subsub_limits}

The following is a further modification of \(D_\dar^r\) that
assumes the representation category of some original
\(P=\mathbb{A}_n(z)\) or \(P^r=\mathbb{A}_n(z,r)\) is embedded into a
poset of similar structure that is lengthened on either end.

There are two advantages to this modification. 1) This modification
obtains stability with \(D_\dbl\) in a way that does not depend on the
original length \(n\) of the poset. 2) This modifies \(D_\dar\) over
pure zigzag orientations (via first modifying to \(D_\dar^r\)) in such
a way that one may consider the modules over a zigzag poset of
\emph{unbounded length}, which may be of independent interest to
many.

\begin{definition}\label{defn_wedge}
  Let \(P=\mathbb{A}_n\) and \(P'=\mathbb{A}_m\) be two orientations
  of \(\mathbb{A}\)-type quivers of any lengths. Assign the labelling
  \(P=\{1\sim 2\sim\ldots\sim n\}\) and \(P'=\{1'\sim2'\sim\ldots\sim
  m'\}\). Then define \[
    P\wedge P'
  \] to be the poset obtained from joining the \(P\)-vertex \(n\)
  with the \(P'\)-vertex \(1'\), along with the original
  \(\leq,\leq'\) relationships and any added inequalities induced by
  the association of \(n\) with \(1'\).
\end{definition}

\begin{definition}
  Let \(P=\mathbb{A}_n(z)\) be some pure zigzag orientation. Let
  \(P^r=\mathbb{A}_n(z,r)\) be its \(r\)-zigzag refinement (Definition
  \ref{defn_r_zz}). For \(f\in\mathbb{Z}_{\geq 1}\), define the poset
  \(P^{r,f}=\mathbb{A}_n(z,r\pm f)\) as follows.

  First define the poset \(U_r=\{1_u\geq \ldots \geq (1+r)_u\leq \ldots \leq
  (1+2r)_u\}\) and \(D_r=\{1_d\leq \ldots \leq (1+r)_d\geq \ldots \geq
  (1+2r)_d\}\) (\(=\) the opposite poset of \(U_r\)). Define
  \(P^{r,1}\) to be
  \begin{itemize}
  \item \(U_r\wedge P^r\wedge U_r\) if
    \(P=\mathbb{A}_n^{\mathrm{uu}}\),
  \item \(U_r\wedge P^r\wedge D_r\) if
    \(P=\mathbb{A}_n^{\mathrm{ud}}\),
  \item \(D_r\wedge P^r\wedge U_r\) if
    \(P=\mathbb{A}_n^{\mathrm{du}}\),
  \item \(D_r\wedge P^r\wedge D_r\) if
    \(P=\mathbb{A}_n^{\mathrm{dd}}\).
  \end{itemize}
  Below is an example of \(P^{3,1}\) for
  \(P=\mathbb{A}_6^{\mathrm{du}}\).
  \begin{center}
    \includegraphics[scale=.5]{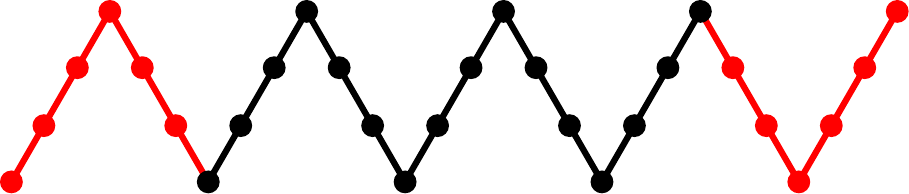}
  \end{center}
  Define \(P^{r,f}\) inductively (i.e., the number of wedges on both
  sides of appropriately chosen \(U_r\) or \(D_r\) is equal to
  \(f\)). In this way, the \(r\)-zigzag structure and sink/source
  orientation of the left and right endpoints remain unchanged from
  \(P^r\) to \(P^{r,f}\).

  Let \(F:\Sigma_{P^r}\to\Sigma_{P^{r,f}}\) be the functor
  \(F([x,y]_{P^r})=[x,y]_{P^{r,f}}\). That is, the supports of
  interval modules remain \emph{fixed} within \(P^r\) considered as a
  subposet of \(P^{r,f}\).
\end{definition}

\begin{definition}
  For \(\sigma,\tau\) over some pure-zigzag orientation
  \(P=\mathbb{A}_n(z)\), define \[
    D_\dar^{r,f}(\sigma,\tau)=D_\dar(F\circ R(\sigma),F\circ R(\tau)).
  \]

  Define \[
    D_\dar^{r,\infty}(\sigma,\tau)=\lim_{f\to\infty}D_\dar^{r,f}(\sigma,\tau).
  \]
\end{definition}


Again, take note that in the following proposition the separation into
pieces of the AR quiver of \(P^r\) when embedded by
\(F_P\) align \emph{precisely} with the \(\langle\cdot,\cdot\rangle\)
partitioning of the AR quiver.

\begin{prop}[\(D_\dar^{r,\infty}\) Separates by \(\zz\)-type]\label{prop_d_inf}
  For \(P=\mathbb{A}_n(z)\),
  \(D_\dar^{r,\infty}\) separates modules by \(\zz\) region.
  That is, the image of the functor
  \(F:\Sigma_{P^r}\to\Sigma_{P^{r,f}}\) consists of the four
  connected components
  \begin{align*}
    &F(\{(\cdot,\cdot)_{\zz,P^r}\})\subset\{(\cdot,\cdot)_{\zz,P^{r,f}}\},\\
    &F(\{[\cdot,\cdot]_{\zz,P^r}\})\subset\{[\cdot,\cdot]_{\zz,P^{r,f}}\},\\
    &F(\{[\cdot,\cdot)_{\zz,P^r}\})\subset\{[\cdot,\cdot)_{\zz,P^{r,f}}\},\\
    &F(\{(\cdot,\cdot]_{\zz,P^r}\})\subset\{(\cdot,\cdot]_{\zz,P^{r,f}}\}.
  \end{align*}

  Moreover, \(D_\dar^{r,\infty}\) is the bottleneck distance given by:
  \begin{itemize}
  \item \(d_\dar^{r,\infty}(\sigma_\an,\tau_\an)=|x_1-x_2|+|y_1-y_2|\)
    if \(\sigma,\tau\) are in the same
    \(\langle\cdot,\cdot\rangle_\zz\) region, and
    \(d_\dar^{r,\infty}(\sigma_\an,\tau_\an)=\infty\) otherwise.
  \item \(W_\dar^{r,\infty}(\sigma)=y_1-x_1+3\) if
    \(\sigma=[x_1,x_1+r]\) where \(x\) is a sink
    or source vertex, and \(W_\dar^{r,\infty}(\sigma)=y_1-x_1+1\)
    otherwise. That is,
    \(W_\dar^{r,\infty}(\sigma)=W_\dar^r(\sigma)\).
  \end{itemize}
\end{prop}

\begin{proof}
  \begin{figure*}[h!]
    \centering
    \includegraphics[scale=.5]{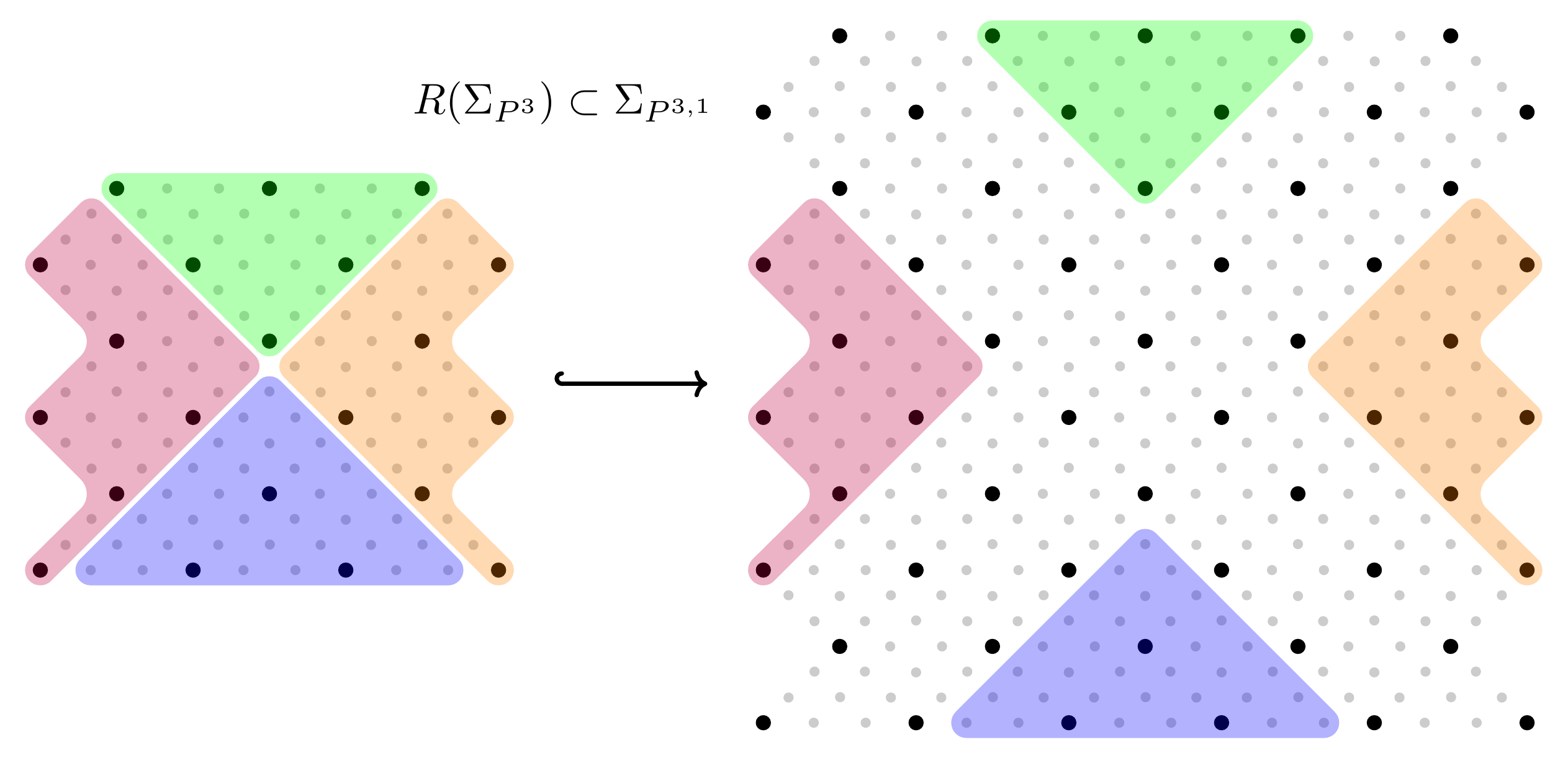}
    \caption{Again, the thicker dots represent
      indecomposables from the AR quiver of
      \(P=\mathbb{A}_6^{\mathrm{du}}\) under the 3-zigzag embedding
      functor \(R\). Depicted here is the embedding \(F\) of modules
      of the AR quiver of \(P^{3}\) into that of the extension by
      \(D_3\) on the left and \(U_3\) on the right.}
    \label{fig_emb_2}
  \end{figure*}
  As we have seen, from \(P=\mathbb{A}_n(z)\) to
  \(P^r=\mathbb{A}_n(z,r)\), the AR
  quiver becomes refined by a factor of \(r\) along both axes while
  the relative positions of the embedded modules from \(P\) remain the
  same (Example \ref{ex_refine}). This separation and the fact that
  \(W_\dar^{r\infty}\) remains completely unchanged from \(W_\dar^r\)
  can be checked individually from the four possible orientations of
  \(P^r\) in Figures \ref{fig_epx} and \ref{fig_epy}. 

  In all four images, when wedging
  with \(U_r\) or \(D_r\), the new axis contains the \(A\)'s in
  sequence, the \(B\)'s in sequence, but separates the two sub-axes by
  the \(C\)'s. Wedges on the \emph{left} side of the
  poset are added to the \emph{middle of the }\(x\)\emph{-axis} and to the
  \emph{ends of the }\(y\)\emph{-axis}. Wedges on the \emph{right}
  side of the poset are added to the \emph{ends of the
  }\(x\)\emph{-axis} and to the \emph{middle of the }\(y\)\emph{-axis}.

  Compare these case by case with the partitions in Table
  \ref{tab_partitions} in Remark \ref{rmk_w_table}.
  \begin{figure*}[h!]
    \centering
    \begin{subfigure}[t]{0.45\textwidth}
      \centering
      \includegraphics[scale=1]{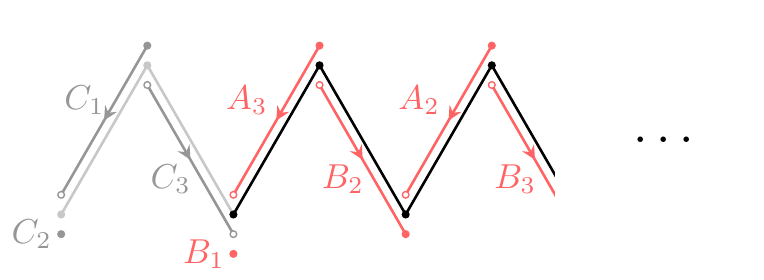}
      \caption{When \(P\) is of \(d*\) orientation, the original
        \(x=1\) (contained in \(B_1\) in the image) is grouped with
        the other \(B_i\)'s, which are all \emph{open} left endpoints.}
    \end{subfigure}%
    \hspace{.5cm}
    \begin{subfigure}[t]{0.45\textwidth}
      \centering
      \includegraphics[scale=1]{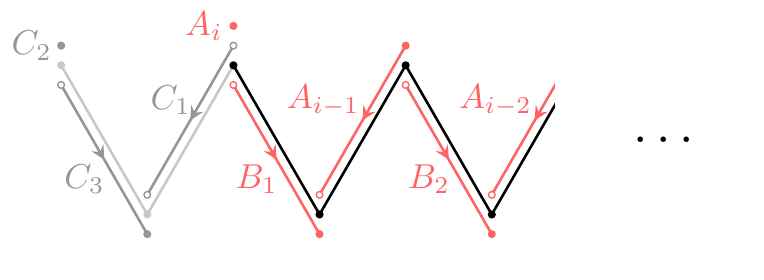}
      \caption{For \(u*\) orientations, the original \(x=1\) is
        grouped with the \emph{closed} endpoints when the original
        axis becomes split by the wedges.}
    \end{subfigure}
    \caption{}
    \label{fig_epx}
  \end{figure*}
  
  \begin{figure*}[h!]
    \centering
    \begin{subfigure}[t]{0.45\textwidth}
      \centering
      \includegraphics[scale=1]{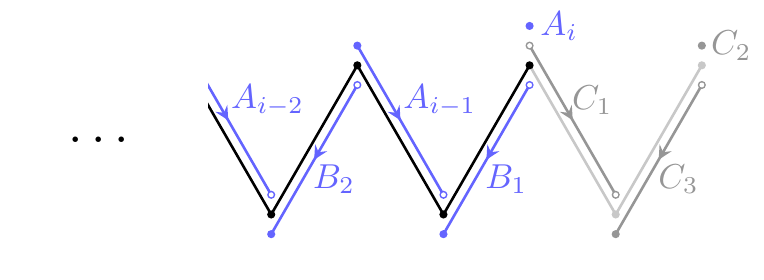}
      \caption{For \(*u\) orientations, the original axis value
        \(y=n\) is a closed right endpoint, and is grouped with the
        other closed endpoints.}
    \end{subfigure}%
    \hspace{.5cm}
    \begin{subfigure}[t]{0.45\textwidth}
      \centering
      \includegraphics[scale=1]{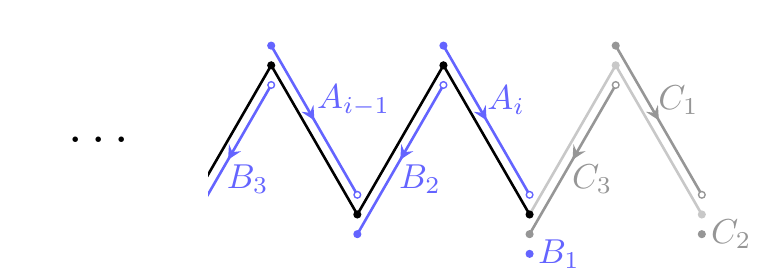}
      \caption{Finally, the original axis value \(y=n\) is open, and
        is grouped in the new axis with the other open endpoints.}
    \end{subfigure}
    \caption{}
    \label{fig_epy}
  \end{figure*}
\end{proof}

\begin{remark}\label{rmk_dar_finite}
  While \(d_\dar^{r,\infty}\) may attain infinite values, the
  final bottleneck distance \(D_\dar^{r,\infty}\) does not, by virtue of
  the fact that \(W_\dar^{r,\infty}=W_\dar^r\) is always bounded by the length
  of the original \(r\)-zigzag orientation (Corollary \ref{cor_dar_n}).
\end{remark}

The following theorem is our concluding result on comparisons
of \(D_\dar\) with \(D_\dbl\).

\begin{theorem}[Sharp \(D_\dar^{r,\infty}\) vs\(.\) \(D_\dbl\)
  Lipschitz Constants]\label{thm_blar_limit}
  Let \(P=\mathbb{A}_n(z)\) be of pure zigzag orientation. The
  following are the four stability results between \(D_\dbl\) and
  \(D_\dar^{r,\infty}\) partitioned by \(\zz\)-type (as neither
  distance directly compares modules from different regions of the
  partition).
  \begin{itemize}
  \item If \(\sigma_\zz,\tau_\zz\in(\cdot,\cdot)_\zz\), then \(
      \,\,r\cdot D_\dbl\leq D_\dar^{r,\infty}\leq 8r\cdot D_\dbl.
    \)
  \item If \(\sigma_\zz,\tau_\zz\in[\cdot,\cdot]_\zz\), then \(
      \,\,r\cdot D_\dbl\leq D_\dar^{r,\infty}\leq 2r\cdot D_\dbl\,\,
    \) (if \(D_\dbl < \infty\)).
  \item If \(\sigma_\zz,\tau_\zz\in[\cdot,\cdot)_\zz\), then \(
      \,\,r\cdot D_\dbl\leq D_\dar^{r,\infty}\leq 4r\cdot D_\dbl.
    \)
  \item If \(\sigma_\zz,\tau_\zz\in(\cdot,\cdot]_\zz\), then \(
      \,\,r\cdot D_\dbl\leq D_\dar^{r,\infty}\leq 4r\cdot D_\dbl.
    \)
  \end{itemize}
\end{theorem}


\begin{proof}
  All left hand inequalities \(r\cdot D_\dbl\leq D_\dar^{r,\infty}\)
  are necessary by the first column of Table \ref{tab_d1}, and
  sufficiency is easy to see by examination of Table \ref{tab_w1}
  (columns two three and four of Table \ref{tab_d1} simply revert
  to problems of comparing values in Table \ref{tab_w1}).

  For the right hand inequalities, a Lipschitz constant of
  \(D_\dar^{r,\infty}\leq 2r\cdot D_\dbl\) is permissible when
  considering only Table \ref{tab_d1}. However, the different
  \(W_\dbl\) behaviors in Table \ref{tab_w1} force some of the values
  to be larger.
\end{proof}

As seen in the initial statement of the proof at the beginning of this
section, one may as well choose the minimal zigzag extension of
\(r=2\) if there is no contextual motivation for selecting a larger
value.

\section{Weighted Interleaving Distance}\label{sec_dwil}

As briefly discussed in the introduction, the weighted interleaving
distance on some orientation of \(\mathbb{A}_n\) measures
similarity between two interval modules by the
depth or shallowness on the `wells' over which their supports differ
(Figure \ref{fig_wil_ex1}).

\begin{definition}
For a general orientation \(P=\mathbb{A}_n\), enumerate the poset's
source vertices from left to right as \(m_1,\ldots,m_p\). Define
\(V_i\) to be the maximal sub-poset given by all elements
comparable to \(m_i\). \[
  V_i=\{x\in P: x\geq m_i\}.
\] Label the left and right sinks of \(V_i\) (if they exist) as
\(1_i\) and \(n_i\) respectively: \[
  \{1_i\gets\ldots\gets m_i\to\ldots\to n_i\}.
\] Let \([V_i]\) denote the interval representation \([1_i,n_i]\).

Lastly, as independent posets, the \emph{wedge} of \(V_i\) and
\(V_{i+1}\) is the poset in which \(n_i\) is identified with
\(1_{i+1}\): \[
  V_i\wedge V_{i+1}=\{1_i\gets\ldots\gets m_i\to\ldots\to
  n_i=1_{i+1}\gets\ldots\gets m_{i+1}\to\ldots\to n_{i+1}\}.
\] as in Definition \ref{defn_wedge}.
\end{definition}

\begin{remark}
  Any orientation \(P=\mathbb{A}_n\) can be uniquely expressed as a
  wedge of \(V_i\)'s \[
    P=V_1\wedge V_2\wedge\ldots\wedge V_l,
  \] where \(V_1\) and \(V_l\) may be equioriented segments.
\end{remark}

Moving forward, we will view representations of an orienation of
\(\mathbb{A}_n\) as persistence modules over a one-vertex refinement
of the original poset.

\begin{definition}
  For a poset \(P\), let \(\tilde{P}\) be the poset \(P\cup\{\infty\}\)
  with the relation \(x\leq y\) if and only if either \(x,y\in
  P\) with \(x\leq_P y\), or \(y=\infty\).

  We call \(\tilde{P}\) \emph{the poset} \(P\) \emph{suspended at
    infinity}.
\end{definition}

Translations on \(P=\mathbb{A}_n\) can be viewed as a wedge of
translations on each individual \(V_i\).

\begin{prop}\label{prop_split}
  Any translation \(\Lambda\) on \(P=V_1\wedge\ldots\wedge V_p\) can
  be fully described by how it acts on the individual
  \(V_i\). Similarly, any collection
  \(\{\Lambda_i\in\mathrm{Trans}(V_i)\}_{1\leq i\leq p}\) determine a
    translation on \(P\).

  Similarly, any translation \(\Lambda\) on \(\tilde{P}\) where
  \(P=V_1\wedge\ldots\wedge V_p\) can
  be fully described by how it acts on the individual \(\tilde{V}_i\).
  However, in reverse we must add the extra condition that
  pairs of translations for adjacent \(V_i\) agree at the points of
  overlap. That is, any collection
  \[
    \{\Lambda_i\in\mathrm{Trans}(\tilde{V}_i)\}_{1\leq i\leq
      p}: \Lambda_i(n_i)=\Lambda_{i+1}(1_{i+1})\text{ for all }1\leq
    i< p
  \]
  determines a translation on \(\tilde{P}\).

\end{prop}

We now define an interleaving-type distance using the poset suspended
at \(\infty\).

\begin{definition}\label{def_wt}
  For a poset \(P\) and a pair \((a,b)\in\mathbb{N}\times\mathbb{N}\),
  define the \emph{weighted height} of a translation \(\Lambda\) over
  \(\tilde{P}\) to be \[
    \tilde{h}(\Lambda)=\max_{x\in P}\delta^{(a,b)}(x,\Lambda x),
  \] where \(\delta^{(a,b)}(x,y)\) is the directed graph distance between
  \(x\) and \(y\), with edges of \(P\) counted with weight \(a\), and
  added edges of \(\tilde{P}\) counted with weight \(b\).
\end{definition}

At a weight of \((1,1)\), this is the directed graph distance induced
by the poset structure. However, as we want to make the movement of
former maximals possible without entirely losing track of the
significance of that operation, we have the ability to feather the
``penalty'' of moving these former maximal to \(\infty\) with the
weight \(b\) (or rather, the weight of \(b\) relative to \(a\)).

\begin{definition}\label{def_wil}
  For a poset \(P\) and a pair \((a,b)\in\mathbb{N}\times\mathbb{N}\),
  we define the \emph{weighted interleaving distance}
  \(D_\dwil^{(a,b)}\) to be the interleaving distance (Definition
  \ref{def_il}) on the set of representations of \(P\), but with
  translations taken over \(\tilde{P}\), using the height function
  given in Definition \ref{def_wt}.

  Throughout, the notation \(D_\dwil^{(a,b)}\) will be reduced to
  \(D_\dwil\).
\end{definition}

We introduce the following notation for future convenience.

\begin{notation}\label{not_t}
  
  For a given \(V_i\), let
  \(T_i=\min\{m_i-1_i,n_i-m_i\}\) be the length of the short side,
  and \(S_i=\max\{m_i-1,n_i-m_i\}\) be the length of the long
  side. Define \(T_i=0\)
  if \(V_i\) is equioriented.

  Define \(T \vcentcolon =\displaystyle\max_{1\leq i\leq p}T_i\).

  Define \(S \vcentcolon =\displaystyle\max_{1\leq i\leq p}S_i\).

\end{notation}

\begin{prop}[Classification of Translations on
  \(\tilde{P}\)]\label{prop_class}
  Let \(P=V_1\wedge V_2\wedge\ldots\wedge V_p\) be an orientation of
  \(\mathbb{A}_n\), and assume \(a\leq b\). Let \(\Lambda\) be a
  translation on \(\tilde{P}\). The collection of full translations
  (Remark \ref{rmk_full}) are described below by how they act on each
  individual  \(\tilde{V}_i\) (Remark \ref{prop_split}).
  \begin{itemize}
  \item If \(\tilde{h}(\Lambda)<a\), then each \(\Lambda_i\) (and
    so \(\Lambda\) itself) is the trivial translation.
  \item If \(a\leq\tilde{h}(\Lambda)<b\), then the sources and sinks
    of each \(V_i\) are fixed by \(\Lambda_i\). All other vertices
    move upwards by
    \(k\) vertices, where \(ak\leq \tilde{h}(\Lambda)< a(k+1)\),
    or to their unique comparable sink, if that is closer than
    \(k\) vertices.
  \item If \(b\leq \tilde{h}(\Lambda)<a(T_i-1)+b\), then \(\Lambda\)
    can be described in the same way as above, save that now the
    sinks are sent to \(\infty\). Also, if \(ak+b\leq
    \tilde{h}(\Lambda)<a(k+1)+b\), then any
    vertices (other than unique source) that are within \(k\) vertices
    from their corresponding sink are also sent to \(\infty\).
  \item If \(a(T_i-1)+b\leq\tilde{h}(\Lambda)\), then the entire shorter
    leg (sans the source) can be sent to \(\infty\) by the
    translation. Each vertex of the longer leg (including the
    source), is sent up the longer side as far as the translation
    permits (including being sent to \(\infty\)).

    Barring extreme differences between the length of the two sides
    combined with small
    values of \(b\),  \(\Lambda^2\) will almost always send every
    vertex of \(V_i\) to \(\infty\).
  \end{itemize}

  To summarize, if \(I\) and \(J\) are two arbitrary persistence
  modules over \(P\):
  \begin{itemize}
  \item \(D_\dwil(I,J)<b\) guarantees that \(I\) and \(J\) are isomorphic
    on every fixed point.
  \item \(D_\dwil(I,J)=ak+b\) for some \(k\geq 0\) guarantees
    that they \(I\) and \(J\) are isomorphic on minimal vertices of any
    \(V_i\) in which \(T_i-1\geq ak\).
  \end{itemize}
\end{prop}

\begin{ex}
  \begin{figure}
    \centering
    \includegraphics[scale=1]{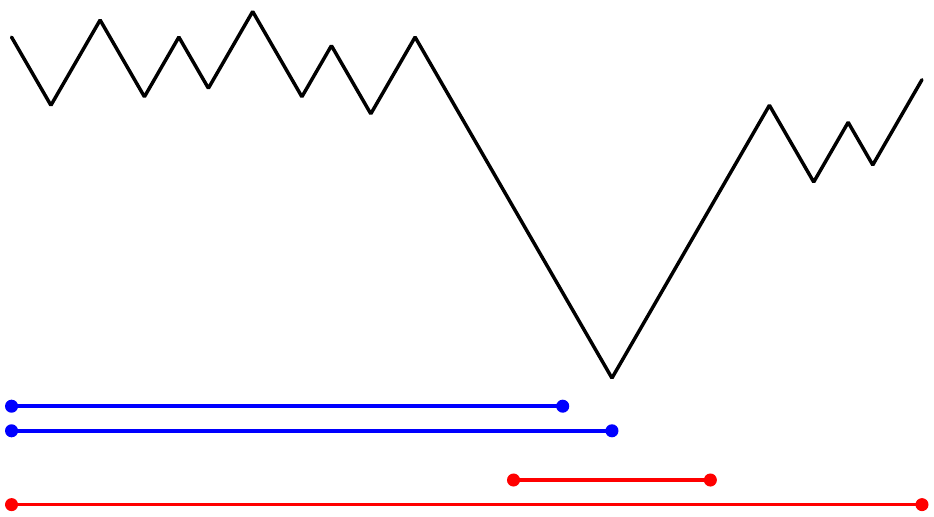}
    \caption{An orientation of \(\mathbb{A}_n\) with two pairs of
      intervals. For \(b>>a\), the red intervals are much closer under
      \(D^{(a,b)}_\dwil\) than the blue intervals.}
    \label{fig_wil_ex1}
  \end{figure}

  With all definitions in place, we can make an easy example to convey
  what \(D_\dwil\) measures, and what it ignores.
  
  In Figure \ref{fig_wil_ex1}, the red intervals are much closer to
  each other in the weighted interleaving distance than the blue
  intervals are.

  In particular, \(D_I(\textrm{red modules})\) requires a translation
  of height sufficient to annihilate all the shallow \(V_i\)'s,
  but not the large one. However, \(D_I(\textrm{blue modules})\)
  immediately requires moving the minimal at the bottom of the deepest
  \(V_i\), already demanding a larger translation than anything involved in
  interleaving the red modules.
\end{ex}

\subsection{Stability of \(D_\dwil\) over \(D_\dar\) as Bottleneck Distances}

\begin{remark}\label{rmk_always_bneck}
  We will compare \(D_\dwil\) and \(D_\dar\) as \emph{bottleneck
    distances}. From here onward, let \(D_\dwil\) denote the bottleneck
  distance induced by the weighted interleaving distance.
  
  The focus of this section is the minimization of weights
  \((a,b)\) (under lexicographic \(\mathbb{N}\times\mathbb{N}\)
  ordering) such that \(D_\dar\leq D_\dwil^{(a,b)}\) (again, as
  bottleneck distances).
\end{remark}

The weighted interleaving distance measures different features than
the other distances in this paper, and was adopted as one of our
directions of investigation due to its ability to preserve an
interleaving-like approach to finite posets that is not immediately
stalled by sink/source vertices, which must remain
\emph{fixed} under the ordinary interleaving distance.

The authors previously proved an algebraic stability result using this
distance for ``branch''-type posets \cite{meehan_meyer_1}. While not
supplying an algebraic stability result for arbitrary \(\mathbb{A}_n\)
quivers, we do compare \(D_\dwil\) (its induced bottleneck distance)
against \(D_\dar\).

Instead of single-variable Lipschitz stability results, we state
\(D_\dwil\) stability against another distance in terms of the
two-parameter weight used to
define it: \((a,b)\in\mathbb{N}\times\mathbb{N}\), where we
consider \(\mathbb{N}\times\mathbb{N}\) to be ordered
lexicographically (Definition \ref{def_wil}). This ordering is to
prioritize first minimizing the weight attached to the original poset
structure, and afterwards the weight that determines distances to
\(\infty\).

\begin{theorem}\label{thm_pairs}
  Let \(P=V_1\wedge\ldots\wedge V_l\) be some orientation of
  \(\mathbb{A}_n\). Let \(T,S\) be as in Notation \ref{not_t}. The
  classification of stable weights for \(D_\dwil\geq D_\dar\) as
  bottleneck distances is:
  \begin{itemize}
  \item for equioriented, see Corollary \ref{cor_equi},
  \item for \(S<3\), see Proposition \ref{prop_s_3},
  \item for \(S\geq 3\) and \(T=1\), see Proposition \ref{prop_s_3},
  \item other non-shallow posets are not addressed (see Definition
    \ref{defn_shallow}), and remaining posets are split by centrality
    (see Definition \ref{defn_central}),
  \item for shallow and central orientations, see Proposition
    \ref{prop_shallow_central},
  \item for shallow and non-central orientations, see Corollary
    \ref{cor_shallow_non_central}.
  \end{itemize}
\end{theorem}

\subsection{Stable values of \(a\)}


\begin{prop}\label{prop_aok}
  Let \(P=V_1\wedge\ldots\wedge V_p\).
  \begin{itemize}
  \item If \(S\geq 3\), the minimal permissible weight is of the form
    \((2,b)\).
  \item If \(S=2\), the minimal permissible weight is of the form
    \((1,b)\) if there is only a single equioriented segment of two
    consecutive edges, and is of the form \((2,b)\) otherwise.
  \item If \(S=1\), the minimal permissible weight is of the form
    \((1,b)\).
  \end{itemize}
\end{prop}

\begin{proof}
  \textbf{Necessity:} For \(S\geq 3\) consider the following diagram.
  \begin{center}
    \includegraphics[scale=1.5]{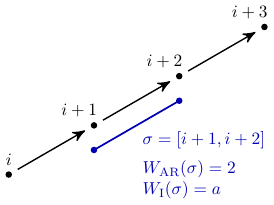}
  \end{center}
  For \(S=2\) consider the following. For any two equioriented
  segments with \(i-2\leftrightarrow i-1\leftrightarrow i\) left of
  \(j\leftrightarrow j+1\leftrightarrow j+2\),
  the intervals \(\sigma=[i-1,j+1]\) and \(\tau=[i,j]\) are
  interleaved by a translation height \(a\) while having an AR
  distance of \(2\). Two of the four possible configurations appear
  below, the remaining two being those with segments of
  \(\nearrow\nearrow\) and \(\searrow\searrow\) arrangements.
  \begin{center}
    \includegraphics[scale=1.5]{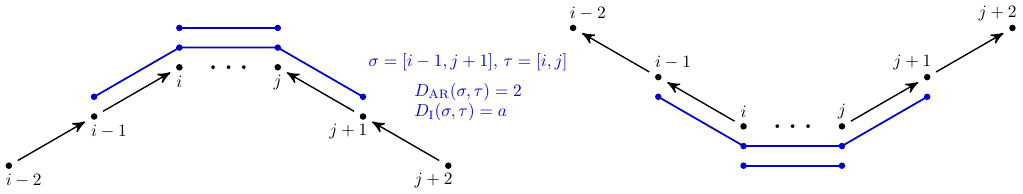}
  \end{center}  
  That is, the only \(S=2\) type of poset permitting \(a=1\) is pure
  zigzag with a single pair of consecutive edges with the same
  orientation.

  \textbf{Sufficiency:} One need only consider \(W\)-values of
  interval modules containing no maximals or minimals, and
  \(d\)-values of pairs of interval modules whose supports share
  precisely the same fixed points. This is easy to check in all
  cases.

\end{proof}

\begin{corollary}\label{cor_aok}
  For any poset \(P\) with \(S\geq 3\) and weight \((2,b)\) with
  \(b>2\), if \(D_\dwil(\sigma,\tau)<b\) then \[
    D_\dar(\sigma,\tau)\leq D_\dwil(\sigma,\tau).
  \]
\end{corollary}

\begin{corollary}
  For any orientation \(Q\) of \(\mathbb{A}_n\) and the appropriate
  choice of \(a=1,2\), the weight \((a,n)\) is \emph{always}
  permissible. I.e., \(b=n\) is an upper bound for the value of \(b\)
  in the minimal permissible weight.
\end{corollary}

\begin{proof}
  This follows immediately from Propositions \ref{prop_aok} and
  \ref{prop_diam}.
\end{proof}

Lastly, a much less general corollary is the resulting minimal
stable weights for equioriented \(\mathbb{A}_n\).

\begin{corollary}\label{cor_equi}
  Let \(Q\) be an equi-orientation of \(\mathbb{A}_n\).
  Then \((2,1)\) is the minimal stable weight.
\end{corollary}

\subsection{Stability when \(S<3\)}

\begin{prop}\label{prop_s_3}
  If \(S<3\), then stability is minimally obtained by \((a,b)\) where
  \(a=1\) or \(a=2\) according to Proposition \ref{prop_aok}, and
  \(b=n\).
\end{prop}

\begin{proof}
  \textbf{Necessity:} \(W_\dwil([1,n])=b\) and
  \(W_\dar([1,n])=n\). \textbf{Sufficiency:} Due to Proposition
  \ref{prop_aok}, one needs only check \(W,d\)-values involving
  modules of the form \([x,n]\).
\end{proof}

\subsection{Short and Long Escape}

As translations split over wedges (Proposition \ref{prop_split}), we now
examine the translations required for realizing \(W_\dar([V_i])\) of
any wedged component, where \(P=V_1\wedge\ldots\wedge V_l\).

\begin{definition}\label{defn_lambda}
  Let \(V=\{1\geq\ldots\geq m\leq\ldots\leq n\}\). Assume that
  the left side is strictly shorter than the right. That is, according
  to Notation \ref{not_t}, \[
    m-1=T<S=n-m.
  \] We first construct the most basic translation \(\Lambda\) such
  that \(\phi,\psi=0\) form a \(\Lambda\)-interleaving of \([1,n]\)
  and \(0\). This is the translation given by:
  \begin{itemize}
  \item \(\Lambda x = \infty\) for all \(x\) in \([1,m)\). Note that
    the distance in the weighted poset from \(m-1\) to \(\infty\) is
    \begin{equation}\label{lambda1}
      \epsilon(b)=2(T-1)+b.
    \end{equation}
  \item For all \(x\in[m,n]\), \(\Lambda x\) moves up the right hand
    side (possibly to \(\infty\)) by a distance of
    \(\mathcal{E}(b)\), where
    \begin{equation}\label{lambda2}
      \mathcal{E}(b)=\left\{
        \begin{array}{ll}
          b & \text{if }b \geq 2S \\
          \\
          1/2(2S+b) & \text{if }2S+b\equiv 0\,\mathrm{mod}\,4 \\
          \\
          1/2(2S+b)+1 & \text{if }2S+b\equiv 2\,\mathrm{mod}\,4 \\
          \\
          \lceil1/2(2S+b)\rceil & \text{if }2S+b\equiv 1,3\,\mathrm{mod}\,4 \\
        \end{array}\right.
    \end{equation}
  \end{itemize}
  In short, by the translation property that \(x\leq y\) demands
  \(\Lambda x\leq \Lambda y\), moving the minimal up one side requires
  that the entire other side by sent to \(\infty\) by
  \(\Lambda\). However, the side
  up which the minimal is moved is relaxed, and may take \emph{two}
  \(\Lambda\)-applications in order to send all vertices to
  \(\infty\), by the properties of interleavings.
    
  Define \(\epsilon(b)\) to be the \emph{short escape}, and
  \(\mathcal{E}(b)\) to be the \emph{long escape}. This
  construction is of minimal height, being
  \(h(\Lambda)=\max\{\epsilon(b),\mathcal{E}(b)\}\), such that
  \(\Lambda^2(x)=\infty\) for any \(x\in V\).

  Replace the prototype translation and define \(\Lambda_{V}^b\)
  to be \emph{the maximal translation on \(V\) of height}
  \(\max\{\epsilon(b),\mathcal{E}(b)\}\). This translation is unique
  (unless it is a symmetric \(V\), in which case choose the left side
  be considered the `short' side).

\end{definition}

\begin{prop}\label{prop_kill}
  The translation \(\Lambda_V^b\) is of minimal height such that
  \(\phi,\psi=0\) form a \(\Lambda_V^b\)-interleaving of \([V]\) and
  \(0\).

  I.e., \(\Lambda_V^b\) realizes \(W_\dwil([V])\) and no translation
  of smaller height does.
\end{prop}

Proposition \ref{prop_kill} pairs extremely well with the
following. (Recall that we are now considering \(D_I\) to always refer
to its induced bottleneck distance as per Remark
\ref{rmk_always_bneck}.)

\begin{corollary}\label{cor_wedges}[to Proposition \ref{prop_split}]
  Let \(P=V_1\wedge\ldots\wedge V_p\). The induced bottleneck distance
  \(D_\dwil\) (by slight abuse of notation) and its generating
  functions \(W,d\) all split over wedges.
  \begin{itemize}
  \item \(
    W_\dwil(I)=\displaystyle\max_{1\leq l\leq p}
    \{W_\dwil(I|_{V_l})\},
    \)
  \item \(
    d_\dwil(I,J)=\displaystyle\max_{1\leq l\leq p}
    \{d_\dwil(I|_{V_l},J|_{V_l})\}.
    \)
  \item \(
    D_\dwil(I,J)=\displaystyle\max_{1\leq l\leq p}
    \{D_\dwil(I|_{V_l},J|_{V_l})\},
    \)
  \end{itemize}
\end{corollary}

\subsection{Shallow Posets}

\begin{ex}
  Using Proposition \ref{prop_kill} and Corollary \ref{cor_wedges},
  let us examine a powerful constraint for stability: \(W\)-values for
  the indecomposable \([1,n]\).

  If we solve simultaneously for the conditions that (a) the largest
  \emph{long escape} exceeds the largest \emph{short escape} (i.e.,
  \(W_\dwil([1,n])\) is determined by some long escape) and (b)
  stability of the form \(D_\dar\leq D_\dwil\), we get the two bounds
  \[
    b\geq 2n-2S\text{ and }b<2S-4T+2.
  \] Combining inequalities, we see that such a \(b\) can only exist if
  (even with some permissive rounding), \[
    2(S-T)+1>n.
  \]
\end{ex}

\begin{remark}
  One immediately sees from the equation above that the
  situation in which \(W_\dwil([1,n])\) is determined by some long
  escape value is incredibly specific, as it requires at the very
  least that the poset have one \(V_i\) with longer side constituting
  \emph{more than half of the entire poset} (using \(T\geq 2\)): \[
    2S > n+1.
  \]
\end{remark}

As long escape dictates \(W_\dwil([1,n])\) only in this extreme case,
we henceforward will only consider the complementary situation.

\begin{definition}\label{defn_shallow}
  An orientation of \(\mathbb{A}_n\) written \(P=V_1\wedge\ldots\wedge
  V_l\) that has \(S\geq 3\) (Proposition \ref{prop_aok} above) is
  \emph{shallow} if \(T\geq 2\) (to keep the hull small) and \(2S\leq
  n\) (to ensure all \(W_\dwil\)'s are determined by short escape).
\end{definition}

\begin{remark}
  Indeed, in a shallow poset \emph{short escape} values are used for
  \emph{any} \(W_\dwil([V_l])\) (and so, by Corollary \ref{cor_wedges},
  all \(W_\dwil\)-values). To see that
  \(W_\dwil([V_l])=\epsilon_l(b)\) for all \(1\leq l\leq p\), simply
  note that \(\mathcal{E}_l(b)=b<2(T_l-1)+b=\epsilon_l(b)\).
\end{remark}



With only this, we can immediately get the stability statement for
\(W\)-values out of the way.

\begin{prop}\label{prop_shallow_w}
  For a shallow poset and any weight
  \((2,b)\) with \(b\geq n-T\), \[
    W_\dar(\sigma)\leq W_\dwil(\sigma)
  \] for any indecomposable \(\sigma\).
\end{prop}

\begin{proof}
  If \(\mathrm{supp}(\sigma)\) contains no sink or source we are done
  by Corollary \ref{cor_aok}.
  
  If \(\sigma\in\mathrm{Hull}(Q)\), then by the \(T\geq 2\) tenet for
  \emph{shallow}, for any \([x,y]\in\mathrm{Hull}(Q)\) either
  \([x,y]\subset(1,m_t]\) or \([x,y]\subset[m_t,n)\). In particular,
  the corresponding \([e]\) and \([E]\) of Lemma \ref{lemma_hull_w}
  obey \(e< E\leq m_t+1\) or \(m_t-1\leq e\leq E\). The formulas of
  Lemma \ref{lemma_hull_w} are all \(\leq n-T\leq b\leq
  W_\dwil([x,y])\). (It is possible for one equation to reach
  \(n-T+1\), but in this case \(m_t\in[x,y]\) and \(W_\dwil([x,y])\geq
  2+n-T\) as \(T\geq 2\)).

  If \(\sigma\not\in\mathrm{Hull}(Q)\) then
  \(W_\dar(\sigma)=\mathrm{dim}(\sigma)\). If
  \(m_t\not\in\mathrm{supp}(\sigma)\), then either \([1_t,m_t]\) or
  \([m_t,n_t]\) are disjoint from \(\mathrm{supp}(\sigma)\) (each of
  which has length at least \(T\)),
  guaranteeing that the dimension \(\mathrm{dim}(\sigma)\leq n-T\).
  Otherwise, \(m_t\in\mathrm{supp}(\sigma)\), and
  \(W_\dwil(\sigma)=W_\dwil([1,n])=2(T-1)+b\geq n+T-2 \geq n
  =\mathrm{diam}(W_\dar)\).

\end{proof}

\subsection{Stability for Shallow and Central}

\begin{lemma}\label{lemma_1_2_3}
  If any of the following are true about a pair of intervals
  \(\sigma,\tau\) over the shallow poset
  \(P=V_1\wedge\ldots\wedge V_p\), then any weight \((a,b)\) with
  \(a=2\) and \(b\geq n-T\) is stable.

  \begin{enumerate}
  \item \(m_t\in\) one of
    \(\mathrm{supp}(\sigma),\mathrm{supp}(\tau)\), but not the other.
  \item \(\mathrm{dim}(\sigma)\leq b\),
  \item \([V_t]\subset\mathrm{supp}(\sigma)\).
  \end{enumerate}

  Throughout, assume the intervals are always labeled such that
  \(W_\dar(\sigma)\geq W_\dar(\tau)\).
\end{lemma}

\begin{proof}[Proof of Lemma \ref{lemma_1_2_3}]
  
  \textbf{(1)} Any interleaving translation must move \(m_t\), and so
  has height \(D_\dwil(\sigma,\tau)\geq 2(T-1)+b\geq 2T-2+n-T=n+T-2\geq
  n=\mathrm{diam}(W_\dar)=\mathrm{diam}(D_\dar)\).

  \textbf{(2)} In the proof of (1) we saw that \(d_\dar(\sigma)\leq
  n-T\leq b\) when \(\sigma\in\mathrm{Hull}\). So for
  any \(\sigma\), if \(\mathrm{dim}(\sigma)\leq b\), then
  \(W_\dar(\sigma)\leq b\).
  But then
  \(D_\dar(\sigma,\tau)=\min\{d_\dar(\sigma,\tau),W_\dar(\sigma)\}\)
  (by the running assumption of \(W_\dar(\sigma)\geq W_\dar(\tau)\)),
  and so \(D_\dar(\sigma,\tau)\leq b\). By Corollary \ref{cor_aok},
  the pair is stable.

  \textbf{(3)} By (2), we may assume that \(\sigma=[x,y]\) where
  \(y-x\geq b\geq n-T\). As \(1+T\leq m_t\leq n-T\), it is immediate
  that \(m_t\in\mathrm{supp}(\sigma)\). Hence, by (1),
  \(m_t\in\mathrm{supp}(\tau)\) also.

  Assume now that, in addition, all of
  \([V_t]\subset\mathrm{supp}(\sigma)\). We will show by cases on the
  equation for \(d_\dar\) that this must also yield stability. First
  note the following inequalities generated by the interleaving
  condition: as \([V_t]\subset\mathrm{supp}(\sigma)\), the
  endpoints of \(\tau=[x_2,y_2]\) are restricted by
  \[
    x_2\leq 1_t+1+\dfrac{D_I-b}{2}
  \]
  \[
    y_2\geq n_t-1-\dfrac{D_I-b}{2}
  \]
  where \(D_I:=D_\dwil(\sigma,\tau)\).

  Stability can now be checked across all possible cases of
  \(\delta^x,\delta^y\). We show only one of them here.
  \begin{align*}
    D_\dar(\sigma,\tau)\leq d_\dar(\sigma,\tau)
    &=|x_1-x_2|+|y_1-y_2|\\
    &\leq 1_t+1+\dfrac{D_I-b}{2}-1+n-\left(n_t-1-\dfrac{D_I-b}{2}\right)\\
    &\leq D_I-(n-T)+n-(n_t-1_t)+1\\
    &\leq D_I+T-(S+T)+1\\
    &\leq D_I
  \end{align*}
\end{proof}


Recall the meanings of \(T,S\) from Notation \ref{not_t}.

\begin{definition}\label{defn_central}
  We say a poset \(P=V_1\wedge\ldots\wedge V_p\) is
  \emph{central} if there is some \(V_t\) with \(T_t=T\)
  positioned in such a way that \[
    [V_t]\subset[T,n-T+1].
  \]  
\end{definition}

\begin{prop}
  A shallow poset is central if and only if every pair of intervals
  fulfill at least one of the conditions of Lemma \ref{lemma_1_2_3}.
\end{prop}

\begin{corollary}\label{prop_shallow_central}
  If \(P\) is a shallow and central poset, then every pair of
  indecomposable modules \(\sigma,\tau\) satisfies the inequality \[
    D_\dar(\sigma,\tau)\leq D_\dwil(\sigma,\tau)
  \] for any weight \((2,b)\) with \(b\geq n-T\).
\end{corollary}

\subsection{Stability for Shallow and non-Central}


\begin{prop}\label{prop_pre_non_central}
  Suppose  \(P=\mathbb{A}_n\) is a shallow and \emph{non-central}
  poset: suppose without loss of generality that \(1_t<T\). Consider a
  weight \((2,b)\) with \(b\geq n-T\).

  Any pair of indecomposables \(\sigma=[x_1,y_1],\tau=[x_2,y_2]\) is
  stable under this weight unless \[
    x_1,x_2\in(1_t,m_t]\text{ and }\delta^y=n-y_1+n-y_2.
  \]
\end{prop}


Proposition \ref{prop_pre_non_central} follows from the subsequent
lemma.

\begin{lemma}\label{lemma_stable_intervals}
  If any of the following are true about a pair of intervals
  \(\sigma,\tau\) over a shallow poset
  \(P=V_1\wedge\ldots\wedge V_p\), then any weight \((a,b)\) with
  \(a=2\) and \(b\geq n-T\) is stable.

  Throughout, assume the intervals are always labeled such that
  \(W_\dar(\sigma)\geq W_\dar(\tau)\).

  \begin{enumerate}

  \item \(\sigma,\tau\) are in the same region of the AR quiver.
    
  \item \(\sigma,\tau\) are in opposite regions of the AR quiver
    (a north-south or east-west pair).

  \item \(1_t\not\in\mathrm{supp}(\sigma)\) and \(x_2\leq 1_t\)
    (symmetrically, \(n_t\not\in\mathrm{supp}(\sigma)\) and
    \(x_2\geq n_t\)).

  \end{enumerate}
\end{lemma}

\begin{proof}

  \textbf{(1)} From Lemma \ref{lemma_1_2_3} we may assume
  \([V_t]\not\subset\mathrm{supp}(\sigma)\). As
  \(m_t\in\mathrm{supp}(\sigma)\), it follows that either
  \(1_t\) or \(n_t\) is in \(\mathrm{supp}(\sigma)\). Suppose then,
  without loss of generality, that
  \(1_t\not\in\mathrm{supp}(\sigma)\): that is, \(x_1\in(1_t,m_t]\).

  We may assume that \(m_t\in\mathrm{supp}(\tau)\). If \(x_2\leq
  1_t\), then the bound on \(|x_1-x_2|+|y_1-y_2|\) proceeds
  identically to the similar equation in the proof of Lemma
  \ref{lemma_1_2_3} (3). Otherwise,
  \(x_2\in(1_t,m_t]\). Then,
  
  \begin{align*}
    D_\dar(\sigma,\tau)\leq d_\dar(\sigma,\tau)
    &=|x_1-x_2|+|y_1-y_2|\\
    &\leq m_t-1_t+n-\left(n_t-1-\dfrac{D_I-b}{2}\right)\\
    &< D_I-(n-T)+n-(n_t-m_t)-1_t+1\\
    &\leq D_I
  \end{align*}

  \textbf{(2)} Again by Lemma \ref{lemma_1_2_3}, assume without
  loss of generality that
  \(x_1\in\mathrm(1_t,m_t]\). Then it must be that \(x_2\leq 1_t\) in
  order to have \(\delta^x=x_1-1+x_2-1\). But in such a situation, the
  bound on \(x_1-1+x_2-1+n-y_1+n-y_2\) proceeds identically
  to the similar equation in the proof of Lemma \ref{lemma_1_2_3} (3).

  \textbf{(3)} Using Lemma \ref{lemma_1_2_3} and this lemma's (1)
  and (2), we may assume without loss of
  generality that \(1_t\not\in\mathrm{supp}(\sigma)\), and either
  \begin{itemize}
  \item \(d_\dar(\sigma,\tau)=x_1-1+x_2-1+|y_1-y_2|\) or
  \item \(d_\dar(\sigma,\tau)=|x_1-x_2|+n-y_1+n-y_2\).
  \end{itemize}
  However, given the assumption
  \(1_t\not\in\mathrm{supp}(\sigma)\), the first equation above also
  yields stability. If \(d_\dar(\sigma,\tau)\) is the first equation,
  then \(x_2\leq 1_t\), and so:
  \begin{align*}
    D_\dar(\sigma,\tau)\leq d_\dar(\sigma,\tau)
    &=x_1-1+x_2-1+|y_1-y_2|\\
    &\leq 1_t+1+\dfrac{D_I-b}{2}-1+1_t-1+n-\left(n_t-1-\dfrac{D_I-b}{2}\right)\\
    &\leq D_I-b+n+2\cdot1_t-n_t\\
    &\leq D_I+1_t+T-(n_t-1_t)\\
    &<D_I+2T-(S+T)
  \end{align*}

  Assume the second equation, and assume that \(x_2\leq
  1_t\). However, one can immediately see from the bound on the
  similar eqation in the proof of Lemma \ref{lemma_1_2_3} (3) that
  this assumption results in stability as well.
\end{proof}

This result allows us to narrow down a \emph{maximally anti-stable}
candidate pair for any shallow non-central poset.

\subsection{Maximally Anti-Stable Pairs}

The structure of this section is as follows.

Suppose \(P\) is a shallow but non-central orientation of
\(\mathbb{A}_n\). Without loss of generality suppose that
\(1_t<T\). We have already shown by Lemmas \ref{lemma_1_2_3} (3) and
\ref{lemma_stable_intervals} that any anti-stable pair
\(\sigma=[x_1,y_1]\), \(\tau=[x_2,y_2]\) has the property that
\(x_1,x_2\in(1_t,m_t]\) and \(y_1,y_2\geq m_t\) are of opposite
orientation from each other.

This means that \(\delta_\dar(\sigma,\tau)=|x_1-x_2|+n-y_1+n-y_2\) for
any anti-stable pair. We measure anti-stability by the size of the
difference \(D_\dar-D_I\), and show that starting from \emph{any}
anti-stable pair, we can reduce down to one of two canonical
anti-stable pairs that between them maximize anti-stability.

First, choosing \(x_1,x_2\) as far apart as possible increases
\(D_\dar\) while having no effect on \(D_I\). But \(y_1\) has a lower
bound dependent on \(x_1\)'s position (while \(y_2\) does not depend
on \(x_2\)), so to maximize later freedom we choose \(x_1=1_t+1\) and
\(x_2=m_t\).

Then, \(y_2\) has two \(d_\dar\)-minimizing possibilities based on the
orientation of \(y_1\). Lastly, \(y_1\) can be shifted left to further
minimize \(d_\dar\). This leftward shifting of \(y_1\) potentially
alters the interleaving distance between \(\sigma\) and \(\tau\), but
as long as \(y_1\) is chosen such that \(\mathrm{dim}(\sigma)>b\)
[Lemma \ref{lemma_1_2_3} (2)] it causes a strict increase in
anti-stability of the pair.

\begin{definition}
  For any \(y>n_t\), define \(k(y)=\displaystyle\max_{t<j\leq
    i}\{T_j\}\) where \(y\in[m_i,m_{i+1})\).
\end{definition}

For and vertex \(y\) right of \(V_t\), the value \(k(y)\) returns the
length of the longest shortest edge of the \(V_i\)'s contained between
\(V_t\) and \(y\). This value determines the interleaving distance
between two modules containing \(m_i\), one of whose right endpoints
is \(y\), and the other of which is contained between \(m_i\) and
\(m_{i+1}\).

As \(D_\dar(\sigma,\tau)\leq W_\dar(\sigma)\), if \(W_\dar(\sigma)\leq
D_\dwil(\sigma,\tau)\) then we are done. It suffices to assume
throughout that \(W_\dar(\sigma)>D_\dwil(\sigma,\tau)\), and to then
show that \(d_\dar(\sigma,\tau)\leq D_\dwil\). The assumption
\(W_\dar(\sigma)>D_\dwil(\sigma,\tau)\) amounts to the inequality \[
  y_1-x_1+1>2(k(y_1)-1)+b.
\] This is clear from Lemma \ref{lemma_stable_intervals} plus the
foreknowledge that we will be adjusting all other vertices such that
the defining feature of \(D_\dwil(\sigma,\tau)\) will be \(W_\dwil\)
of the \(V_p\)'s between \(n_t\) and \(y_1\), as these are in the
support of \(\sigma\) and outside the support of
\(\tau\).

More conveniently, we will replace \(x_1=1_t+1\) and write the above
inequality as \[
  y_1>2k(y_1)-2+b+1_t.
\]

\begin{definition}\label{defn_k_stuff}  
  For a weight \((2,b)\) and vertex \(y>n_t\), consider the statement
  \[
    \Theta(y): y>2k(y)-2+b+1_t.
  \]
  Define \[y_u(b)=\min\{y:\Theta(y)\text{ holds and }y\text{ is upward
      oriented}\}\] and \[y_d(b)=\min\{y:\Theta(y)\text{ holds and }y\text{ 
      is downward oriented}\}\]
  where we will simply write \(y_u\) and \(y_d\) when context makes
  clear the value of \(b\).

\end{definition}

\begin{corollary}\label{cor_min_pair}
  If there is any pair that violates stability for the weight
  \((2,n-T)\), then at least one of the
  pairs \[(\sigma_u=[1_t+1,y_u],\tau_u=[m_t,n_t])\text{
  or }(\sigma_d=[1_t+1,y_d],\tau_d=[m_t,n_t-k(y_d)])\]
  also violates stability for that weight and is maximally anti-stable
  out of all pairs of intervals over the poset (that is, the
  value of \(R=D_\dar-D_\dwil\) is positive and maximal for the correct
  pair).
\end{corollary}


In the event that there is any anti-stable pair for the poset, call
the pair above with the greater anti-stability the \emph{maximal
anti-stable pair for the poset}. If both pairs are just as
anti-stable, choose \((\sigma_u,\tau_u)\).

\begin{proof}\label{proof_prop_min_pair}
  This follows from Propositions \ref{prop_anti1} and
  \ref{prop_anti2}.
\end{proof}

\subsection{Maximally Anti-stable Pairs}

Let \(P\) be a shallow and non-central orientation of
\(\mathbb{A}_n\).

Suppose there exists a pair
\(\hat{\hat{\sigma}}=[x_1',y_1'],\hat{\hat{\tau}}=[x_2',y_2']\) with
\(W_\dar(\hat{\hat{\sigma}})\geq W_\dar(\hat{\hat{\tau}})\) such
that \((\hat{\hat{\sigma}},\hat{\hat{\tau}})\) is an anti-stable pair
for any weight \((2,b)\) with \(b\geq n-T\).

\begin{prop}\label{prop_anti1}
  If \((\hat{\hat{\sigma}},\hat{\hat{\tau}})\) is an anti-stable pair,
  then  \(\hat{\sigma}=[1_t+1,y_1'],\hat{\tau}=[m_t,y_2']\) also
  comprise an anti-stable pair. Furthermore,
  \(R(\hat{\sigma},\hat{\tau})\geq
  R(\hat{\hat{\sigma}},\hat{\hat{\tau}})\) and
  \(W_\dar(\hat{\sigma})\geq W_\dar(\hat{\tau})\).
\end{prop}

\begin{proof}
  It is immediate that this choice of \(x_1,x_2\) maximize the value
  of \(\delta_\dar(\sigma,\tau)\). The opposite assignment would do
  the same, however, \(y_1\) (which maximizes \(\delta_\dar\) by being
  \emph{small}) has an \(x_1\)-dependent lower bound, while \(y_2\)
  has no \(x_2\)-dependency. For this reason the precise assignment of
  \(x_1,x_2\) in the proposition is ideal going forward.
\end{proof}


Suppose there exists a pair
\(\hat{\sigma}=[1_t+1,y_1'],\hat{\tau}=[m_t,y_2']\) with
\(W_\dar(\hat{\sigma})\geq W_\dar(\hat{\tau})\) such
that \((\hat{\sigma},\hat{\tau})\) is an anti-stable pair
for any weight \((2,b)\) with \(b\geq n-T\).

\begin{prop}\label{prop_anti2}
  If \((\hat{\sigma},\hat{\tau})\) is an anti-stable pair,
  then \(\hat{\sigma}=[1_t+1,y_1'],\tau=[m_t,y_2]\) also comprise an
  anti-stable pair, where
  \(y_2=n_t\) or \(y_2=n_t-k(y_1')\): whichever has opposite
  \(y\)-orientation from \(y_1'\). Furthermore,
  \(R(\hat{\sigma},\tau)\geq R(\hat{\sigma},\hat{\tau})\), and
  \(\hat{\sigma}\) has larger dimension than \(\tau\).
\end{prop}

\begin{proof}
  \textbf{(1)} Suppose \(y_1'\in[\textrm{max},\textrm{next min})\). Then
  \(\tau=[m_t,n_t-k(y_1')]\) and \(\hat{\tau}=[m_t,y_2]\), with
  \(y_2\geq n_t-k(y_1')\) and having orientation
  \(y_2\in[\textrm{min},\textrm{next max})\).

  If \(n_t-1-k(y_1')<y_2<n_t\), then \[
    D_I(\hat{\sigma},\tau)=D_I(\hat{\sigma},\hat{\tau})
  \] but \[
    D_{\mathrm{AR}}(\hat{\sigma},\tau)-D_{\mathrm{AR}}(\hat{\sigma},\hat{\tau})
    =y_2-(n_t-k(y_1'))\geq 0,
  \] and so \[
    R(\hat{\sigma},\tau)\geq R(\hat{\sigma},\hat{\tau}).
  \]

  Otherwise, \(y_2\in[m_p,n_p)\) for some \(p\geq t+1\). From \(\tau\)
  to \(\hat{\tau}\), the right endpoint \emph{increases}, and so the
  value of \(D_I\) may \emph{decrease}. Specifically, if
  \(D_I(\hat{\sigma},\tau)\) was determined by a particularly large
  \(2\)-V that is then included in the larger support of
  \(\hat{\tau}\), it will not be taken into account for that
  interleaving distance, and we will have a non-zero value for \[
    D_I(\hat{\sigma},\tau)-D_I(\hat{\sigma},\hat{\tau})=
    2\left(\max_{m_t<m_i\leq y_1'}\{T_i\}-\max_{y_2< m_i\leq y_1'}\{T_i\}\right).
  \] Let \(\displaystyle T_j=\max_{m_t<m_i\leq y_1'}\{T_i\}\). Then the difference
  above is at most \(2(T_j-1)\). If we can show that the difference
  between the \(D_{\mathrm{AR}}\)'s is larger than this, we will have shown a
  net increase in \(R(\hat{\sigma},\tau)\) over
  \(R(\hat{\sigma},\hat{\tau})\). \[
    D_{\mathrm{AR}}(\hat{\sigma},\tau)-D_{\mathrm{AR}}(\hat{\sigma},\hat{\tau})
    =y_2-(n_t-k(y_1'))\geq m_j-n_t+k(y_1'),
  \] as the drop in \(D_I\)'s was assumed to have happened by \(y_2\)
  exceeding the value of \(m_j\) (and so \(n_j\) by orientation
  conditions). As \(k(y_1')=T_j-1\), the difference
  in \(D_{\mathrm{AR}}\)'s becomes \[
    m_j-n_t+k(y_1')\geq T_j+T_j-1=2(T_j)-1.
  \] This is precisely what was desired, and so we have the inequality
  for \(R\)-values.

  \textbf{(2)} Suppose next that \(y_1'\in[\textrm{min},\textrm{next max})\). Let
  \(y_2>n_t\) of orientation \([\textrm{max},\textrm{next min})\).

  If \(n_t<y_2<m_{t+1}\), then \[
    D_I(\hat{\sigma},[m_t,n_t])=D_I(\hat{\sigma},[m_t,y_2])
  \] and \[
    D_{\mathrm{AR}}(\hat{\sigma},[m_t,n_t])>D_{\mathrm{AR}}(\hat{\sigma},[m_t,y_2]),
  \] so \(R\) strictly increases from choosing the left endpoint of
  \(\tau\) to be \(n_t\).

  Otherwise, by the requirement of orientation, \(n_{t+1}\leq
  y_2\). Then \[
    D_I(\hat{\sigma},[m_t,n_t])-D_I(\hat{\sigma},[m_t,y_2])=
    2\left(\max_{m_t<m_i\leq y_1'}\{T_i\}
    -\max_{y_2<m_i\leq y_1'}\{T_i\}\right).
  \] 
  The above difference is bounded above by \(2(T_j-1)\), where
  \(\displaystyle T_j:=\max_{y_2<m_i\leq y_1}\{T_i\}\).

  At the same time, \[
    D_{\mathrm{AR}}(\hat{\sigma},[m_t,n_t])-D_{\mathrm{AR}}(\hat{\sigma},[m_t,y_2])=y_2-n_t,
  \] where \(y_2\geq n_j\). But, \(y_2-m_j\geq 2T_j\), and so
  \(y_2-n_t>2T_j\).
  
  Combined, we see that
  \(R(\hat{\sigma},[m_t,n_t])>R(\hat{\sigma},[m_t,y_2])\) for any
   choice of \(y_2>n_t\).
\end{proof}

\subsection{Permissibility of \(n-T/2\)}

\begin{corollary}\label{cor_shallow_non_central} Let \(P=V_1\wedge
  V_2\wedge\ldots\wedge V_p\) be a shallow and \emph{non}-central
  orientation of \(\mathbb{A}_n\). The minimal weight such
  that \(D_\dar\leq D_\dwil\) is \((2,b)\) where \(b\) is bounded
  above by \[
    b\leq n-T/2-1.
  \]
\end{corollary}

\begin{proof}
  \emph{The minimal pair is always stable for} \(b\geq n-T/2-1\)
  : Of the two possible minimals pairs of Corollary \ref{cor_min_pair}
  we will only show the proof of \(\sigma_d=[1_t+1,y_d]\) and
  \(\tau_d=[m_t,n_t-k(y_d)]\). (The proof for \(\sigma_u,\tau_u\) is
  incredibly similar, and a slightly less restrictive inequality.)
  Recall that \(y_d\) is minimal such that
  \(x_1+D_I=1_t+n-T/2+2k(y_d)\leq y_d\)
  (Definition \ref{defn_k_stuff}).

  So \(D_I=b+2(k(y_d)-1)\) and 
  \(\delta_\dar(\sigma,\tau)=m_t-1_t-1+n-(n_t-k(y_d))+n-y_d\). Comparing,
  we get
  \begin{align*}
    \delta_\dar(\sigma_d,\tau_d)&\leq D_I(\sigma,\tau)\text{ if}\\
    m_t-1_t-1+n-n_t+k(y_d)+n-y_d &\leq b+2(k(y_d)-1)\text{ if} \\
    m_t-1_t-1+2n-n_t+k(y_d)-(1_t+b+2k(y_d))&\leq
    b+2(k(y_d)-1)\text{ if}\\
    m_t-1-2\cdot1_t+2n-n_t-k(y_d)-2b&\leq 2(k(y_d)-1)\text{
                                     if}\\
    m_t-1-2\cdot1_t+2n-n_t-2n+T&\leq 2k(y_d)-2\\
    m_t-n_t+T-2\cdot1_t+3&\leq 3k(y_d)\text{ if}\\
    -2\cdot1_t+3&\leq 3k(y_d)
  \end{align*}
  the last statement of which is true due to the left being \(\leq 1\)
  and the right being \(\geq 3\).

  

\end{proof}


\begin{figure}
  \includegraphics[scale = 4]{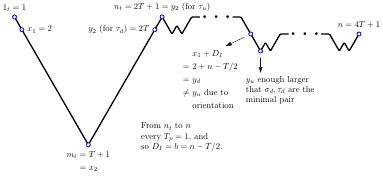}
  \caption{General example of a poset that attains \((2,n-T/2-1)\)
    as its minimal stable weight.}
  \label{fig_upper_bound}
\end{figure}

\begin{ex}\label{ex_upper_bound}(See Figure \ref{fig_upper_bound}.)
  We show a sample poset in which the minimal stable value equals the
  upper bound \(b=n-T/2-1\).

  For \(T>1\), let \(1_t=1\), \(m_t=T+1\), \(n_t=2T+1\),
  \(n=1+4T\). Let the region from \(n_t\) to \(n\) consist of
  \(V_p\)'s with \(T_p=1\) and of orientation such that \(y_u\) is
  forced to be (even just slightly) larger than the minimization given
  by \ref{defn_k_stuff}.

  Then \((\sigma_d,\tau_d)\) form the minimal pair, and we can
  explicity check that \(b=n-T/2-1\) is permissible while no smaller
  weight will be:
  \begin{align*}
    \delta(\sigma_d,\tau_d)&\leq D_I(\sigma_d,\tau_d)\text{ iff}\\
    T-1+n-2T+n-(2+b)&\leq b\text{ iff}\\
    2n-T-3&\leq 2b\text{ iff}\\
    n-T/2-1&\leq b\text{ if }T\text{ is even, or}\\
    n-T/2-2&\leq b\text{ if }T\text{ is odd.}
  \end{align*}

\end{ex}

\bibliography{master_bib}
\bibliographystyle{alpha}

\bigskip
\footnotesize

Killian~Meehan, \textsc{Kyoto University Institute of Advanced Study}\par\nopagebreak
\textit{E-mail address}: \texttt{killian.f.meehan@gmail.com}

\medskip

David~C.~Meyer, \textsc{Smith College}\par\nopagebreak
\textit{E-mail address}: \texttt{dmeyer@smith.edu}

\end{document}